\documentclass[a4paper,12pt]{article}
\usepackage[cp1251]{inputenc}
\usepackage[russian]{babel}
\usepackage{amsfonts, amssymb, amsmath, amsthm, amscd}
\usepackage{cite}
\author{A.A. Vasil'eva\footnote{E-mail address: vasilyeva\_nastya@inbox.ru}}
\title{An embedding theorem for weighted Sobolev classes on a John domain:
case of weights that are functions of a distance to a certain
$h$-set}
\date{}
\begin{document}

\maketitle

\newenvironment{Biblio}{%
                  \renewcommand{\refname}{\footnotesize REFERENCES}%
                  \begin{thebibliography}{99}%
                  \itemsep -0.2em}%
                  {\end{thebibliography}}

\renewcommand{\le}{\leqslant}
\renewcommand{\ge}{\geqslant}
\newcommand{\sgn}{\mathrm {sgn}\,}
\newcommand{\inter}{\mathrm {int}\,}
\newcommand{\dist}{\mathrm {dist}}
\newcommand{\supp}{\mathrm {supp}\,}
\newcommand{\R}{\mathbb{R}}
\renewcommand{\C}{\mathbb{C}}
\newcommand{\Z}{\mathbb{Z}}
\newcommand{\N}{\mathbb{N}}
\newcommand{\Q}{\mathbb{Q}}
\theoremstyle{plain}
\newtheorem{Trm}{Theorem}
\newtheorem{trma}{Theorem}
\newtheorem{Def}{Definition}
\newtheorem{Cor}{Corollary}
\newtheorem{Lem}{Lemma}
\newtheorem{Rem}{Remark}
\newtheorem{Sta}{Proposition}
\renewcommand{\proofname}{\bf Proof}
\renewcommand{\thetrma}{\Alph{trma}}

\section{Introduction}

Let $\Omega \subset \R^d$ be a bounded domain (an open connected
set), and let $g$, $v:\Omega\rightarrow \R_+$ be measurable
functions. For each measurable vector-valued function $\varphi:\
\Omega\rightarrow \R^m$, $\varphi=(\varphi_k) _{1\le k\le m}$, and
for each $p\in [1, \, \infty]$ we put
$$
\|\varphi\|_{L_p(\Omega)}= \Big\|\max _{1\le k\le m}|\varphi _k |
\Big\|_p.
$$
Let $\overline{\beta}=(\beta _1, \, \dots, \, \beta _d)\in
\Z_+^d:=(\N\cup\{0\})^d$, $|\overline{\beta}| =\beta _1+
\ldots+\beta _d$. For any distribution $f$ defined on $\Omega$ we
write $\displaystyle \nabla ^r\!f=\left(\partial^{r}\! f/\partial
x^{\overline{\beta}}\right)_{|\overline{\beta}| =r}$ (here partial
derivatives are taken in the sense of distributions), and denote
by $m_r$ the number of components of the vector-valued
distribution $\nabla ^r\!f$. Set
$$
W^r_{p,g}(\Omega)=\left\{f:\ \Omega\rightarrow \R\big| \; \exists
\varphi :\ \Omega\rightarrow \R^{m_r}\!:\ \| \varphi \|
_{L_p(\Omega)}\le 1, \, \nabla ^r\! f=g\cdot \varphi\right\}
$$
\Big(we denote the corresponding function $\varphi$ by
$\displaystyle\frac{\nabla ^r\!f}{g}$\Big),
$$
\| f\|_{L_{q,v}(\Omega)}{=}\| f\|_{q,v}{=}\|
fv\|_{L_q(\Omega)},\qquad L_{q,v}(\Omega)=\left\{f:\Omega
\rightarrow \R| \; \ \| f\| _{q,v}<\infty\right\}.
$$
We call the set $W^r_{p,g}(\Omega)$ a weighted Sobolev class.

For properties of weighted Sobolev spaces and their
generalizations, see the books \cite{triebel, kufner, turesson,
edm_trieb_book, triebel1, edm_ev_book} and the survey paper
\cite{kudr_nik}. Sufficient conditions for boundedness and
compactness of embeddings of weighted Sobolev spaces into weighted
$L_q$-spaces were obtained by Kudryavtsev \cite{kudrjavcev},
Ne\v{c}as \cite{j_necas}, Kufner \cite{kuf_cz, kufner_69, kufner},
Yakovlev \cite{g_yakovlev}, Triebel \cite{triebel}, Lizorkin and
Otelbaev \cite{liz_otel}, Gurka and Opic \cite{gur_opic1,
gur_opic2, gur_opic3}, Besov \cite{besov1, besov2, besov3,
besov4}, Antoci \cite{antoci}, Gol’dshtein and Ukhlov
\cite{gold_ukhl}, and other authors. Notice that in these papers
weighted Sobolev classes were defined as $W^r_{p,g}(\Omega)\cap
L_{p,w}(\Omega)$ for some weight $w$, or as $\cap _{l=0}^r
W^l_{p,g_i}(\Omega)$ for some different weight functions $g_i$.

For a Lipschitz domain $\Omega$, a $k$-dimensional manifold
$\Gamma\subset \Omega$, and for weights depending only on distance
from $x$ to $\Gamma$, the following results were obtained. The
case $r=1$, $p=q$ was considered in papers of Ne\v{c}as
\cite{j_necas} (the case of power weights and $\Gamma=\partial
\Omega$), Kufner \cite{kuf_cz} (weights are powers of distance
from a fixed point), Yakovlev \cite{g_yakovlev} (weights depend on
distance to $k$-dimensional manifold), Kadlec and Kufner
\cite{kadl1, kadl2} (here weights are powers with a logarithmic
factor, $\Gamma=\partial \Omega$), Kufner \cite{kufner_69} (here
weights are arbitrary functions of distance from $\partial
\Omega$). For $p=q$, $r\in \N$, $\Gamma=\partial \Omega$ and for
power type weights, the embedding theorem was obtained by El Kolli
\cite{el_kolli}. By using Banach space interpolation, Triebel
\cite{trieb_mat_sb} extended this result to the case $p\le q$. For
$p=q$, $r=1$, a $k$-dimensional manifold $\Gamma$ and general
weights Kufner and Opic \cite{kuf_op} obtained some sufficient
conditions for compactness of embeddings. For $p>q$, $r\in \N$,
for an arbitrary $k$-dimensional manifold $\Gamma$ and power type
weights the criterion of the embedding was obtained in\cite{jain1,
jain2, jain3}. In addition, in \cite{jain3} for $r=1$ the
criterion was obtained for arbitrary functions depending on
distance from the manifold $\Gamma$.

Notice that for $p\ge q$ in the proof of embedding theorems
two-weighted Hardy-type inequalities were applied.

In \cite{gur_opic1} sufficient conditions for the embedding were
obtained for $r=1$ and general weights. The norm in the weighted
Sobolev space was defined by $\|f\|_{g,w}=\left\|\frac{\nabla
f}{g}\right\|_{L_p(\Omega)}+\|wf\|_{L_p(\Omega)}$. The idea of the
proof was the following. First the Besikovic covering of $\Omega$
was constructed, then for each ball of this covering the Sobolev
embedding theorem was applied. After that the obtained estimates
were summarized. Here it was essential to use the second weight
$w$, which satisfied rather tight restricti\-ons. If the boundary
$\partial \Omega$ is Lipshitz and weight functions are powers of
distance from $\partial \Omega$, then it is possible to take more
weak restrictions on $w$. To this end, the other method of proof
is used (employing the Hardy inequality). In \cite{gur_opic2}
embedding theorems were obtained for a H\"{o}lder domain $\Omega$
and power type weights depending on distance from $\partial
\Omega$.

It is also worth noting the paper \cite{j_lehrback}, where the
result on embedding of $\mathaccent'27 W^1_{p,g}(\Omega)$ into
$L_{p,v}(\Omega)$ was obtained for $r=1$, $p=q$ and weights that
are powers of the distance from the irregular boundary of
$\partial \Omega$.

In the present paper, we consider a John domain $\Omega$, an
$h$-set $\Gamma\subset \partial \Omega$ and weight functions
depending on distance from $\Gamma$ (their form will be written
below).

Let $X$, $Y$ be sets, $f_1$, $f_2:\ X\times Y\rightarrow \R_+$. We
write $f_1(x, \, y)\underset{y}{\lesssim} f_2(x, \, y)$ (or
$f_2(x, \, y)\underset{y}{\gtrsim} f_1(x, \, y)$) if, for any
$y\in Y$, there exists $c(y)>0$ such that $f_1(x, \, y)\le
c(y)f_2(x, \, y)$ for each $x\in X$; $f_1(x, \,
y)\underset{y}{\asymp} f_2(x, \, y)$ if $f_1(x, \, y)
\underset{y}{\lesssim} f_2(x, \, y)$ and $f_2(x, \,
y)\underset{y}{\lesssim} f_1(x, \, y)$.

For $x\in \R^d$ and $a>0$ we shall denote by  $B_a(x)$ the closed
Euclidean ball of radius $a$ in $\R^d$ centered at the point $x$.

Let $|\cdot|$ be an arbitrary norm on $\R^d$, and let $E, \,
E'\subset \R^d$, $x\in \R^d$. We set
$$
{\rm diam}_{|\cdot|}\, E=\sup \{|y-z|:\; y, \, z\in E\}, \;\; {\rm
dist}_{|\cdot|}\, (x, \, E)=\inf \{|x-y|:\; y\in E\},
$$
$$
{\rm dist}_{|\cdot|}\, (E', \, E)=\inf \{|x-y|:\; x\in E, \, y\in
E'\}.
$$

\begin{Def}
\label{fca} Let $\Omega\subset\R^d$ be a bounded domain, and let
$a>0$. We say that $\Omega \in {\bf FC}(a)$ if there exists a
point $x_*\in \Omega$ such that, for any $x\in \Omega$, there
exists a curve $\gamma _x:[0, \, T(x)] \rightarrow\Omega$ with the
following properties:
\begin{enumerate}
\item $\gamma _x\in AC[0, \, T(x)]$, $|\dot \gamma _x|=1$ a.e.,
\item $\gamma _x(0)=x$, $\gamma _x(T(x))=x_*$,
\item $B_{at}(\gamma _x(t))\subset \Omega$ for any $t\in [0, \, T(x)]$.
\end{enumerate}
\end{Def}

\begin{Def}
We say that $\Omega$ satisfies the John condition (and call
$\Omega$ a John domain) if $\Omega\in {\bf FC}(a)$ for some $a>0$.
\end{Def}
For a bounded domain, the John condition is equivalent to the
flexible cone condition (see the definition in \cite{besov_il1}).
Reshetnyak in the papers \cite{resh1, resh2} constructed the
integral representation for functions defined on a John domain
$\Omega$ in terms of their derivatives of order $r$. This integral
representation yields that in the case $\frac rd-\left(\frac
1p-\frac 1q\right)_+>0$ the class $W^r_p(\Omega)$ is compactly
embedded in the space $L_q(\Omega)$ (i.e., the conditions of the
compact embedding are the same as for $\Omega=[0, \, 1]^d$).

\begin{Rem}
\label{rem2} If $\Omega\in {\bf FC}(a)$ and a point $x_*$ is such
as in Definition \ref{fca}, then
\begin{align}
\label{diam_dist} {\rm diam}_{|\cdot|}\, \Omega
\underset{d,a,|\cdot|}{\lesssim} {\rm dist}_{|\cdot|}\, (x_*, \,
\partial \Omega).
\end{align}
\end{Rem}

Denote by $\mathbb{H}$ the set of all non-decreasing positive
functions defined on $(0, \, 1]$.

We introduce the concept of $h$-set according to
\cite{m_bricchi1}.
\begin{Def}
\label{h_set} Let $\Gamma\subset \R^d$ be a compact set, and let
$h\in \mathbb{H}$. We say that $\Gamma$ is an $h$-set if there
exists a finite measure $\mu$ on $\R^d$ such that $\supp
\mu=\Gamma$ and $\mu(B_t(x))\asymp h(t)$ for each $x\in \Gamma$,
$t\in (0, \, 1]$.
\end{Def}
Notice that the measure $\mu$ is non-negative.

The concept of $h$-sets for functions $h$ of a special type
appeared earlier (see papers of Edmunds, Triebel and Moura
\cite{tr_fract, et1, et2, s_moura}). In these and some other
papers (see, for example, \cite{m_bricchi, triebel_fractal,
i_piotr, i_piotr1, caet_lop}) properties of the operator ${\rm
tr}|_{\Gamma}$ in Besov and Triebel -- Lizorkin spaces and its
composition with the operator $(\Delta)^{-1}$ were studied. Here
${\rm tr}|_{\Gamma}$ is the operator of restriction on the $h$-set
$\Gamma$. In \cite{har_piot} Besov spaces with Muckenhoupt weights
were studied; weight functions depending on the distance from a
certain $h$-set were considered as examples.

In the sequel we suppose that
\begin{align}
\label{def_h} h(t)=t^{\theta}\Lambda(t), \;\;\; 0\le \theta<d,
\end{align}
where $\Lambda:(0, \, +\infty)\rightarrow (0, \, +\infty)$ is an
absolutely continuous function such that
\begin{align}
\label{yty} \frac{t\Lambda'(t)}{\Lambda(t)} \underset{t\to+0}{\to}
0.
\end{align}

Let $\Omega\in {\bf FC}(a)$ be a bounded domain, and let
$\Gamma\subset \partial \Omega$ be an $h$-set. In the sequel for
convenience we suppose that $\Omega\subset \left[-\frac 12, \,
\frac 12\right]^d$ (the general case can be reduced to this case).
Let $1<p\le \infty$, $1\le q<\infty$, $r\in \N$, $\delta:=r+\frac
dq-dp>0$, $g(x)=\varphi_g({\rm dist}_{|\cdot|}(x, \, \Gamma))$,
$v(x)=\varphi_v({\rm dist}_{|\cdot|}(x, \, \Gamma))$,
\begin{align}
\label{ghi_g0} \varphi_g(t)=t^{-\beta_g}\Psi_g(t), \;\;
\varphi_v(t)=t^{-\beta_v} \Psi_v(t),
\end{align}
with absolutely continuous functions $\Psi_g$, $\Psi_v$ such that
\begin{align}
\label{psi_cond} \frac{t\Psi'_g(t)}{\Psi_g(t)}
\underset{t\to+0}{\to}0, \;\; \frac{t\Psi'_v(t)}{\Psi_v(t)}
\underset{t\to+0}{\to}0;
\end{align}
in addition, we suppose that
\begin{align}
\label{muck} -\beta_vq+d-\theta>0.
\end{align}
Also we assume that
\begin{align}
\label{beta} \text{a) }\beta_g+\beta_v<\delta-\theta\left(\frac
1q-\frac 1p\right)_+ \quad\text{ or b) }\beta_g+\beta_v=\delta-
\theta\left(\frac 1q-\frac 1p\right)_+.
\end{align}
In the case b) we suppose that
\begin{align}
\label{phi_g} \Lambda(t)=|\log t|^\gamma\tau(|\log t|), \;\;
\Psi_g(t)=|\log t|^{-\alpha_g} \rho_g(|\log t|),\;\;
\Psi_v(t)=|\log t|^{-\alpha_v} \rho_v(|\log t|),
\end{align}
functions $\rho_g$, $\rho_v$, $\tau$ are absolutely continuous,
\begin{align}
\label{ll} \lim \limits _{y\to +\infty}\frac{y\tau'(y)}{\tau(y)}=
\lim \limits _{y\to +\infty}\frac{y\rho_g'(y)}{\rho_g(y)}=\lim
\limits _{y\to +\infty}\frac{y\rho_v'(y)}{\rho_v(y)}=0,
\end{align}
\begin{align}
\label{g0ag} \gamma<0\text{ and }\alpha:=\alpha_g+\alpha_v
>(1-\gamma)\left(\frac 1q-\frac 1p\right)_+.
\end{align}
It is easy to show that the functions $\Lambda$, $\Psi_g$ and
$\Psi_v$ satisfy (\ref{yty}) and (\ref{psi_cond}).
\begin{Rem}
\label{ste} If functions $\Psi_g$ and $\Psi_v$ (respectively
$\rho_g$ and $\rho_v$) satisfy (\ref{psi_cond}) (respectively
(\ref{ll})), then their product and each degree of these functions
satisfies the similar condition.
\end{Rem}

Denote $$\beta=\beta_g+\beta_v, \quad \rho(y)=\rho_g(y)\rho_v(y),
\quad \Psi(y)=\Psi_g(y)\Psi_v(y),$$ $\mathfrak{Z}=(r,\, d, \, p,
\, q, \, \beta_g, \, \beta_v, \, \theta, \, \Lambda, \, \Psi_g, \,
\Psi_v, \, a)$.

Let ${\cal P}_{r-1}(\R^d)$ be the space of polynomials on $\R^d$
of degree not exceeding $r-1$. For a measurable set $E\subset
\R^d$, we put ${\cal P}_{r-1}(E)= \{f|_E:\, f\in {\cal
P}_{r-1}(\R^d)\}$.
\begin{Trm}
\label{main} For any function $f\in \operatorname{span}\,
W^r_{p,g}(\Omega)$ there exists a polynomial $Pf\in {\cal
P}_{r-1}(\Omega)$ such that
$$
\|f-Pf\|_{L_{q,v}(\Omega)}\underset{\mathfrak{Z}}{\lesssim}
\left\|\frac{\nabla^rf}{g}\right\|_{L_p(\Omega)}.
$$
Here the mapping $f\mapsto Pf$ can be extended to a linear
continuous operator $P:L_{q,v}(\Omega)\rightarrow {\cal
P}_{r-1}(\Omega)$.
\end{Trm}

Later we shall give a more general formulation of this theorem. It
can be used in problems on estimating of approximation of the
class $W^r_{p,g}(\Omega)$ by piecewise polynomial functions in the
space $L_{q,v}(\Omega)$ and in problems on estimating of
$n$-widths.

We may assume that the norm on $\R^d$ is given by
$$
|(x_1, \, \dots, \, x_d)|=\max _{1\le i\le d} |x_i|.
$$

The paper is organized as follows. In Sections 2 and 3, we give
necessary notations and formulate the results which will be
required in the sequel.  In Section 4, we describe the domain
$\Omega$ in terms of a tree ${\cal T}$ (see \cite{vas_john}) and
construct a special partition of this tree. In Section 5, the
discrete weighted Hardy-type inequality on a combinatorial tree is
obtained for $p=q$. If the tree is regular, i.e., the number of
vertices that follow the given vertex depends only on the distance
between this vertex and the root of the tree, then we employ some
convexity arguments and reduce the problem to the proof of a
Hardy-type inequality for sequences. The tree which was
constructed in Section 4 is not regular in general; however, it
satisfies some more weak condition of regularity. For such trees
it is possible to reduce the problem to the case of regular trees.
To this end, a discrete analogue for theorem of Evans -- Harris --
Pick \cite{ev_har_pick} is proved. At this step, some quantity
$B_{{\cal D}}$ emerges; it is defined for subtrees ${\cal D}$ and
can be calculated recursively. Under some conditions on weights,
we prove that $B_{{\cal D}}$ can be estimated by some more simple
quantity $S_{{\cal D}}$. Then for any subtree ${\cal D}$ we
construct a subtree $\hat{\cal D}$ in some regular tree $\hat{\cal
A}$, such that $S_{{\cal D}}$ can be estimated from above by
$S_{\hat{\cal D}}$. In Section 6, the discrete Hardy-type
inequality on a tree is proved for $p\ne q$. To this end, the
problem is reduced to consider the cases $p=q$ and $p=\infty$;
here the H\"{o}lder inequality is applied. In Section 7, the
embedding theorem is proved. The problem is reduced to considering
the case $r=d$ and employing the discrete Hardy-type inequality on
a tree.

Embedding theorems and related results for function classes on
metric and combinatorial trees were studied by different authors.
Naimark and Solomyak \cite{naim_sol} obtained Hardy-type
inequalities on regular metric trees. For a weighted summation
operator (i.e., a Hardy-type operator) on a combinatorial tree
acting from $l_2$ into $l_\infty$, Lifshits and Linde
\cite{lifs_m, l_l, l_l1} obtained estimates of entropy numbers. In
\cite{e_h_l, ev_har_lang, solomyak} Evans, Harris, Lang and
Solomyak obtained estimates of approximation numbers for weighted
Hardy-type operators on metric trees. Also it is worth noting
results of Evans and Harris \cite{evans_har} on embeddings of
Sobolev classes on ridged domains into Lebesgue spaces; here the
definition of a ridged domain was given in terms of metric trees.

\section{Notation}
In what follows $\overline{A}$ (${\rm int}\, A$, ${\rm mes}\,A$,
${\rm card}\,A$, respectively) be, respectively, the closure
(interior, Lebesgue measure, cardinality) of $A$. If a set $A$ is
contained in some subspace $L\subset \R^d$ of dimension  $(d-1)$,
then we denote by ${\rm int}_{d-1}A$ the interior of $A$ with
respect to the induced topology on the space $L$. We say that sets
$A$, $B\subset \R ^d$ do not overlap if $A\cap B$ is a Lebesgue
nullset. For a convex set $A$ we denote by $\dim A$  the dimension
of the affine span of the set $A$.

A set $A\subset \R^d$ is said to be a parallelepiped if there are
$s_j\le t_j$, $1\le j\le d$, such that
$$
\prod _{j=1}^d (s_j, \, t_j)\subset A\subset\prod _{j=1}^d [s_j,
\, t_j].
$$
If $t_j-s_j=t_1-s_1$ for any $j=1, \, \dots , \, d$, then a
parallelepiped is referred to as a cube.

Let ${\cal K}$ be a family of closed cubes in $\R^d$ with axes
parallel to coordinate axes. For a cube $K \in {\cal K}$ and $s\in
\Z_+$ we denote by $\Xi _s(K)$ the set of $2^{sd}$ closed
non-overlapping cubes of the same size that form a partition of
$K$, and write $\Xi(K):=\bigcup_{s\in \Z_+} \Xi _s(K)$. We
generally consider that these cubes are close (except the proof of
Lemma \ref{nu_st}).

We recall some definitions from graph theory. Throughout, we
assume that the graphs have neither multiple edges nor loops.

Let $\Gamma$ be a graph containing at most countable number of
vertices. We shall denote by ${\bf V}(\Gamma)$ and by ${\bf
E}(\Gamma)$ the set of vertices and the set of edges of $\Gamma$,
respectively. Two vertices are called {\it adjacent} if there is
an edge between them. We shall identify pairs of adjacent vertices
with edges that connect them. Let $\omega_i\in {\bf V}(\Gamma)$,
$1\le i\le n$. The sequence $(\omega_1, \, \dots, \, \omega_n)$ is
called a {\it path}, if the vertices $\omega_i$ and $\omega_{i+1}$
are adjacent for any $i=1, \, \dots , \, n-1$. We say that a graph
is {\it connected} if any two vertices are connected by a finite
path. A connected graph is a {\it tree} if it has no cycles.

Let $({\cal T}, \, \omega_0)$ be a tree with a distinguished
vertex (or a root) $\omega_0$. We introduce a partial order on
${\bf V}({\cal T})$ as follows: we say that $\omega'>\omega$ if
there exists a path $(\omega_0, \, \omega_1, \, \dots , \,
\omega_n, \, \omega')$ such that $\omega=\omega_k$ for some $k\in
\overline{0, \, n}$. In this case, we set $\rho_{{\cal T}}(\omega,
\, \omega') =n+1-k$ and call this quantity the distance between
 $\omega$ and $\omega'$. In addition, we set
$\rho_{{\cal T}}(\omega, \, \omega)=0$. If $\omega'>\omega$ or
$\omega'=\omega$, then we write $\omega'\ge \omega$ and denote
$[\omega, \, \omega']:= \{\omega''\in {\bf V}({\cal T}):\omega \le
\omega''\le \omega'\}$ (this set of vertices is called a segment).
This partial order on ${\cal T}$ induces a partial order on its
subtree.

For any $j\in \Z_+$ we set
$$
\label{v1v}{\bf V}_j(\omega):={\bf V}_j ^{{\cal T}}(\omega):=
\{\omega'>\omega:\; \rho_{{\cal T}}(\omega, \, \omega')=j\}.
$$
Given $\omega\in {\bf V}({\cal T})$, we denote by ${\cal
T}_\omega=({\cal T}_\omega, \, \omega)$ a subtree of ${\cal T}$
with the set of vertices
\begin{align}
\label{vpvtvpv} \{\omega'\in {\bf V}({\cal T}):\omega'\ge
\omega\}.
\end{align}

Let ${\cal G}$ be a subgraph in ${\cal T}$. Denote by ${\bf
V}_{\max} ({\cal G})$ and by ${\bf V}_{\min}({\cal G})$,
respectively, the sets of maximal and minimal vertices in ${\cal
G}$.

Let ${\bf W}\subset {\bf V}({\cal T})$. We say that ${\cal
G}\subset {\cal T}$ is a {\it maximal subgraph} on the set of
vertices ${\bf W}$ if ${\bf V}({\cal G})={\bf W}$ and if any two
vertices $\xi'$, $\xi''\in {\bf W}$ that are adjacent in ${\cal
T}$ are also adjacent in ${\cal G}$.

We need the concept of a metric tree. Let $(\cal{T}, \, \omega_*)$
be a tree with a finite set of vertices, and let $\Delta:{\bf
E}({\cal{T}})\rightarrow 2^{\R}$ be a mapping such that for any
$\lambda\in {\bf E}({\cal{T}})$ the set
$\Delta(\lambda)=[a_\lambda, \, b_\lambda]$ is a non-trivial
segment. Then the pair $\mathbb{T}=({\cal T}, \, \Delta)$ is
called a metric tree. A point on the edge $\lambda$ of the metric
tree $\mathbb{T}$ is a pair $(t, \, \lambda)$, $t\in [a_\lambda,
\, b_\lambda]$, $\lambda\in {\bf E}({\cal{T}})$ (if $\omega'\in
{\bf V}_1(\omega)$, $\omega''\in {\bf V}_1(\omega')$,
$\lambda=(\omega, \, \omega')$, $\lambda'=(\omega', \, \omega'')$,
then we set $(b_\lambda, \, \lambda)=(a_{\lambda'}, \,
\lambda')$). The distance between two points of $\mathbb{T}$ is
defined as follows: if $(\omega_0, \, \omega_1, \, \dots, \,
\omega_n)$ is a path in the tree $\mathbb{T}$, $n\ge 2$,
$\lambda_i=(\omega_{i-1}, \, \omega_i)$, $x=(t_1, \, \lambda_1)$,
$y=(t_n, \, \lambda_n)$, the we set
$$
|y-x|_{\mathbb{T}}=|b_{\lambda_1}-t_1|+\sum \limits _{i=2}^{n-1}
|b_{\lambda_i}-a_{\lambda_i}| +|t_n-a_{\lambda_n}|;
$$
if $x=(t', \, \lambda)$, $y=(t'', \, \lambda)$, then
$|y-x|_{\mathbb{T}}=|t'-t''|$.

We say that $(t', \, \lambda')\le(t'', \, \lambda'')$ if
$\lambda'\le\lambda''$ and $t'\le t''$ in the case
$\lambda'=\lambda''$. If $(t', \, \lambda')\le (t'', \,
\lambda'')$ and $(t', \, \lambda')\ne(t'', \, \lambda'')$, then we
write $(t', \, \lambda')<(t'', \, \lambda'')$. If $a$, $x\in
\mathbb{T}$, $a\le x$, then we set $[a, \, x]=\{y\in \mathbb{T}:\,
a\le y\le x\}$.

A subset $\mathbb{A}=\{(t, \, \lambda):\; \lambda\in {\bf
E}(\mathbb{T}), \; t\in A_\lambda\}$ is said to be measurable, if
$A_\lambda$ is measurable for any $\lambda\in {\bf
E}(\mathbb{T})$. The Lebesgue measure of $\mathbb{A}$ is defined
by
$$
|\mathbb{A}|=\sum \limits _{\lambda\in {\bf E}(\mathbb{T})}
|A_\lambda|.
$$
A function $f:\mathbb{A}\rightarrow \R$ is said to be integrable
if $f_\lambda:=f|_{\{(t, \, \lambda):\, t\in A_\lambda\}}$ is
integrable for any $\lambda\in {\bf E}(\mathbb{T})$ and the sum
$\sum \limits _{\lambda \in {\bf E}(\mathbb{T})} \int \limits
_{A_\lambda} |f_\lambda(t)|\, dt$ is finite. In this case, we set
$$
\int \limits _{\mathbb{A}} f(x)\, dx=\sum \limits _{\lambda \in
{\bf E}(\mathbb{T})} \int \limits _{A_\lambda} f_\lambda(t)\, dt.
$$

Let $\mathbb{D}\subset \mathbb{T}$ be a connected subset. Denote
by ${\cal T}_{\mathbb{D}}$ the maximal subtree in $\cal{T}$ such
that for any $\lambda \in {\bf E}({\cal T}_{\mathbb{D}})$ the set
$\Delta(\lambda)\cap \mathbb{D}$ is a non-trivial segment. Set
$\Delta_{\mathbb{D}}(\lambda) =\Delta(\lambda) \cap \mathbb{D}$,
$\lambda \in {\bf E}({\cal T})$. Then $({\cal T}_{\mathbb{D}}, \,
\Delta_{\mathbb{D}})$ is a metric tree, which will be identified
with the set $\mathbb{D}$ and which will be called a metric
subtree in $\mathbb{T}$.

Let $\mathbb{D}$ be a metric subtree in $\mathbb{T}$. A point
$t\in \mathbb{D}$ is said to be maximal if $x\in
\mathbb{T}\backslash \mathbb{D}$ for any $x>t$.

\section{Preliminary results}
Let $\Delta$ be a cube with a side of length $2^{-m}$, $m\in \Z$.
Set ${\bf m}(\Delta)=m$. In particular, if $\Delta \in
\Xi\left(\left[-\frac 12, \, \frac 12\right]^d\right)$, then
$\Delta \in \Xi _{{\bf m}(\Delta)}\left(\left[-\frac 12, \, \frac
12\right]^d\right)$.

We shall need Whitney's covering theorem (see, e.g., \cite[p.
562]{leoni1}).
\begin{trma}
\label{whitney} Let $\Omega\subset \left[-\frac 12, \, \frac
12\right]^d$ be an open set. Then there exists a family of closed
pairwise non-overlapping cubes $\Theta(\Omega)=
\{\Delta_j\}_{j\in\N}\subset \Xi\left(\left[-\frac 12, \, \frac
12\right]^d\right)$ with the following properties:
\begin{enumerate}
\item $\Omega=\cup _{j\in \N}\Delta_j$;
\item ${\rm dist}\, (\Delta_j, \, \partial \Omega)\underset{d}{\asymp} 2^{-{\bf m}(\Delta_j)}$;
\item for any $j\in \N$
\begin{align}
\label{card_sopr} {\rm card}\, \{i\in \N:\dim (\Delta_i\cap
\Delta_j)=d-1\}\le 12^d;
\end{align}
\item if $\dim (\Delta_i\cap \Delta_j)=d-1$, then
\begin{align}
\label{bfm} {\bf m}(\Delta_j)-2\le {\bf m}(\Delta_i)\le {\bf
m}(\Delta_j)+2.
\end{align}
\end{enumerate}
\end{trma}

Andersen and Heinig in \cite{and_hein, hein1} proved discrete
analogues of the two-weighted Hardy-type inequality. We formulate
a particular case of their result, which will be used in the
sequel.
\begin{trma}
\label{hardy_diskr} Let $1\le p\le q<\infty$, and let
$\{u_n\}_{n\in \Z}$, $\{w_n\}_{n\in \Z}$ be nonnegative sequences
such that
$$
C:=\sup _{m\in \Z}\left(\sum \limits _{n=m}^\infty
w_n^q\right)^{1/q}\left( \sum \limits _{n=-\infty}^m
u_n^{p'}\right)^{1/p'}<\infty.
$$
Then, for any sequence $\{a_n\}_{n\in\Z}$,
\begin{align}
\label{a_h1} \left(\sum \limits _{n=-\infty}^\infty\left|w_n\sum
\limits _{k=-\infty}^n
u_ka_k\right|^q\right)^{1/q}\underset{p,q}{\lesssim} C\left( \sum
\limits _{n\in \Z}|a_n|^p\right)^{1/p}.
\end{align}
\end{trma}

Evans, Harris and Pick in \cite{ev_har_pick} proved a criterion
for boundedness of a two-weighted Hardy-type operator on a metric
tree.

Let $\mathbb{T}=({\cal T}, \, \Delta)$ be a metric tree, $x_0\in
\mathbb{T}$, and let $u$, $w:\mathbb{T}\rightarrow \R_+$ be
measurable functions. We set $\mathbb{T}_{x_0}=\{x\in
\mathbb{T}:\; x\ge x_0\}$,
$$
I_{u,w,x_0}f(x)=w(x)\int \limits_{x_0}^x u(t)f(t)\, dt.
$$

Denote by ${\cal J}_{x_0}={\cal J}_{x_0}(\mathbb{T})$ a family of
metric subtrees $\mathbb{D}\subset \mathbb{T}$ with the following
properties:
\begin{enumerate}
\item $x_0$ is a minimal vertex in $\mathbb{D}$;
\item if $x\in \partial \mathbb{D}\backslash \{x_0\}$, then $x$ is
a maximal point in $\mathbb{D}$.
\end{enumerate}

For $\mathbb{D}\in {\cal J}_{x_0}$, we set
$$
\alpha_{\mathbb{D}}=\inf \left\{\|f\|_{L_p(\mathbb{T})}:\, \int
\limits _{x_0}^t |f(x)|u(x)\, dx=1\text{ for any }t\in
\partial \mathbb{D}\right\}.
$$
\begin{trma}
\label{cr_har} Let $1\le p\le q\le \infty$. Then the operator
$I_{u,w,x_0}:L_p(\mathbb{T}_{x_0})\rightarrow
L_q(\mathbb{T}_{x_0})$ is bounded if and only if
$$
C_{u,w}:=\sup _{\mathbb{D}\in {\cal
J}_{x_0}}\frac{\|w\chi_{\mathbb{T}_{x_0}\backslash \mathbb{D}}\|
_{L_q(\mathbb{T})}}{\alpha _{\mathbb{D}}}<\infty.
$$
Moreover, $C_{u,w}\le \|I_{u,w}\| _{L_p(\mathbb{T}_{x_0})
\rightarrow L_q(\mathbb{T}_{x_0})}\le 4C_{u,w}$.
\end{trma}

The quantity $\alpha_{\mathbb{D}}$ is calculated recursively. The
following theorem is also proved in \cite{ev_har_pick}.
\begin{trma}
\label{rekurs} Let $\mathbb{D}\in {\cal J}_{x_0}$,
$\mathbb{D}=\cup _{j=0}^m \mathbb{D}_j$, $\mathbb{D}_0=[x_0, \,
y_0]$, $x_0<y_0$, $\mathbb{D}_j\in {\cal J}_{y_0}$, $1\le j\le m$,
$\mathbb{D}_i\cap \mathbb{D}_j=\{y_0\}$, $i\ne j$. Then
$$
\frac{1}{\alpha_{\mathbb{D}}}=\left\|\left(\alpha_{\mathbb{D}_0}^{-1},
\, \left\|(\alpha_{\mathbb{D}_i})_{i=1}^m\right\|^{-1}
_{l_p^m}\right)\right\|_{l_{p'}^2}.
$$
\end{trma}
Notice that if $x_0=(t', \, \lambda)$, $\lambda\in {\bf
E}(\mathbb{T})$, $t'\in \Delta(\lambda)$, and $y_0$ is such as in
Theorem \ref{rekurs}, then $y_0$ is a right end of
$\Delta(\lambda)$.

The following theorem is proved in \cite{sobol38, adams, adams1};
see also \cite[p. 51]{mazya1} and \cite[p. 566]{leoni1}.
\begin{trma}
\label{adams_etc} Let $1<p<q<\infty$, $d\in \N$, $r>0$, $\frac
rd+\frac{1}{q}-\frac{1}{p}=0$. Then the operator
$$
Tf(x)=\int \limits _{\R^d} f(y)|x-y|^{r-d}\, dy
$$
is bounded from $L_p(\R^d)$ in $L_q(\R^d)$.
\end{trma}
Reshetnyak \cite{resh1, resh2} constructed the integral
representation for smooth functions defined on a John domain
$\Omega$ in terms of their derivatives of order $r$. We shall use
the following form of his result (see also \cite{vas_john}).
\begin{trma}
\label{reshteor} Let $\Omega\in {\bf FC}(a)$, let the point $x_*$,
the curves $\gamma_x$ and the numbers $T(x)$ be such as in
Definition \ref{fca}, and let $R_0={\rm
dist}_{\|\cdot\|_{l_2^d}}\, (x_*, \, \partial \Omega)$, $r\in \N$.
Then there exist measurable functions
$H_{\overline{\beta}}:\Omega\times \Omega\rightarrow \R$,
$\overline{\beta}=(\beta_1, \, \dots, \, \beta_d)\in \Z_+^d$,
$|\overline{\beta}|=r$, such that the inclusion ${\rm supp}\,
H_{\overline{\beta}}(x, \, \cdot)\subset \cup _{t\in [0, \,
T(x)]}B_{at}(\gamma_x(t))$ and the inequality
$|H_{\overline{\beta}}(x, \, y)| \underset{a,d,r}
{\lesssim}|x-y|^{r-d}$ hold for any $x\in \Omega$. Moreover, the
following representation holds:
$$
f(x)=\sum \limits _{|\overline{\beta}|=r}\int \limits _\Omega
H_{\overline{\beta}}(x, \, y) \frac{\partial ^r f(y)}{\partial
^{\overline{\beta}} y} \, dy, \quad f\in C^\infty(\Omega), \quad
f|_{B_{R_0/2}(x_*)}=0.
$$
\end{trma}
The proof of the following lemma is straightforward and will be
omitted.
\begin{Lem}
\label{sum_lem} Let $\Phi:(0, \, +\infty) \rightarrow (0, \,
+\infty)$, $\rho:(0, \, +\infty) \rightarrow (0, \, +\infty)$ be
absolutely continuous functions and let $\lim \limits_{t\to
+0}\frac{t\Phi'(t)}{\Phi(t)}=0$, $\lim \limits_{y\to
+\infty}\frac{y\rho'(y)}{\rho(y)}=0$. Then for any $\varepsilon
>0$
$$
t^\varepsilon\underset {\varepsilon} {\lesssim} \Phi(t)
\underset{\varepsilon}{\lesssim} t^{-\varepsilon}, \quad \text{if}
\quad t\in (0, \, 1],\quad t^{-\varepsilon} \underset
{\varepsilon} {\lesssim} \rho(t) \underset {\varepsilon}
{\lesssim} t^\varepsilon, \quad \text{if} \quad t\in [1, \,
\infty).
$$

Let $\sigma\in \R$, $\mu<-1$. Then for any sequence
$\{k_j\}_{j=0}^l\subset \Z_+$ such that $k_0<k_1<\dots<k_l$, the
following estimates hold:
$$
\sum \limits _{j=0}^l 2^{\sigma
k_j}\Phi(2^{-k_j})\underset{\sigma,\Phi}{\lesssim} 2^{\sigma
k_0}\Phi(2^{-k_0}), \text{ if }\sigma <0,
$$
$$
\sum \limits _{j=0}^l 2^{\sigma
k_j}\Phi(2^{-k_j})\underset{\sigma,\Phi}{\lesssim} 2^{\sigma
k_l}\Phi(2^{-k_l}), \text{ if }\sigma >0,
$$
$$
\sum \limits _{j=0}^l k_j^\mu \rho(k_j)
\underset{\mu,\rho}{\lesssim} k_0^{1+\mu}\rho(k_0).
$$
\end{Lem}

\section{Construction of the partition of the tree}
Let $\Theta\subset \Xi\left(\left[-\frac 12, \, \frac
12\right]^d\right)$ be a set of non-overlapping cubes.

\begin{Def}
Let ${\cal G}$ be a graph, and let $F:{\bf V}({\cal G})
\rightarrow \Theta$ be a one-to-one mapping. We say that $F$ is
consistent with the structure of the graph ${\cal G}$ if the
following condition holds: for any adjacent vertices $\xi'$,
$\xi''\in {\bf V}({\cal G})$ the set $\Gamma _{\xi', \xi''}
:=F(\xi')\cap F(\xi'')$ has dimension $d-1$.
\end{Def}

Let $({\cal T}, \, \xi_*)$ be a tree, and let $F:{\bf V}({\cal T})
\rightarrow \Theta$ be a one-to-one mapping consistent with the
structure of the tree ${\cal T}$. For any adjacent vertices
$\xi'$, $\xi''$, we set $\mathaccent '27 \Gamma
_{\xi',\xi''}=\inter _{d-1}\Gamma _{\xi',\xi''}$, and for each
subtree ${\cal T}'$ of ${\cal T}$, we put
\begin{align}
\label{def_dom_by_tree} \Omega _{{\cal T'},F}=\left(\cup _{\xi\in
{\bf V}({\cal T'})}\inter F(\xi)\right) \cup \left(\cup
_{(\xi',\xi'')\in {\bf E}({\cal T'})}\mathaccent '27 \Gamma
_{\xi',\xi''}\right).
\end{align}
For $\xi \in {\bf V}({\cal T})$, $\Delta=F(\xi)$, denote
$m_\xi={\bf m}(\Delta)$, $\Omega_{\le \Delta}=\Omega_{[\xi_*, \,
\xi],F}$.

Let $\Theta(\Omega)$ be a Whitney covering of $\Omega$ (see
Theorem \ref{whitney}). The following lemma is proved in
\cite{vas_john}.
\begin{Lem}
\label{cor_omega_t} Let $\Omega\subset \left[-\frac 12, \, \frac
12\right]^d$, $\Omega\in {\bf FC}(a)$. Then there exist a tree
${\cal T}$ and a one-to-one mapping $F:{\bf V}({\cal T})
\rightarrow \Theta(\Omega)$ consistent with the structure of
${\cal T}$ and which satisfies the following properties:
\begin{enumerate}
\item for any subtree ${\cal T}'$ of ${\cal T}$,
\begin{align}
\label{fghjkl} \Omega_{{\cal T}',F}\in {\bf FC}(b_*), \text{ where
}b_*=b_*(a, \, d)>0;
\end{align}
\item if $x\in F(\xi)$, then a curve $\gamma_x$ from Definition \ref{fca}
can be chosen so that $B_{b_*t}(\gamma_x(t))\subset \Omega_{\le
F(\xi)}$ for any $t\in [0, \, T(x)]$; if $\xi_{{\cal T}'}$ is a
minimal vertex of ${\cal T}'$, then the center of the cube
$F(w_{{\cal T}'})$ can be taken as a point $x_*$ from Definition
\ref{fca} with $\Omega_{{\cal T}',F}\in {\bf FC}(b_*)$.
\end{enumerate}
\end{Lem}

For $\xi\in {\bf V}({\cal T})$, we set $\Omega_\xi=\Omega_{{\cal
T}_{\xi},F}$; the number $k_\xi\in \Z_+$ is chosen so that
\begin{align}
\label{2kw} 2^{-k_\xi}\le {\rm dist}_{|\cdot|}(F(\xi), \, \Gamma)<
2^{-k_\xi+1}.
\end{align}
By Theorem \ref{whitney},
\begin{align}
\label{2mxi} 2^{-m_\xi}\underset{d}{\asymp} {\rm
dist}_{|\cdot|}(F(\xi), \, \partial \Omega) \le {\rm
dist}_{|\cdot|}(F(\xi), \, \Gamma)\asymp 2^{-k_\xi};
\end{align}
hence, there exists $\vartheta(d)\in \Z_+$ such that
\begin{align}
\label{kwmw} k_\xi\le m_\xi+\vartheta(d).
\end{align}
Let $z_\xi\in F(\xi)$ be such that ${\rm dist}_{|\cdot|}(z_\xi, \,
\Gamma)={\rm dist}_{|\cdot|}(F(\xi), \, \Gamma)$, and let $\tilde
z_\xi$ be a center of the cube $F(\xi)$. Then the first relation
in (\ref{2mxi}) together with
\begin{align}
\label{zx} |z_\xi-\tilde z_\xi|\le 2^{-m_\xi}
\end{align}
imply that for any $x\in \Omega_\xi$
$$
|x-z_\xi|\le {\rm diam}_{|\cdot|}\Omega_\xi \stackrel
{(\ref{diam_dist})}{\underset{a,d}{\lesssim}} {\rm dist}_{|\cdot|}
(\tilde z_\xi, \, \partial \Omega)\le {\rm dist}_{|\cdot|}(F(\xi),
\, \partial \Omega)+|z_\xi-\tilde z_\xi|\stackrel{(\ref{2mxi}),
(\ref{zx})}{\underset{d}{\lesssim}} 2^{-m_\xi}.
$$
Hence, there exists $c(a, \, d)>0$ such that
\begin{align}
\label{xxiw} |x-z_\xi|\le c(a, \, d)\cdot 2^{-m_\xi}, \;\;x\in
\Omega_\xi.
\end{align}
Prove that
\begin{align} \label{dist_xg} {\rm dist}_{|\cdot|}(x,
\, \Gamma)\underset{d}{\asymp}2^{-k_\xi}
\end{align}
for any $x\in F(\xi)$. Indeed,
$$
2^{-k_\xi}\stackrel{(\ref{2kw})}{\le} {\rm dist}_{|\cdot|}(F(\xi),
\, \Gamma)\le {\rm dist}_{|\cdot|}(x, \, \Gamma)\le {\rm dist}
_{|\cdot|}(F(\xi), \,\Gamma)+|x-z_\xi| \stackrel {(\ref{2kw})}
{\underset{d}{\lesssim}}
$$
$$
\lesssim 2^{-k_\xi+1}+ 2^{-m_\xi}\stackrel{(\ref{kwmw})}
{\underset{d}{\lesssim}} 2^{-k_\xi}.
$$

Denote
\begin{align}
\label{hatw} \hat{\bf W}=\{\xi\in {\bf V}({\cal T}):\, m_\xi\le
k_\xi+1+\log c(a, \, d)\}.
\end{align}
From (\ref{kwmw}) it follows that for any $\xi\in \hat{\bf W}$
\begin{align}
\label{mwas_kw} 2^{-m_\xi}\underset{a,d}{\asymp} 2^{-k_\xi}.
\end{align}

Let $\xi\notin \hat{\bf W}$. We show that for any $x\in
\Omega_\xi$
\begin{align}
\label{nw_dg} {\rm dist}_{|\cdot|}(x, \,
\Gamma)\underset{a,d}{\asymp} 2^{-k_\xi}.
\end{align}
Indeed,
$$
{\rm dist}_{|\cdot|}(x, \, \Gamma)\le {\rm dist}_{|\cdot|}(z_\xi,
\, \Gamma)+|x-z_\xi|
\stackrel{(\ref{2kw}),(\ref{xxiw})}{\underset{a,d}{\lesssim}}
2^{-k_\xi}+2^{-m_\xi}\stackrel{(\ref{kwmw})}{\underset{d}{\lesssim}}
2^{-k_\xi},
$$
$$
{\rm dist}_{|\cdot|}(x, \, \Gamma)\ge {\rm dist}_{|\cdot|}(z_\xi,
\, \Gamma)-|x-z_\xi| \stackrel{(\ref{2kw}),(\ref{xxiw})}{\ge}
2^{-k_\xi}-c(a, \, d)\cdot 2^{-m_\xi}\stackrel{(\ref{hatw})}{\ge}
2^{-k_\xi-1}.
$$
Denote
\begin{align}
\label{h_w_nu} \hat{\bf W}_\nu=\{\xi\in \hat{\bf W}:\,k_\xi=\nu\}.
\end{align}
Then (\ref{xxiw}) and (\ref{mwas_kw}) imply that for any $\xi\in
\hat{\bf W}_\nu$ and for any tree ${\cal T}'\subset {\cal T}_\xi$
rooted at $\xi$
\begin{align}
\label{diam_ot} {\rm diam}_{|\cdot|}\, \Omega_{{\cal T}', \, F}
\underset{a,d}{\asymp} 2^{-\nu}.
\end{align}
\begin{Lem}
\label{razb_der} There exist a partition of the tree ${\cal T}$
into subtrees ${\cal T}_{k,i}$ with minimal vertices $\xi_{k,i}$,
$k\in \Z_+$, $i\in I_k$, $I_k\ne \varnothing$, and numbers
$\nu_k\in \N$, satisfying the following conditions:
\begin{enumerate}
\item $\nu_0<\nu_1<\dots<\nu_k<\dots$;
\item $\xi_{k,i}\in \hat{\bf W}_{\nu_k}$;
\item ${\rm dist}_{|\cdot|}(x, \, \Gamma)\underset{a,d}{\asymp}
2^{-\nu_k}$ for any $x\in \Omega_{{\cal T}_{k,i},F}$;
\item if $\xi_{k',i'}<\xi_{k,i}$, then $k'<k$.
\end{enumerate}
\end{Lem}
\begin{proof}
Let $\nu\in \Z_+$, $\hat \xi\in \hat{\bf W}_\nu$. Denote by
$\mathfrak{T}(\hat\xi)$ a set of subtrees ${\cal T}'\subset {\cal
T}_{\hat \xi}$ with the minimal vertex $\hat \xi$ such that
\begin{align}
\label{tw_def} {\bf V}({\cal T}')\bigcap \left(\bigcup_{l\ge
\nu+1}\hat {\bf W}_l\right)=\varnothing
\end{align}
(this set is nonempty, since $\{\hat \xi\} \in
\mathfrak{T}(\hat\xi)$). Denote by ${\cal S}({\cal T}_{\hat \xi})$
a subtree in ${\cal T}_{\hat \xi}$ such that ${\bf V}({\cal
S}({\cal T}_{\hat \xi}))=\cup _{S\in \mathfrak{T}(\xi)} {\bf
V}(S)$. Then ${\cal S}({\cal T}_{\hat \xi})\in
\mathfrak{T}(\hat\xi)$.

Prove that there exists $\hat\nu=\hat\nu(a, \, d)\in \N$ such that
for any $x\in \Omega_{{\cal S}({\cal T}_{\hat \xi}),F}$
$$
2^{-\nu -\hat\nu}\le {\rm dist}_{|\cdot|}(x, \, \Gamma)\le 2^{-\nu
+\hat\nu}.
$$
Indeed,
$$
{\rm dist}_{|\cdot|}(x, \, \Gamma)\le |x-z_{\hat \xi}|+{\rm
dist}_{|\cdot|}(z_{\hat \xi}, \, \Gamma)
\stackrel{(\ref{2kw}),(\ref{xxiw}),(\ref{h_w_nu})}{\underset{a,d}{\lesssim}}
2^{-m_{\hat \xi}}+2^{-\nu}
\stackrel{(\ref{kwmw})}{\underset{d}{\lesssim}}2^{-\nu}.
$$
Prove the estimate from below. Let $x\in F(\eta)$, $\eta\in {\bf
V}({\cal S}({\cal T}_{\hat \xi}))$. Set
$$
\hat \eta=\max \{\hat {\bf W}\cap [\hat \xi, \, \eta]\}.
$$
Then $\hat \eta\in \hat{\bf W}_j$ for some $j\in \Z_+$; since
${\cal S}({\cal T}_{\hat \xi})\in \mathfrak{T}(\hat \xi)$, we have
$j\le \nu$. If $\eta=\hat \eta$, then
\begin{align}
\label{dfu_2nu} {\rm dist}_{|\cdot|}(F(\eta), \, \Gamma)
\stackrel{(\ref{2kw})} {\ge} 2^{-j}\ge 2^{-\nu}.
\end{align}
Let $\eta>\hat \eta$, $\hat \zeta\in [\hat \eta, \, \eta]\cap {\bf
V}_1(\hat \eta)$. Then $\hat \zeta\notin \hat {\bf W}$, $\dim
(F(\hat \eta)\cap F(\hat \zeta))=d-1$, and for any $x\in F(\eta)$
$$
{\rm dist}_{|\cdot|}(x, \, \Gamma)\stackrel{(\ref{nw_dg})}
{\underset{a,d}{\asymp}}2^{-k_{\hat \zeta}}
\stackrel{(\ref{kwmw})} {\underset{d}{\gtrsim}} 2^{-m_{\hat
\zeta}} \stackrel{(\ref{bfm})} {\ge} 2^{-m_{\hat \eta}-2}
\stackrel{(\ref{mwas_kw})} {\underset{a,d} {\asymp}} 2^{-k_{\hat
\eta}} \stackrel{(\ref{h_w_nu}),(\ref{dfu_2nu})}{\ge} 2^{-\nu}.
$$

Let $\xi_0$ be a minimal vertex of the tree ${\cal T}$. Prove that
$\xi_0\in \hat {\bf W}$. Indeed, otherwise (\ref{nw_dg}) imply
that ${\rm dist}_{|\cdot|}(x, \, \Gamma)\underset{a,d}{\asymp}
2^{-k_{\xi_0}}$ for any $x\in \Omega_{{\cal T},F}$. Hence, $\Gamma
\not \subset \partial \Omega$, which leads to a contradiction.

The further arguments are the same as the arguments in Lemma 2
from \cite{vas_sing}.
\end{proof}

\begin{Sta}
\label{dif_nu} Let $\xi_{k,i}<\xi_{k',i'}$, $\{\xi:\, \xi_{k,i}\le
\xi<\xi_{k',i'}\}\subset {\bf V}({\cal T}_{k,i})$. Then
$\nu_{k'}\le \nu_k+\overline{s}$, with $\overline{s}=\overline{s}
(a, \, d)$.
\end{Sta}
\begin{proof}
Let $\xi\in [\xi_{k,i}, \, \xi_{k',i'}]$ be the direct predecessor
of $\xi_{k',i'}$. Then $\xi\in {\bf V}({\cal T}_{k,i})$. By
Assertion 3 of Lemma \ref{razb_der}, ${\rm dist}_{|\cdot|}(x, \,
\Gamma) \underset{a,d}{\asymp} 2^{-\nu_k}$ for any $x\in F(\xi)$,
as well as ${\rm dist}_{|\cdot|}(x, \, \Gamma)
\underset{a,d}{\asymp} 2^{-\nu_{k'}}$ for any $x\in
F(\xi_{k',i'})$. Since the mapping $F$ is consistent with the
structure of the tree ${\cal T}$, then $F(\xi)\cap
F(\xi_{k',i'})\ne \varnothing$. Hence,
$2^{-\nu_k}\underset{a,d}{\asymp} 2^{-\nu_{k'}}$. This completes
the proof.
\end{proof}

\begin{Lem}
\label{nu_st} Let $\hat \xi\in \hat{\bf W}_{\nu_0}$. In addition,
suppose that there exists $c_0\ge 1$ such that
\begin{align}
\label{h_cond_1} \forall j\in \N, \;\; t, \; s\in [2^{-j-1}, \,
2^{-j+1}]\quad \quad c_0^{-1}\le \frac{h(t)}{h(s)}\le c_0.
\end{align}
Given $\nu\ge \nu_0$, we denote
$$
\hat{\bf W}_{\nu}(\hat\xi)=\hat{\bf W}_\nu\cap {\bf V}({\cal
T}_{\hat\xi}).
$$
Then
\begin{align}
\label{w_nu_d0} {\rm card}\, \hat{\bf
W}_{\nu}(\hat\xi)\underset{a,d,c_0}{\lesssim}
\frac{h(2^{-\nu_0})}{h(2^{-\nu})}.
\end{align}
If, in addition, $\hat\xi$ is a root of the tree ${\cal T}$, then
there is $\hat k=\hat k(a, \, d)\in \N$ such that
\begin{align}
\label{w_nu_d0_as} \sum \limits_{j=\nu}^{\nu+\hat k} {\rm card}\,
\hat{\bf W}_j(\hat\xi)\underset{a,d,c_0}{\gtrsim}
\frac{h(2^{-\nu_0})}{h(2^{-\nu})}.
\end{align}
\end{Lem}
\begin{proof}
Throughout the proof of this lemma we suppose that a cube is a
product of semi-intervals $\prod _{j=1}^d [a_j, \, b_j)$. Then any
two non-overlapping cubes do not intersect.

If $\xi\in \hat{\bf W}_{\nu}$, then it follows from
(\ref{mwas_kw}) that there is $k_*(a, \, d)$ such that $\nu-k_*(a,
\, d)\le m_\xi\le \nu+k_*(a, \, d)$. Further,
\begin{align}
\label{diam_o_xi} {\rm diam}\, \Omega_{\hat\xi} \stackrel
{(\ref{xxiw}), (\ref{mwas_kw})}{\underset{a,d}{\asymp}}
2^{-\nu_0}.
\end{align}
Therefore, if $\xi\in \hat{\bf W}_{\nu}(\hat \xi)$, then
$2^{-\nu}\underset{a,d}{\asymp} 2^{-m_\xi} \le {\rm diam}\,
\Omega_{\hat \xi}\underset{a,d}{\asymp} 2^{-\nu_0}$. Hence,
\begin{align}
\label{wn} 2^\nu\underset{a,d}{\gtrsim} 2^{\nu_0}, \quad
\text{if}\quad \hat {\bf W}_\nu(\hat \xi)\ne \varnothing.
\end{align}
Let $k\in \Z$, $-k_*(a, \, d)\le k\le k_*(a, \, d)$. Denote by
$\hat{\bf W}_{\nu,k}(\hat \xi)$ the set of $\xi\in \hat{\bf
W}_{\nu}(\hat \xi)$ such that $m_\xi=\nu+k$.

It follows from (\ref{2kw}) and (\ref{h_w_nu}) that
$$
{\rm dist}_{|\cdot|}(F(\xi), \, \Gamma)\le 2^{-\nu+1}, \quad
\xi\in \hat{\bf W}_{\nu}(\hat \xi).
$$
This together with (\ref{diam_o_xi}) implies that there exists a
cube $\Delta_0$ and a number $k_0=k_0(a, \, d)\in \Z_+$ such that
\begin{align}
\label{fxi1} F(\hat\xi)\in \Xi(\Delta_0),\quad \nu_0-k_0(a, \,
d)\le {\bf m} (\Delta_0)\le \nu_0+k_0(a, \, d),
\end{align}
\begin{align}
\label{fxi2} \Omega_{\hat \xi}\subset \Delta_0, \quad \exists
\overline{x}\in \Gamma\cap \Delta_0:\;\; {\rm
dist}_{|\cdot|}(\overline{x}, \, \partial
\Delta_0)\underset{a,d}{\gtrsim} 2^{-\nu_0},
\end{align}
\begin{align}
\label{fxi3} \forall \xi\in\hat{\bf W}_{\nu}(\hat \xi) \;\exists
x\in \Gamma \cap \Delta_0:\; {\rm dist}_{|\cdot|}(x, \, F(\xi))\le
2^{-\nu+1}.
\end{align}

Let $j\in \N$, $j\ge \nu_0+k_*(a, \, d)$. In this case, if
$\Delta\in \Xi_j\left(\left[-\frac 12, \, \frac
12\right]^d\right)$, then either $\Delta\subset \Delta_0$ or
$\Delta$ does not overlap with $\Delta_0$. It follows from the
conditions $F(\hat \xi)\in \Xi(\Delta_0)$ and $j\ge \nu_0+k_*(a,
\, d)\ge m_{\hat \xi}$. Denote by $\Delta_{0,j}$ a cube that is
obtained from $\Delta_0$ by a dilatation in respect to its center,
with a side length ${\bf m}(\Delta_0)+2\cdot 2^{-j}$. Set
$$
\Theta_j(\Delta_0)=\left\{\Delta \in \Xi_j\left(\left[-\frac 12,
\, \frac 12\right]^d\right):\; \Delta\subset \Delta_0, \; \Delta
\cap \Gamma \ne \varnothing\right\},
$$
$$
\tilde\Theta_j(\Delta_0)=\left\{\Delta \in \Xi_j\left(\left[-\frac
12, \, \frac 12\right]^d\right):\; \Delta\subset \Delta_{0,j}, \;
\Delta \cap \Gamma \ne \varnothing\right\}.
$$

Prove that
\begin{align}
\label{tn0} {\rm card}\, \Theta_j(\Delta_0) \underset {a,d,c_0}
{\asymp} \frac{h(2^{-\nu_0})}{h(2^{-j})}.
\end{align}
Let $\Delta\in \Theta_j(\Delta_0)$. Since $\Delta\cap \Gamma \ne
\varnothing$, there is a cube $K_\Delta$ centered at $z_\Delta$,
such that
\begin{align}
\label{d_in_xi} \Delta\in \Xi_1(K_\Delta),\quad {\rm dist}
_{|\cdot|}(z_\Delta, \, \Gamma)\le 2^{-{\bf m}(\Delta)-1}.
\end{align}
Then there are
\begin{align}
\label{zd_t_b} \tilde z_\Delta\in \Gamma, \;\; t_\Delta
\underset{d}{\gtrsim} 2^{-j}, \;\; \tilde t_\Delta
\underset{d}{\lesssim} 2^{-j} \;\;\text{ such that }\;\;
B_{t_\Delta} (\tilde z_\Delta) \subset K_\Delta \subset B _{\tilde
t_\Delta}(\tilde z_\Delta).
\end{align}
Let $\mu$ be a measure from the definition \ref{h_set} (in
particular, ${\rm supp}\, \mu\subset \Gamma$). Then
(\ref{h_cond_1}) and (\ref{zd_t_b}) imply that $\mu(K_\Delta)
\underset {d,c_0} {\asymp} h(2^{-j})$. On the other hand, by
(\ref{h_cond_1}), (\ref{fxi1}) and (\ref{fxi2}),
$$
h(2^{-\nu_0})\underset{a,d,c_0}{\asymp} \mu(\Delta_0) =\sum
\limits _{\Delta\in \Theta_j(\Delta_0)}\mu(\Delta),
$$
$$
h(2^{-\nu_0})\underset{a,d,c_0}{\asymp} \mu(\Delta_{0,j}) =\sum
\limits _{\Delta\in \tilde\Theta_j(\Delta_0)}\mu(\Delta).
$$
Therefore, in order to prove (\ref{tn0}) it is sufficient to check
that $\sum \limits _{\Delta\in \Theta_j(\Delta_0)}\mu(\Delta) \le
\sum \limits _{\Delta\in \Theta_j(\Delta_0)}\mu(K_\Delta)
\underset{d}{\lesssim} \sum \limits _{\Delta\in
\tilde\Theta_j(\Delta_0)}\mu(\Delta)$. The first inequality holds
since the measure $\mu$ is nonnegative. Prove the second
inequality. Since $\Delta\in \Xi_1(K_\Delta)$, we have
$K_\Delta\subset \Delta_{0,j}$. Denote
$$
\Theta_{j,\Delta}=\{\Delta '\in \tilde\Theta_j(\Delta_0):\,
\Delta'\subset K_\Delta\} \quad \mbox{for}\quad \Delta\in
\Theta_j(\Delta_0),
$$
$$
\Theta'_{j,\Delta'}=\{\Delta\in \Theta_j(\Delta_0):\;
\Delta'\subset K_\Delta\} \quad \mbox{for}\quad \Delta' \in \tilde
\Theta_j(\Delta_0).
$$
Since ${\rm card}\, \Theta'_{j,\Delta'}\underset{d}{\lesssim} 1$
for any $\Delta'\in \tilde\Theta_j(\Delta_0)$, we have
$$
\sum \limits _{\Delta\in \Theta_j(\Delta_0)}\mu(K_\Delta) =\sum
\limits _{\Delta\in \Theta_j(\Delta_0)} \sum \limits
_{\Delta'\in\Theta_{j,\Delta}} \mu(\Delta')  \le\sum \limits
_{\Delta '\in \tilde\Theta_j(\Delta_0)} \sum
\limits_{\Delta\in\Theta'_{j,\Delta'}} \mu(\Delta')
\underset{d}{\lesssim} \sum \limits _{\Delta '\in
\tilde\Theta_j(\Delta_0)} \mu(\Delta').
$$
This proves (\ref{tn0}).

Show that if $-k_*(a, \, d)\le k\le k_*(a, \, d)$, $\nu\ge
\nu_0+2k_*(a, \, d)$, then
\begin{align}
\label{cw_n_k} {\rm card}\, \hat{\bf W}_{\nu,k}(\hat \xi)
\underset{a,d} {\lesssim} {\rm card}\,\Theta_{\nu+k}(\Delta_0)
\end{align}
(recall that $\Theta_j(\Delta_0)$ was defined for $j\ge \nu_0
+k_*(a, \, d)$).

Set
$$
A:=\left\{\Delta'\in \Xi_{\nu+k}\left(\left[-\frac 12, \, \frac
12\right]^d\right):\;\exists \Delta\in \Theta_{\nu+k}(\Delta_0):\;
{\rm dist}_{|\cdot|}(\Delta', \, \Delta)\le 2^{-\nu+1}\right\}.
$$
From (\ref{fxi3}) it follows that $\{F(\xi):\, \xi\in\hat{\bf
W}_{\nu,k}(\hat \xi)\}\subset A$ and ${\rm card}\,\hat{\bf
W}_{\nu,k}(\hat \xi)\le {\rm card}\, A \underset{a,d}{\lesssim}
{\rm card}\,\Theta_{\nu+k}(\Delta_0)$.

If $\nu\ge \nu_0+2k_*(a, \, d)$, then (\ref{h_cond_1}),
(\ref{tn0}) and (\ref{cw_n_k}) imply (\ref{w_nu_d0}). If $\hat
{\bf W}_\nu(\hat\xi)\ne \varnothing$ and $\nu<\nu_0+2k_*(a, \,
d)$, then by (\ref{wn}) we get $2^{-\nu}\underset{a,d}{\asymp}
2^{-\nu_0}$; hence, $\frac{h(2^{-\nu_0})}{h(2^{-\nu})}
\stackrel{(\ref{h_cond_1})}{\underset{a,d,c_0}{\asymp}} 1$. This
together with (\ref{mwas_kw}), (\ref{h_w_nu}), (\ref{fxi1}) and
(\ref{fxi2}) yield (\ref{w_nu_d0}).

Let us prove (\ref{w_nu_d0_as}). By (\ref{tn0}), it is sufficient
to check
\begin{align}
\label{sljn} \sum \limits_{j=\nu}^{\nu+\hat k} {\rm card}\,
\hat{\bf W}_j(\hat\xi)\underset{a,d,c_0}{\gtrsim} {\rm card}\,
\Theta_\nu(\Delta_0).
\end{align}
Let $\Delta\in \Theta_l(\Delta_0)$, and let $K_\Delta$ be the cube
defined above (see (\ref{d_in_xi})),
\begin{align}
\label{xd_kd} x_\Delta\in K_\Delta\cap \Gamma, \quad
|x_\Delta-z_\Delta|\le 2^{-{\bf m}(\Delta)-1}.
\end{align}
Since $\Gamma \subset \partial \Omega$, by Theorem \ref{whitney}
and Lemma \ref{cor_omega_t} for any $m\in \N$ there is a vertex
$\xi_{\Delta}\in {\bf V}({\cal T})$ such that
\begin{align}
\label{dxd} {\rm dist}\, (x_\Delta, \, F(\xi_{\Delta}))< 2^{-m},
\quad 2^{-m_{\xi_\Delta}} \underset{d}{\lesssim} 2^{-m}.
\end{align}
If $m$ is sufficiently large, then it follows from (\ref{xd_kd})
and (\ref{dxd}) that $F(\xi_\Delta)\subset K_\Delta$. Denote
$\eta_\Delta=\min\{\eta\in [\hat \xi, \, \xi_\Delta]:\;
F(\eta)\subset K_\Delta\}$. Show that
\begin{align}
\label{2medel} 2^{-m_{\eta_\Delta}}\underset{a,d}{\asymp} 2^{-l}
\quad\text{and}\quad \eta_\Delta\in \hat{\bf W}\quad\text{for
sufficiently large }m.
\end{align}
Indeed, since $\Delta\in \Theta_l(\Delta_0)$, we have $2^{-{\bf
m}(\Delta)}=2^{-l}$. The inclusion $F(\eta_\Delta)\subset
K_\Delta$ and the first relation in (\ref{d_in_xi}) imply that
$2^{-m_{\eta_\Delta}}\lesssim 2^{-{\bf m}(\Delta)}$. Check that
$2^{-m_{\eta_\Delta}}\underset{a,d}{\gtrsim} 2^{-{\bf
m}(\Delta)}$. It follows from the definition of $\eta_\Delta$ that
$\partial F(\eta_\Delta)\cap \partial K_\Delta\ne \varnothing$.
Let $\hat x\in \partial F(\eta_\Delta)\cap
\partial K_\Delta$. Then $|\hat
x-z_\Delta|\stackrel{(\ref{d_in_xi})}{=}2^{-{\bf m}(\Delta)}$. By
(\ref{dxd}), there is a point $\hat y\in F(\xi_\Delta)$ such that
$|x_\Delta-\hat y|\le 2^{-m}$. Since $F(\xi_\Delta) \subset
\Omega_{\eta_\Delta}$, we have
$$
2^{-m_{\eta_\Delta}}\stackrel{(\ref{xxiw})}{\underset{a,d}{\gtrsim}}{\rm
diam}_{|\cdot|} \Omega_{\eta_\Delta} \ge |\hat x-\hat y|\ge
$$
$$
\ge|\hat x-z_\Delta|-|z_\Delta-x_\Delta| -|x_\Delta-\hat y|
\stackrel{(\ref{xd_kd})} {\ge} 2^{-{\bf m}(\Delta)-1}-2^{-m}\ge
2^{-{\bf m}(\Delta)-2}
$$
for large $m$. The first relation in (\ref{2medel}) is proved.
Check the second relation. Let $\eta_\Delta\notin \hat{\bf W}$.
Then, by (\ref{nw_dg}), for any $x\in \Omega_{\eta_\Delta}$ we
have ${\rm dist}_{|\cdot|}(x, \, \Gamma)\underset{a,d}{\asymp}
2^{-k_{\eta_\Delta}}$. Taking $x\in F(\xi_\Delta)\subset
\Omega_{\eta_\Delta}$, we get
$$
2^{-m_{\eta_\Delta}} \stackrel{(\ref{2mxi})}{\underset {a,d}
{\lesssim}} 2^{-k_{\eta_\Delta}}\underset{a,d}{\asymp} {\rm
dist}_{|\cdot|}(x, \, \Gamma)\le |x-x_\Delta|
\stackrel{(\ref{dxd})}{\underset{d}{\lesssim}} 2^{-m}.
$$
It follows from the proved first relation in (\ref{2medel}) that
$2^{-l}\underset{a,d}{\lesssim} 2^{-m}$. It is impossible for
large $m$.

It follows from (\ref{mwas_kw}), (\ref{h_w_nu}) and (\ref{2medel})
that $\eta_\Delta\in \hat{\bf W}_j$ for some $j\in \Z_+$ such that
$2^{-j}\underset{a,d}{\asymp} 2^{-l}$. Therefore,
\begin{align}
\label{llls} l-l_*\le j\le l+l_*, \quad \mbox{with}\quad
l_*=l_*(a, \, d)\in \N.
\end{align}
Set $\hat k=2l_*$. In order to prove (\ref{w_nu_d0_as}), we take
$j=\nu+l_*$ and apply (\ref{llls}). We get
$${\rm card}\, \{\eta_\Delta:\; \Delta\in \Theta_{\nu+l_*}(\Delta_0)\} \le \sum \limits
_{l=\nu}^{\nu+\hat k} {\rm card}\, \hat{\bf W}_l.$$ Hence, in
order to prove (\ref{sljn}) it is sufficient to check
\begin{align}
\label{ceta} {\rm card}\, \Theta _j(\Delta_0)\underset{d}
{\lesssim} {\rm card}\, \{\eta_\Delta:\; \Delta\in
\Theta_j(\Delta_0)\}
\end{align}
and to apply (\ref{tn0}) with (\ref{h_cond_1}). Let $\Delta$,
$\Delta'\in \Theta_j(\Delta_0)$, $\eta_\Delta=\eta_{\Delta'}$.
Then $K_\Delta\cap K_{\Delta'}\supset\Delta''$, $\Delta''\in
\Xi_1(K_\Delta)$ and $\Delta''\in \Xi_1(K_{\Delta'})$. Therefore,
${\rm card}\, \{\eta_{\Delta'}:\, \Delta'\in \Theta_j(\Delta_0),
\; \eta_{\Delta'}=\eta_\Delta\}\underset{d}{\lesssim}1$, which
implies (\ref{ceta}). This completes the proof.
\end{proof}
\vskip 0.5cm

Let $m\in \N$. For $0<t_0<t_1\le \infty$ denote by ${\cal G}^{t_0,
\, t_1}$ the maximal subgraph in ${\cal T}$ on the set of vertices
$$
{\bf V}({\cal G}^{t_0, \, t_1}):=\bigcup _{t_0\le
\nu_k<t_1}\bigcup _{i\in I_k} {\cal T}_{k,i}
$$
(the index set $I_k$ was defined in Lemma \ref{razb_der});
\label{aij}by $\{{\cal D}_{j,i}\}_{i\in \tilde I_j}$ we denote the
set of all connected components of the graph ${\cal G}^{1+mj, \,
1+m(j+1)}$; by $\hat \xi_{j,i}=\hat \xi_{j,i}^m$ denote the
minimal vertex of the tree ${\cal D}_{j,i}$, $j\in \Z_+$. Then
\begin{enumerate}
\item $\hat \xi_{j,i}\in \hat{\bf W}_{\nu_k}$ for some
$\nu_k\in [1+mj, \, 1+m(j+1))$; in particular,
\begin{align}
\label{diam_dj} {\rm diam}\, \Omega_{{\cal
D}_{j,i},F}\stackrel{(\ref{diam_ot})}{\underset{a,d,m}{\asymp}}
2^{-mj};
\end{align}
\item for any $x\in \Omega _{{\cal D}_{j,i},F}$
\begin{align}
\label{dist_x1} {\rm dist}_{|\cdot|}(x, \, \Gamma)
\underset{a,d,m}{\asymp} 2^{-mj}
\end{align}
(it follows from Assertion 3 of Lemma \ref{razb_der});
\item if $\hat \xi_{j,i}<\hat \xi_{j',i'}$, then $j<j'$
(indeed, $\hat \xi_{j,i}=\xi_{k,t}$ and $\hat
\xi_{j',i'}=\xi_{k',t'}$ for some $k$, $t$, $k'$, $t'$; by
Assertions 4 and 1 of Lemma \ref{razb_der}, $\nu_k<\nu_{k'}$; it
implies that $j\le j'$; the equality $j=j'$ is impossible; indeed,
in this case the vertices $\hat \xi_{j,i}$ and $\hat \xi_{j',i'}$
are incomparable).
\end{enumerate}
Let $\hat \xi_{j,i}<\hat \xi_{j',i'}$, $$\{\xi: \; \hat
\xi_{j,i}\le \xi<\hat \xi_{j',i'}\}\subset {\cal D}_{j,i}.$$ Then
we say that the tree ${\cal D}_{j',i'}$ follows the tree ${\cal
D}_{j,i}$.
\begin{Rem}
\label{jjj} Let $\overline{s}$ be such as in Proposition
\ref{dif_nu}, let $m\ge \overline{s}$, and let ${\cal D}_{j',i'}$
follow the tree ${\cal D}_{j,i}$. Then $j'=j+1$.
\end{Rem}
Indeed, let $\xi_{t,s}\in {\cal D}_{j,i}$, $\{\xi:\,
\xi_{t,s}<\xi<\hat \xi_{j',i'}\}\subset {\bf V}({\cal T}_{t,s})$,
$\hat\xi_{j',i'}=\xi_{t',s'}$. By Proposition \ref{dif_nu},
$1+mj'\le \nu_{t'}\le \nu_t+\overline{s}< 1+m(j+1)+\overline{s}$.
Hence, $m(j'-j-1)< \overline{s}$. Since $m\ge\overline{s}$, the
last inequality is possible only for $j'=j+1$.

Given $j\in \Z_+$, $l\in \N$, $t\in \tilde I_j$, we denote
\begin{align}
\label{ijtl} \tilde I_{j,t}^l=\tilde I_{j,t}^{l,m}=\{i\in \tilde
I_{j+l}:\, \hat \xi^m_{j+l,i}>\hat \xi^m_{j,t}\}.
\end{align}

\begin{Lem}
\label{litt} Let $m\in \N$ be divisible by $\overline{s}$. Suppose
that (\ref{h_cond_1}) holds for some $c_0\ge 1$. Then
\begin{align}
\label{ijitl} {\rm card} \tilde I^l_{j,t} \underset {a,d,c_0}
{\lesssim} \frac{h(2^{-mj})}{h(2^{-m(j+l)})}.
\end{align}
\end{Lem}
\begin{proof}
First consider the case $m=\overline{s}$.

By the property 1 of the trees ${\cal D}_{j,t}$ and ${\cal
D}_{j+l,i}$,
$$
\hat \xi _{j,t}\in
\bigcup_{\nu'=1+\overline{s}j}^{\overline{s}(j+1)} \hat{\bf
W}_{\nu'}, \quad \hat \xi_{j+l,i}\in
\bigcup_{\nu=1+\overline{s}(j+l)} ^{\overline{s}(j+l+1)} \hat{\bf
W}_{\nu}(\hat \xi_{j,t})
$$
(recall that $\hat{\bf W}_{\nu}(\hat \xi_{j,t})=\hat{\bf W}_{\nu}
\cap {\bf V}({\cal T}_{\hat \xi_{j,t}})$). Therefore, from Lemma
\ref{nu_st} and (\ref{h_cond_1}) it follows that
$$
{\rm card}\, \tilde I^l_{j,t} \le \sum \limits
_{\nu=1+\overline{s}(j+l)}^{\overline{s}(j+l+1)} {\rm card}\,
\hat{\bf W}_{\nu}(\hat \xi_{j,t}) \underset{a,d,c_0}{\lesssim}
\frac{h(2^{-\overline{s}j})}{h(2^{-\overline{s}(j+l)})}.
$$

Consider the case $m=m'\overline{s}$. Then $\hat \xi^m_{j,t}=\hat
\xi^{\overline{s}}_{j',t'}$ for some $j'\ge m'j$, $\hat
\xi^m_{j+l,i}= \hat\xi ^{\overline{s}}_{m'(j+l),i'}$ (by Remark
\ref{jjj}). Hence,
$$
{\rm card}\, \tilde I^{l,m}_{j,t}={\rm card}\, \tilde
I^{m'l,{\overline{s}}}_{m'j,t'} \underset{a,d,c_0}{\lesssim}
\frac{h(2^{-\overline{s}m'j})}{h(2^{-\overline{s}m'(j+l)})}=
\frac{h(2^{-mj})}{h(2^{-m(j+l)})}.
$$
This completes the proof.
\end{proof}

\section{The discrete Hardy-type inequality on a tree: case $p=q$}
\subsection{The analogue of Evans -- Harris -- Pick theorem}
Let $({\cal A}, \, \xi_0)$ be a tree with a finite vertex set, let
$1\le p\le \infty$, and let $u$, $w:{\bf V}({\cal A})\rightarrow
\R_+$ be weight functions. Denote by $\mathfrak{S}_{{\cal A},u,w}$
the minimal constant $C$ in the inequality
\begin{align}
\label{dh} \left(\sum \limits_{\xi \in {\bf V}({\cal A})} w^p(\xi)
\left( \sum \limits _{\xi'\le
\xi}u(\xi')f(\xi')\right)^p\right)^{1/p} \le C\left(\sum
\limits_{\xi \in {\bf V}({\cal A})} f^p(\xi)\right)^{1/p}, \;\;
f:{\bf V}({\cal A})\rightarrow \R_+.
\end{align}

\begin{Rem}
\label{subtr} If ${\cal D}\subset {\cal A}$ is a subtree, then
$\mathfrak{S}_{{\cal D},u,w}\le \mathfrak{S}_{{\cal A},u,w}$.
\end{Rem}

Let us obtain two-sided estimates for $\mathfrak{S}_{{\cal
A},u,w}$. We reduce this problem to estimating the constant in the
Hardy-type inequality on a metric tree and use the result from the
article \cite{ev_har_pick}.

Let $\hat\xi\in {\bf V}({\cal A})$, ${\cal D}\subset {\cal
A}_{\hat \xi}$. We say that ${\cal D}\in {\cal J}'_{\hat \xi}$ if
the following conditions hold:
\begin{enumerate}
\item $\hat \xi$ is the minimal vertex in ${\cal D}$,
\item if $\xi\in {\bf V}({\cal D})$ is not a maximal vertex
${\cal D}$, then ${\bf V}_1(\xi)\subset {\bf V}({\cal D})$.
\end{enumerate}

Denote by $\mathaccent'27 {\cal D}$ the subtree in ${\cal D}$ such
that ${\bf V}(\mathaccent'27 {\cal D})={\bf V}({\cal D})\backslash
{\bf V}_{\max}({\cal D})$.

For any subgraph ${\cal G}\subset {\cal A}$ and for any function
$f:{\bf V}({\cal G})\rightarrow \R$, we denote
\begin{align}
\label{flpg} \|f\|_{l_p({\cal G})}=\left ( \sum \limits _{\omega
\in {\bf V} ({\cal G})}|f(\omega)|^p\right )^{1/p}.
\end{align}
By $l_p({\cal G})$ we denote the space of functions $f:{\bf
V}({\cal G})\rightarrow \R$ equipped with the norm
$\|f\|_{l_p({\cal G})}$.

For ${\cal D}\in {\cal J}'_{\hat \xi}$ we set
\begin{align}
\label{betad} \beta_{{\cal D}}=\inf \left\{ \| f\| _{l_p({\cal
A})}: \, \sum \limits _{\hat \xi\le \xi'\le \xi} f(\xi')
u(\xi')=1, \;\; \forall \xi \in {\bf V}_{\max}({\cal D})\right\}.
\end{align}
Notice that if ${\cal D}=\{\hat \xi\}$, then
\begin{align}
\label{odno_toc} \beta_{\{\hat \xi\}}=\inf \{|f(\hat \xi)|:\;
f(\hat \xi)u(\hat \xi)=1\}=u^{-1}(\hat \xi).
\end{align}

\begin{Lem}
\label{hardy_cr} Suppose that there exists $\hat C\ge 1$ such that
for any $\xi\in {\bf V}({\cal A})$
\begin{align}
\label{ccc} {\rm card}\, {\bf V}_1(\xi)\le \hat C,
\end{align}
and let for any adjacent vertices $\xi$, $\xi'\in {\bf V}({\cal
A})$
\begin{align}
\label{uuu} \hat C^{-1}\le \frac{u(\xi)}{u(\xi')}\le \hat C, \;\;
\hat C^{-1}\le \frac{w(\xi)}{w(\xi')}\le \hat C.
\end{align}
Then
$$
\mathfrak{S}_{{\cal A}_{\hat \xi},u,w}\underset{p,\hat C}{\asymp}
\sup _{{\cal D}\in {\cal J}'_{\hat \xi}} \frac{\|w\chi_{{\cal
A}_{\hat \xi}\backslash \mathaccent'27{\cal D}}\|_{l_p({\cal
A}_{\hat \xi})}} {\beta _{{\cal D}}}.
$$
\end{Lem}
\begin{proof}
If ${\bf V}({\cal A})=\{\xi_0\}$, then the assertion is trivial.

Let ${\bf V}({\cal A})\ne\{\xi_0\}$. Add to the set ${\bf V}({\cal
A})$ a vertex $\xi_*$ and join it with $\xi_0$ by an edge. Thus we
obtain the tree $(\tilde{\cal A}, \, \xi_*)$. Define the mapping
$\Delta$ by $\Delta(\lambda)=[0, \, 1]$, $\lambda \in {\bf
E}(\tilde{\cal A})$. Thus we get the metric tree
${\mathbb{A}}=(\tilde{\cal A}, \, \Delta)$. For any function
$\psi:{\bf V}({\cal A})\rightarrow \R$ we define $\psi^{\#}
:{\mathbb{A}}\rightarrow \R$ as follows. Let $e=(\xi', \, \xi)\in
{\bf E}({\cal A})$, $\xi>\xi'$. Then we set
$\psi^{\#}|_{\Delta(e)}=\psi(\xi)$.

Let $\lambda_{\hat\xi}\in {\bf E}({\cal A})$ be an edge with the
end $\hat \xi$, $x_0=(0, \, \lambda_{\hat \xi})\in\mathbb{A}$. By
H\"{o}lder inequality,
$$
\|I_{u^{\#},w^{\#},x_0}\| _{L_p(\mathbb{A}_{x_0})\rightarrow
L_p(\mathbb{A}_{x_0})} \asymp \mathfrak{S}_{{\cal A}_{\hat
\xi},u,w}.
$$
It follows from Theorem \ref{cr_har} that
$$
\mathfrak{S}_{{\cal A}_{\hat \xi},u,w}\asymp \sup _{\mathbb{D} \in
{\cal J}_{x_0}} \frac{\|w^{\#}\chi_{\mathbb{A}_{x_0} \backslash
\mathbb{D}}\|_{L_p(\mathbb{A}_{x_0})}}{\alpha_{\mathbb{D}}},
$$
with
$$
\alpha_{\mathbb{D}}=\inf \left\{\|\phi\|_{L_p(\mathbb{A}_{x_0})}:
\, \int \limits _{x_0}^t |\phi(x)|u^{\#}(x)\, dx=1\quad \forall
t\in \partial \mathbb{D}\right\}.
$$
Applying the H\"{o}lder inequality once again (see also
\cite{ev_har_pick}), we obtain that
$$
\alpha_{\mathbb{D}}=\inf
\left\{\|\phi\|_{L_p(\mathbb{A}_{x_0})}:\, \phi \in L_p^{{\rm
discr}}(\mathbb{A}_{x_0}),\, \int \limits _{x_0}^t
|\phi(x)|u^{\#}(x)\, dx=1\quad \forall t\in
\partial \mathbb{D}\right\};
$$
here $L_p^{{\rm discr}}(\mathbb{A})$ is the set of functions
$\phi: \mathbb{A}\rightarrow \R$ that are constants on each edge
of the metric tree $\mathbb{A}$.

Let $\mathbb{D}=({\cal D}, \, \Delta_{\mathbb{D}})\in {\cal
J}_{x_0}$. Set $\mathbb{D}^+=({\cal D}, \, \Delta)$, $\mathbb{D}^-
=(\mathaccent'27 {\cal D}, \, \Delta)$. Prove that ${\cal D}\in
{\cal J}'_{\hat\xi}$. Indeed, let $\xi\in {\bf V}({\cal D})$, and
suppose that there exist vertices $\xi'\in {\bf V}_1(\xi)
\backslash {\bf V}({\cal D})$ and $\xi''\in {\bf V}_1(\xi) \cap
{\bf V}({\cal D})$. Let $\eta$ be a vertex in $\tilde {\cal A}$
that is the direct predecessor of $\xi$. Then the point $(1, \,
(\eta, \, \xi))=(0, \, (\xi, \, \xi'))= (0, \, (\xi, \, \xi''))$
belongs to the boundary of $\mathbb{D}$, as well as it is not
maximal.

We have
$$
\|w^{\#}\chi _{\mathbb{A} _{x_0}\backslash
\mathbb{D}}\|_{L_p(\mathbb{A})} \le \|w^{\#}\chi _{\mathbb{A}
_{x_0}\backslash \mathbb{D}^-}\|_{L_p(\mathbb{A})}= \|w\chi
_{{\cal A}_{\hat \xi}\backslash \mathaccent'27 {\cal
D}}\|_{l_p({\cal A})},
$$
$$
\alpha_{\mathbb{D}}\ge \inf \left\{\|\phi\|_{L_p(\mathbb{A}
_{x_0})}:\, \phi \in L_p^{{\rm discr}}(\mathbb{A}_{x_0}),\, \int
\limits _{x_0}^t |\phi(x)|u^{\#}(x)\, dx=1\quad \forall t\in
\partial \mathbb{D}^+\right\}=\beta_{{\cal D}}.
$$
This implies the upper estimate for $\mathfrak{S}_{{\cal A}
_{\hat\xi},u,w}$. Prove the lower estimate. Notice that if
$\mathbb{D}=\mathbb{D}^+$, then $\alpha _{\mathbb{D}}=
\beta_{{\cal D}}$. If in addition ${\bf V}_{\max}({\cal D})\cap
{\bf V}_{\max}({\cal A})=\varnothing$, then
$$
\|w^{\#}\chi _{\mathbb{A}_{x_0}\backslash \mathbb{D}} \|
_{L_p(\mathbb{A}_{x_0})} = \|w\chi _{{\cal A} _{\hat\xi}\backslash
{\cal D}}\|_{l_p({\cal A}_{\hat\xi})} \stackrel {(\ref{ccc}),
(\ref{uuu})}{\underset{p,\hat C}{\asymp}} \|w\chi _{{\cal A}
_{\hat\xi}\backslash \mathaccent'27 {\cal D}}\|_{l_p({\cal A}
_{\hat\xi})}.
$$
Hence,
$$
\mathfrak{S}_{{\cal A} _{\hat\xi},u,w} \underset{p, \hat
C}{\gtrsim} \sup\left\{ \frac{\|w\chi _{{\cal A}
_{\hat\xi}\backslash \mathaccent'27 {\cal D}}\|_{l_p({\cal A}
_{\hat\xi})}}{\beta_{{\cal D}}}:\; {\cal D}\in {\cal J}'_{\hat
\xi}, \; {\bf V}_{\max}({\cal D})\cap {\bf V}_{\max}({\cal
A})=\varnothing\right\}=:\Sigma.
$$
Prove that
$$
\Sigma\underset{p,\hat C}{\asymp}\sup\left\{ \frac{\|w\chi _{{\cal
A} _{\hat\xi}\backslash \mathaccent'27 {\cal D}}\|_{l_p({\cal A}
_{\hat\xi})}}{\beta_{{\cal D}}}:\; {\cal D}\in {\cal J}'_{\hat
\xi}\right\}.
$$
To this end, it is sufficient to show that if ${\bf V}({\cal
D})\ne \{\hat \xi\}$, then
$$
\frac{\|w\chi _{{\cal A} _{\hat\xi}\backslash \mathaccent'27 {\cal
D}}\|_{l_p({\cal A} _{\hat\xi})}}{\beta_{{\cal
D}}}\underset{p,\hat C}{\lesssim} \Sigma.
$$
Indeed, set ${\cal D}_1=\mathaccent'27 {\cal D}$. Then from
(\ref{betad}), (\ref{ccc}) and (\ref{uuu}) it follows that
$\|w\chi _{{\cal A} _{\hat\xi}\backslash \mathaccent'27 {\cal
D}}\|_{l_p({\cal A} _{\hat\xi})}\underset{p, \hat C}{\asymp}
\|w\chi _{{\cal A} _{\hat\xi}\backslash \mathaccent'27 {\cal
D}_1}\|_{l_p({\cal A} _{\hat\xi})}$ and $\beta_{{\cal D}}
\underset{p, \hat C}{\asymp} \beta_{{\cal D}_1}$. It remains to
observe that ${\cal D}_1\in {\cal J}'_{\hat \xi}$ and ${\bf
V}_{\max}({\cal D}_1)\cap {\bf V}_{\max}({\cal A})=\varnothing$.
\end{proof}

\begin{Sta}
\label{cor2} Let $\xi_*\in {\bf V}({\cal A})$, ${\bf
V}_1(\xi_*)=\{\xi_1, \, \dots, \, \xi _m\}$, ${\cal D}_j\in {\cal
J}'_{\xi_j}$, $1\le j\le m$, ${\cal D}=\{\xi_*\}\cup {\cal
D}_1\cup\dots \cup {\cal D}_m$. Then
\begin{align}
\label{bjoin} \beta _{{\cal D}}^{-1}= \left\|\left(\beta^{-1}
_{\{\xi_*\}}, \, \left\|(\beta_{{\cal D}_j}) _{j=1} ^m \right\|
^{-1} _{l_p^m} \right) \right\|_{l^2_{p'}}.
\end{align}
\end{Sta}
This assertion follows from Theorem \ref{rekurs}.

\subsection{The reduction lemma}

Let $\psi:\R_+\rightarrow \R_+$ be an increasing function,
$\psi(0)=0$, let $({\cal A}, \, \xi_0)$ be a tree with a finite
vertex set. In addition, suppose that there exists $C_*\ge 1$ such
that for any $j_0$, $j\in \Z_+$, $j\ge j_0$, $\xi\in {\bf
V}_{j_0}(\xi_0)$
\begin{align}
\label{vjj0} {\rm card}\, {\bf V}^{\cal A}_{j-j_0}(\xi)\le
C_*2^{\psi(j)-\psi(j_0)}.
\end{align}

Let $u:{\bf V}({\cal A})\rightarrow (0, \, +\infty)$, $u(\xi)
=u_j$ for $\xi \in {\bf V}^{{\cal A}}_j(\xi_0)$. Suppose that
there is $\sigma \in (0, \, 1)$ such that for any $j\in \N$
\begin{align}
\label{sig_uj} \frac{u_j 2^{-\frac{\psi(j)}{p}}}{u_{j-1}
2^{-\frac{\psi(j-1)}{p}}} \ge \sigma ^{-\frac{1}{p'}}.
\end{align}

For each $\xi_*\in {\bf V}({\cal A})$ and for any subtree ${\cal
D}\in {\cal J}'_{\xi_*}$ we define the quantity $\beta_{{\cal D}}$
by (\ref{betad}). Then $\beta_{\{\xi\}}\stackrel
{(\ref{odno_toc})}{=}u_j^{-1}$ for $\xi\in {\bf V}_j(\xi_0)$, and
if ${\cal D}\ne \{\xi_*\}$, then (\ref{bjoin}) holds.

We set
$$
B_{{\cal D}}=\frac{1}{\beta_{{\cal D}}}, \;\;\;\; S_{{\cal D}}=
\left( \sum \limits _{\xi \in {\bf V}_{\max}({\cal D})}
u^{-p}(\xi)\right)^{-\frac 1p}.
$$

Let ${\cal D}\in {\cal J}'_{\xi_*}$, $\hat\xi\in {\bf V}({\cal
D})$, ${\bf V}_1(\hat\xi)=\{\xi_1, \, \dots, \, \xi_{m_1}\}$. Then
$$
{\cal D}_{\hat\xi} =\{\hat\xi\} \cup {\cal D}_{\xi_1} \cup \dots
\cup {\cal D}_{\xi_{m_1}}, \;\; {\cal D}_{\xi_j}\in {\cal
J}'_{\xi_j}, \;\; 1\le j\le m_1.
$$

Let $\varepsilon >0$, $1\le i\le m_1$. A vertex $\xi_i$ is said to
be $(\varepsilon, \, {\cal D})$-regular if
\begin{align}
\label{prav} B_{{\cal D}_{\xi_i}}^{-p}\ge \varepsilon
\frac{B^{-p}_{{\cal D}_{\xi_1}}+\dots+B^{-p}_{{\cal
D}_{\xi_{m_1}}}}{m_1}.
\end{align}
Notice that if $\varepsilon <1$, then at least one of the vertices
$\xi_i$ is $(\varepsilon, {\cal D})$-regular. A path $(\eta_0, \,
\dots, \, \eta_l)$ in ${\cal D}$ is said to be $(\varepsilon, \,
{\cal D})$-regular if $\eta_0<\eta_1<\dots<\eta_l$ and for any
$1\le j\le l$ the vertex $\eta_j$ is $(\varepsilon, \, {\cal
D})$-regular.
\begin{Lem}
\label{reduct} There exists $\hat\sigma=\hat\sigma(p, \, C_*)>0$
such that if (\ref{sig_uj}) holds with $\sigma \in (0, \,
\hat\sigma)$, then for any $\xi_*\in {\bf V}({\cal A})$ and for
any subtree ${\cal D} \in {\cal J}'_{\xi_*}$
\begin{align}
\label{sdasbd} S_{{\cal D}}\le B_{{\cal D}}\le 2S_{{\cal D}}.
\end{align}
\end{Lem}
\begin{proof}
Let
$$
\nu_{{\cal D}}=\max \{j\in \Z_+:{\bf V}_j(\xi_*)\ne \varnothing\}.
$$
If $\nu_{{\cal D}}=0$, then it follows from the definition that
$S_{{\cal D}}= B_{{\cal D}}$. Let us prove the assertion for
$\nu_{{\cal D}}>0$. In this case,
\begin{align}
\label{dxi} {\cal D}=\{\xi_*\}\cup {\cal D}_1\cup \dots \cup{\cal
D}_{m_1}, \;\;{\cal D}_j\in {\cal J}'_{\xi_j}, \;\; \xi_j\in {\bf
V}_1(\xi_*).
\end{align}
Notice that
\begin{align}
\label{sdp} S_{{\cal D}}^{-p}=\sum \limits _{i=1}^{m_1} S_{{\cal
D}_i}^{-p}.
\end{align}

Prove the first inequality, i.e.,
\begin{align}
\label{sdbd} S_{{\cal D}}\le B_{{\cal D}}.
\end{align}
Let $\nu\in \Z_+$, and let the assertion be proved for any ${\cal
D}$ such that $\nu_{{\cal D}}\le \nu$. Prove the assertion for
$\nu_{{\cal D}}=\nu+1$. From (\ref{bjoin}) and the induction
assumption it follows that
$$
B_{{\cal D}}^{p'}=B_{\{\xi_*\}}^{p'}+\left(\sum \limits
_{j=1}^{m_1} B_{{\cal D}_j}^{-p}\right)^{-\frac{p'}{p}}\ge
\left(\sum \limits _{j=1}^{m_1} S_{{\cal D}_j}^{-p}\right)
^{-\frac{p'}{p}}\stackrel{(\ref{sdp})}{=} S_{{\cal D}}^{p'}.
$$

Prove the second inequality. It is sufficient to check that
\begin{align}
\label{bdle888} B_{{\cal D}}\le \left(\prod _{j=1}^\infty (1+
\sigma ^{j/2})^{\frac{2}{p'}} \right)S_{{\cal D}}
\end{align}
holds for  $\sigma\in (0, \, \hat\sigma(p, \, C_*))$.

Let $\varepsilon \in (0, \, 1)$ (it will be chosen later). Then
the end of any $(\varepsilon, \, {\cal D})$-regular path that has
a maximal length and starts from $\xi_*$ is a maximal vertex in
${\cal D}$ (otherwise one of its direct successors is
$(\varepsilon, \, {\cal D})$-regular). Denote by $l_{{\cal D}}$
the maximal length of $(\varepsilon, \, {\cal D})$-regular paths
that start in $\xi_*$. We show by induction on $\nu_{{\cal D}}$
that for $\sigma\in (0, \, \hat\sigma(p, \, C_*))$
\begin{align}
\label{bdle} B_{{\cal D}}\le \left(\prod _{j=1}^{l_{{\cal D}}}(1+
\sigma ^{j/2})^{\frac{2}{p'}} \right)S_{{\cal D}}.
\end{align}
This implies (\ref{bdle888}).

If $\nu_{{\cal D}}=0$, then ${\cal D}$ is a single vertex.
Therefore, $B_{{\cal D}}=S_{{\cal D}}$ and (\ref{bdle}) is true.

Let $\nu_{{\cal D}}>0$. Then (\ref{dxi}) holds, and by
(\ref{vjj0})
$$
m_1\le C_*2^{\psi(j_0+1)-\psi(j_0)}.
$$

Let $l_i= l_{{\cal D}_i}$. Denote by $I_1$ the set of $i\in \{1,
\, \dots, \, m_1\}$ such that $\xi_i$ is $(\varepsilon, \, {\cal
D})$-regular, $I_2= \{1, \, \dots, \, m_1\}\backslash I_1$. Set
$$\hat l=\max \{l_i:\, i\in I_1\}+1.$$ Then $\hat l=l_{{\cal
D}}$.

Prove that there exists $\sigma_*=\sigma_*(\varepsilon, \, C_*, \,
p)>0$ such that for any $\sigma\in (0, \, \sigma_*)$
\begin{align}
\label{bw0} \beta_{\{\xi_*\}}^{-p'}\le \sigma^{\frac{\hat l}{2}}
\left( \sum \limits _{i=1} ^{m_1} B _{{\cal D}_i} ^{-p}
\right)^{-\frac{p'}{p}}.
\end{align}
Suppose the converse, i.e.,
\begin{align}
\label{bw1} \beta_{\{\xi_*\}}^{-p'}> \sigma^{\frac{\hat l}{2}}
\left( \sum \limits _{i=1} ^{m_1} B _{{\cal D}_i} ^{-p}
\right)^{-\frac{p'}{p}}.
\end{align}
Let $\xi_*\in {\bf V}_{j_0}(\xi_0)$. Then
$$
\beta_{\{\xi_*\}}^{-1} =u_{j_0} =2^{\frac{\psi(j_0)}{p}} \cdot
\left(\prod _{j=1} ^{\hat l} \frac{u_{j_0+j-1} 2^{-\frac{
\psi(j_0+j-1)}{p}}} {u_{j_0+j} 2^{-\frac{\psi(j_0+j)}{p}}}
\right)\cdot u_{j_0+\hat l} 2^{-\frac{\psi(j_0+\hat l)}{p}}
\stackrel{(\ref{sig_uj})}{\le}
$$
$$
\le u_{j_0+\hat l}2^{-\frac{ \psi(j_0+\hat l)}{p}
+\frac{\psi(j_0)}{p}}\cdot \sigma ^{\frac{\hat l}{p'}}.
$$
This together with (\ref{bw1}) yields
$$
\left( \sum \limits _{i=1} ^{m_1} B _{{\cal D}_i} ^{-p}
\right)^{-\frac{p'}{p}} <\sigma ^{\frac{\hat l}{2}} u_{j_0+\hat
l}^{p'}2^{-\frac{ p'\psi(j_0+\hat l)}{p} +\frac{p'\psi(j_0)}{p}},
$$
i.e.,
\begin{align}
\label{b_i_sigma} \sum \limits _{i=1} ^{m_1} B _{{\cal D}_i} ^{-p}
>\sigma ^{-\frac{p\hat l}{2p'}} 2^{\psi(j_0+\hat l)-\psi(j_0)}
u_{j_0+\hat l} ^{-p}.
\end{align}
Let $(\xi_*, \, \eta_1, \, \dots, \, \eta_{\hat l})$ be an
$(\varepsilon, \, {\cal D})$-regular path in ${\cal D}$. Then
\begin{align}
\label{bdx1} B_{{\cal D}_{\eta_1}}^{-p} \ge
\frac{\varepsilon}{m_1} \sum \limits _{i=1} ^{m_1} B _{{\cal D}_i}
^{-p}>\frac{\varepsilon}{m_1} \sigma ^{-\frac{p\hat l}{2p'}}
2^{\psi(j_0+\hat l)-\psi(j_0)} u_{j_0+\hat l} ^{-p}.
\end{align}
Let $2\le j\le \hat l$, ${\cal D}_{\eta_{j-1}} =\{\eta_{j-1}\}
\cup {\cal D}_{j,1}\cup \dots \cup {\cal D}_{j,m_j}$. Then
$$
B_{{\cal D}_{\eta_{j-1}}}^{p'} \stackrel{(\ref{bjoin})}{=}
B_{\{\eta_{j-1}\}}^{p'}+\left( \sum \limits _{i=1}^{m_j}
B^{-p}_{{\cal D}_{j,i}} \right)^{-\frac{p'}{p}}\ge \left( \sum
\limits _{i=1}^{m_j} B^{-p}_{{\cal D}_{j,i}}
\right)^{-\frac{p'}{p}},
$$
i.e., $B_{{\cal D}_{\eta_{j-1}}}^{-p}\le \sum \limits _{i=1}^{m_j}
B^{-p}_{{\cal D}_{j,i}}$. Since the vertex $\eta_j$ is
$(\varepsilon, \, {\cal D})$-regular, we have
$$
B^{-p}_{{\cal D} _{\eta_j}}\ge \frac{\varepsilon}{m_j} \sum
\limits _{i=1}^{m_j} B^{-p}_{{\cal D}_{j,i}} \ge
\frac{\varepsilon}{m_j} B^{-p} _{{\cal D}_{\eta_{j-1}}}.
$$
Therefore,
$$
B^{-p}_{{\cal D}_{\eta_{\hat l}}} \ge \frac{\varepsilon}{m_{\hat
l}} B^{-p}_{{\cal D}_{\eta_{\hat l-1}}} \ge \frac{\varepsilon
^2}{m_{\hat l-1}m_{\hat l}}B^{-p}_{{\cal D}_{\eta_{\hat l-2}}} \ge
\dots \ge \frac{\varepsilon ^{\hat l-1}}{m_2\dots m_{\hat l}}
B^{-p} _{{\cal D}_{\eta_1}} \stackrel{(\ref{bdx1})}{\ge}
$$
$$
\ge \frac{\varepsilon ^{\hat l}}{m_1m_2\dots m_{\hat l}}
2^{\psi(j_0+\hat l)-\psi(j_0)} u^{-p}_{j_0+\hat l}\sigma
^{-\frac{p\hat l}{2p'}}.
$$
The vertex $\eta_{\hat l}$ is maximal in ${\cal D}$. Hence,
$B_{{\cal D}_{\eta_{\hat l}}}=B_{\{\eta_{\hat l}\}}=u_{j_0+\hat
l}$. In addition,
$$
m_1m_2\dots m_{\hat l}\stackrel{(\ref{vjj0})}{\le} C_*^{\hat l}
\prod _{i=0}^{\hat l-1}2^{\psi(j_0+i+1)-\psi(j_0+i)}=C_*^{\hat
l}2^{\psi(j_0+\hat l)-\psi(j_0)}.
$$
Thus,
$$
u_{j_0+\hat l}^{-p}\ge \frac{\varepsilon ^{\hat l}}{C_*^{\hat
l}2^{\psi(j_0+\hat l)-\psi(j_0)}}2^{\psi(j_0+\hat l)-\psi(j_0)}
u^{-p}_{j_0+\hat l}\sigma ^{-\frac{p\hat l}{2p'}},
$$
i.e., $1\ge \varepsilon^{\hat l}C_*^{-\hat l}\sigma ^{-\frac{p\hat
l}{2p'}}$, or $\sigma ^{\frac{p}{2p'}}\ge \varepsilon C_*^{-1}$.
For $0<\sigma\le \frac{(C_*^{-1}\varepsilon)^{\frac{2p'}{p}}}{2}$
we get the contradiction. This proves (\ref{bw0}).

Now let us prove (\ref{bdle}). We have
\begin{align}
\label{bbi} B_{{\cal D}}^{p'} =\beta^{-p'}_{\{\xi_*\}}+ \left(\sum
\limits _{i=1}^{m_1} B_{{\cal D}_i} ^{-p}\right)^{-\frac{p'}{p}}
\stackrel{(\ref{bw0})}{\le} \left(1+\sigma ^{\hat
l/2}\right)\left(\sum \limits _{i=1}^{m_1} B_{{\cal D}_i}
^{-p}\right)^{-\frac{p'}{p}}.
\end{align}

Show that there exists $\varepsilon_*=\varepsilon_*(p)\in (0, \,
1)$ such that for any $\varepsilon \in (0, \, \varepsilon_*)$,
$0<\sigma <\min \left(\frac 12, \, \sigma_*(\varepsilon, \, C_*,
\, p)\right)$
\begin{align}
\label{as_d} \left(B_{{\cal D}_1}^{-p}+\dots +B_{{\cal
D}_{m_1}}^{-p}\right)^{-\frac 1p}\le \left(S_{{\cal D}_1}
^{-p}+\dots +S_{{\cal D}_{m_1}}^{-p}\right)^{-\frac 1p}
\left(\prod _{j=1}^{\hat l-1}(1+ \sigma ^{j/2})^{\frac{2}{p'}}
\right) \cdot (1+\sigma ^{\hat l/2})^{\frac{1}{p'}}.
\end{align}
Then (\ref{bbi}), (\ref{sdp}) and (\ref{as_d}) yield (\ref{bdle}).

The relation (\ref{as_d}) is equivalent to
\begin{align}
\label{asd1} \sum \limits _{i=1}^{m_1} B_{{\cal D}_i}^{-p} -
\frac{\sum \limits _{i=1}^{m_1}S_{{\cal D}_i}^{-p}}
{(1+\sigma^{\hat l/2})^{\frac{p}{p'}}\prod _{j=1}^{\hat
l-1}(1+\sigma^{j/2})^{\frac{2p}{p'}}}\ge 0.
\end{align}
Consider separately sums in $i\in I_1$ and in $i\in I_2$. Let
$l=\max _{1\le i\le m_1} l_i+1$. By the induction hypotheses,
\begin{align}
\label{bdi_si} B_{{\cal D}_i}\le \left(\prod _{j=1}^{l_i}(1+
\sigma ^{j/2})^{\frac{2}{p'}} \right)S_{{\cal D}_i}\le \left(\prod
_{j=1}^{\hat l-1}(1+ \sigma ^{j/2})^{\frac{2}{p'}} \right)
S_{{\cal D}_i}, \;\;i\in I_1,
\end{align}
\begin{align}
\label{bdi_si1} B_{{\cal D}_i}\le \left(\prod _{j=1}^{l_i}(1+
\sigma ^{j/2})^{\frac{2}{p'}} \right)S_{{\cal D}_i}\le \left(\prod
_{j=1}^{l-1}(1+ \sigma ^{j/2})^{\frac{2}{p'}} \right) S_{{\cal
D}_i}, \;\;i\in I_2.
\end{align}
Hence,
$$
\sum \limits _{i\in I_1} B_{{\cal D}_i}^{-p} -\frac{\sum \limits
_{i\in I_1}S_{{\cal D}_i}^{-p}}{(1+\sigma^{\hat
l/2})^{\frac{p}{p'}}\prod _{j=1}^{\hat
l-1}(1+\sigma^{j/2})^{\frac{2p}{p'}}}\stackrel{(\ref{bdi_si})}{\ge}
$$
$$
\ge \frac{\sum \limits _{i\in I_1} S_{{\cal D}_i}^{-p}}{\prod
_{j=1}^{\hat l-1}(1+ \sigma ^{j/2})^{\frac{2p}{p'}}}-\frac{\sum
\limits _{i\in I_1}S_{{\cal D}_i}^{-p}}{(1+\sigma^{\hat
l/2})^{\frac{p}{p'}}\prod _{j=1}^{\hat
l-1}(1+\sigma^{j/2})^{\frac{2p}{p'}}}=
$$
$$
=\left(\sum \limits _{i\in I_1}S_{{\cal D}_i}^{-p}\right) \frac{
(1+\sigma^{\hat l/2})^{\frac{p}{p'}}-1 } { (1+\sigma^{\hat
l/2})^{\frac{p}{p'}}\prod _{j=1}^{\hat l-1} (1+\sigma^{j/2})
^{\frac{2p}{p'}} }\underset{p}{\gtrsim} \left(\sum \limits _{i\in
I_1}S_{{\cal D}_i}^{-p}\right)\sigma^{\hat
l/2}\stackrel{(\ref{sdbd})}{\ge} \left(\sum \limits _{i\in
I_1}B_{{\cal D}_i}^{-p}\right)\sigma^{\hat l/2}
$$
(the penultimate relation holds for $0<\sigma<\frac 12$).
Therefore, there exists $C_1(p)>0$ such that
\begin{align}
\label{ii1} \sum \limits _{i\in I_1} B_{{\cal D}_i}^{-p}
-\frac{\sum \limits _{i\in I_1}S_{{\cal
D}_i}^{-p}}{(1+\sigma^{\hat l/2})^{\frac{p}{p'}}\prod _{j=1}^{\hat
l-1}(1+\sigma^{j/2})^{\frac{2p}{p'}}}\ge C_1(p) \left(\sum \limits
_{i\in I_1}B_{{\cal D}_i}^{-p}\right)\sigma^{\hat l/2}.
\end{align}
If $l=\hat l$, then the sum in $i\in I_2$ is estimated similarly.
In this case, (\ref{asd1}) is proved. Let $l\ge \hat l+1$. Then we
have for $0<\sigma<\min \left(\frac 12, \, \sigma_*(\varepsilon,
\, C_*, \, p)\right)$
$$
\sum \limits _{i\in I_2} B_{{\cal D}_i}^{-p} -\frac{\sum \limits
_{i\in I_2}S_{{\cal D}_i}^{-p}}{(1+\sigma^{\hat
l/2})^{\frac{p}{p'}}\prod _{j=1}^{\hat
l-1}(1+\sigma^{j/2})^{\frac{2p}{p'}}}\stackrel{(\ref{bdi_si1})}{\ge}
$$
$$
\ge \frac{\sum \limits _{i\in I_2}S_{{\cal D}_i}^{-p}}{\prod
_{j=1}^{l-1}(1+ \sigma ^{j/2})^{\frac{2p}{p'}}}- \frac{\sum
\limits _{i\in I_2}S_{{\cal D}_i}^{-p}}{(1+\sigma^{\hat
l/2})^{\frac{p}{p'}}\prod _{j=1}^{\hat
l-1}(1+\sigma^{j/2})^{\frac{2p}{p'}}}\ge
$$
$$
\ge \left(\sum \limits _{i\in I_2}S_{{\cal D}_i}^{-p}\right)
\frac{1-(1+\sigma^{\hat l/2})^{\frac{p}{p'}}\prod _{j=\hat
l+1}^{l-1}(1+\sigma^{j/2})^{\frac{2p}{p'}}}{\prod _{j=1}^{l-1}(1+
\sigma ^{j/2})^{\frac{2p}{p'}}}\ge
$$
$$
\ge -C_2(p)\left(\sum \limits _{i\in I_2}S_{{\cal
D}_i}^{-p}\right) \sigma^{\hat l/2}\stackrel{(\ref{bdi_si1})}{\ge}
-C_3(p)\left(\sum \limits _{i\in I_2}B_{{\cal
D}_i}^{-p}\right)\sigma^{\hat l/2},
$$
where $C_2(p)>0$, $C_3(p)>0$. Thus,
\begin{align}
\label{ii2} \sum \limits _{i\in I_2} B_{{\cal D}_i}^{-p}
-\frac{\sum \limits _{i\in I_1}S_{{\cal
D}_i}^{-p}}{(1+\sigma^{\hat l/2})^{\frac{p}{p'}}\prod _{j=1}^{\hat
l-1}(1+\sigma^{j/2})^{\frac{2p}{p'}}}\ge -C_3(p) \left(\sum
\limits _{i\in I_2}B_{{\cal D}_i}^{-p}\right)\sigma^{\hat l/2}.
\end{align}

From definitions of $I_1$ and $I_2$ we get
$$
\sum \limits _{i\in I_2} B^{-p}_{{\cal D}_i} \le \sum \limits
_{i\in I_2}\frac{\varepsilon}{m_1} \sum \limits _{j=1}^{m_1}
B^{-p}_{{\cal D}_j}\le \varepsilon \sum \limits _{j=1}^{m_1}
B^{-p}_{{\cal D}_j},
$$
$$
\sum \limits _{i\in I_1}B^{-p}_{{\cal D}_i}=\sum \limits
_{i=1}^{m_1} B^{-p}_{{\cal D}_i}-\sum \limits _{i\in I_2}
B^{-p}_{{\cal D}_i}\ge (1-\varepsilon) \sum \limits _{i=1}^{m_1}
B^{-p}_{{\cal D}_i}.
$$

This together with (\ref{ii1}) and (\ref{ii2}) implies that
$$
\sum \limits _{i=1}^{m_1} B_{{\cal D}_i}^{-p} - \frac{\sum \limits
_{i=1}^{m_1}S_{{\cal D}_i}^{-p}} {(1+\sigma^{\hat
l/2})^{\frac{p}{p'}}\prod _{j=1}^{\hat
l-1}(1+\sigma^{j/2})^{\frac{2p}{p'}}}\ge
$$
$$
\ge C_1(p)\left(\sum \limits _{i\in I_1}B_{{\cal
D}_i}^{-p}\right)\sigma^{\hat l/2} -C_3(p) \left(\sum \limits
_{i\in I_2}B_{{\cal D}_i}^{-p}\right)\sigma^{\hat l/2}\ge
$$
$$
\ge \sigma^{\hat l/2} \left((1-\varepsilon) C_1(p)-\varepsilon
C_3(p)\right) \sum \limits _{i=1}^{m_1} B^{-p}_{{\cal D}_i} >0
$$
for sufficiently small $\varepsilon$. This completes the proof of
(\ref{asd1}).
\end{proof}

Let $w:{\bf V}({\cal A})\rightarrow (0, \, \infty)$, $w(\xi)=w_j$
for $\xi\in {\bf V}_j^{\cal A}(\xi_0)$. Suppose that there exists
$\sigma\in \left(0, \, \frac 12\right)$ such that for any $j\in
\N$
\begin{align}
\label{wjps} \frac{w_j \cdot 2^{\frac{\psi(j)}{p}}} {w_{j-1}\cdot
2^{\frac{\psi(j-1)}{p}}}\le \sigma^{\frac 1p}.
\end{align}
Given ${\cal D}\in {\cal J}'_{\xi_*}$, we denote
\begin{align}
\label{rd_def_qd} R_{{\cal D}}=\left(\sum \limits _{\xi\in {\bf
V}_{\max} ({\cal D})}\sum \limits _{\xi'\ge \xi}
w^p(\xi')\right)^{1/p}, \;\;\;\; Q_{{\cal D}}= \left(\sum \limits
_{\xi\in {\bf V}_{\max}({\cal D})} w^p(\xi)\right)^{1/p}.
\end{align}
From (\ref{wjps}) and (\ref{vjj0}) it follows that there exists
$\sigma_*=\sigma_*(p, \, C_*)>0$ such that for any
$0<\sigma<\sigma_*$
\begin{align}
\label{rdqd} Q_{{\cal D}}\le R_{{\cal D}}\le 2Q_{{\cal D}}.
\end{align}

Construct the function $\psi_*$ by induction as follows:
\begin{align}
\label{psi_st} \psi_*(0)=0,\quad
2^{\psi_*(j)-\psi_*(j-1)}=\left[2^{\psi(j)-\psi_*(j-1)}\right],
\quad j\in \N.
\end{align}
Then $2^{\psi_*(j)-\psi_*(j-1)}\in \N$ and
\begin{align}
\label{2p_j} 2^{\psi_*(j)}\le 2^{\psi(j)}\le 2^{\psi_*(j)+1}.
\end{align}

\label{symha}Let $\xi_*\in {\bf V}^{\cal A}_{j_0} (\xi_0)$, and
let $(\hat{\cal A}, \, \hat{\xi})$ be a tree such that
\begin{align}
\label{def_hata} {\rm card}\, {\bf V}_1^{\hat{\cal
A}}(\xi)=2^{\psi_*(j+1)-\psi_*(j)}, \quad \xi\in {\bf
V}_{j-j_0}^{\hat{\cal A}}(\hat \xi), \quad j\ge j_0.
\end{align}

\begin{Lem}
\label{symmetr} Let ${\cal D}\subset {\cal A}$ be a tree rooted at
$\xi_*\in {\bf V}^{\cal A}_{j_0} (\xi_0)$. Then there exists
$\sigma_0=\sigma_0(p, \, C_*)>0$ satisfying the following
property: if (\ref{sig_uj}) and (\ref{wjps}) hold for some
$\sigma\in (0, \, \sigma_0)$, then there exists a tree $\hat {\cal
D}\subset \hat{\cal A}$ rooted at $\hat\xi$ such that $S_{{\cal
D}}\underset{p}{\lesssim} S_{\hat{\cal D}}$ and $Q_{{\cal
D}}\underset{p,C_*}{\lesssim} Q_{\hat{\cal D}}$.
\end{Lem}
\begin{proof}
Set
$$
\{j_1, \, \dots, \, j_s\}=\{j\in \N:\; {\bf V}_{\max}({\cal
D})\cap {\bf V}_{j-j_0}(\xi_*)\ne \varnothing\}, \;\;
j_1<\dots<j_s.
$$
For each $1\le l\le s$, we denote ${\bf V}_{l,{\cal D}}={\bf
V}_{\max}({\cal D})\cap {\bf V}^{\cal A}_{j_l-j_0}(\xi_*)$,
$$
{\bf U}_s={\bf V}_{j_{s-1}-j_0}^{\cal A}(\xi_*)\cap {\bf V}({\cal
D})\backslash {\bf V}_{\max}({\cal D}).
$$
Then
\begin{align}
\label{xi_u_s} {\bf V}_{s,{\cal D}}\subset \cup _{\xi\in {\bf
U}_s}{\bf V}^{\cal A}_{j_s-j_{s-1}}(\xi).
\end{align}
By (\ref{vjj0}) and (\ref{wjps}), there exists
$\sigma_1=\sigma_1(p, \, C_*)$ such that for any $\sigma \in (0,
\, \sigma_1)$, $1\le \nu\le s$
\begin{align}
\label{ujs} w_{j_\nu}^{p}{\rm card}\, {\bf V}_{j_\nu-j_{\nu-1}}^{\cal A}(\xi)
\le w_{j_{\nu-1}}^{p}.
\end{align}
Show that for any $\sigma \in (0, \, \sigma_1)$
\begin{align}
\label{sll23333}\sum \limits _{t=\nu}^s w_{j_t}^{p} {\rm card}\,
{\bf V}_{t,{\cal D}} \le w_{j_\nu}^{p}\cdot {\rm card}\,{\bf
V}^{\cal A}_{j_\nu-j_0}(\xi_*).
\end{align}
We use induction on $s-\nu$. If $s-\nu=0$, then the inequality is
trivial. Let $s-\nu\ge 1$. Denote by $\tilde{\cal D}$ the subtree
in ${\cal A}_{\xi_*}$ with the set of maximal vertices $\left(\cup
_{t=1}^{s-1} {\bf V}_{t,{\cal D}} \right)\cup {\bf U}_s$ and the
root $\xi_*$. Then
$$
\sum \limits _{t=\nu}^s w_{j_t}^{p} {\rm card}\, {\bf V}_{t,{\cal
D}} \stackrel{(\ref{xi_u_s}),(\ref{ujs})}{\le} \sum \limits
_{t=\nu}^{s-1}w_{j_t}^{p} {\rm card}\, {\bf V}_{t,{\cal D}} +
w_{j_{s-1}}^{p}{\rm card}\, {\bf U}_s=
$$
$$
=\sum \limits _{t=\nu}^{s-1} w_{j_t}^{p} {\rm card}\, {\bf
V}_{t,\tilde{\cal D}} \le w_{j_\nu}^{p}{\rm card}\, {\bf V}^{\cal
A}_{j_\nu-j_0}(\xi_*)
$$
(the last inequality holds by the induction assumption). This
completes the proof of (\ref{sll23333}).

Applying induction on $l$, construct the set ${\bf V}_{l,\hat{\cal
D}}\subset {\bf V}(\hat{\cal A}_{\hat \xi})$ with the following
properties:
\begin{enumerate}
\item if $1\le t<\nu\le l$, then
\begin{align}
\label{zzz} {\bf V}_{\nu,\hat{\cal D}}\cap \left(\cup _{\xi\in
{\bf V}_{t,\hat{\cal D}}} {\bf V}^{\hat{\cal
A}}_{j_\nu-j_t}(\xi)\right) =\varnothing;
\end{align}
\item if
\begin{align}
\label{t1l} \cup _{t=1}^l \cup _{\xi\in {\bf V}_{t,\hat{\cal D}}}
{\bf V}_{j_l-j_t}^{\hat{\cal A}}(\xi)={\bf V}^{\hat{\cal
A}}_{j_l-j_0}(\hat \xi),
\end{align}
then the tree $\hat{\cal D}$ with the set of vertices
$$
{\bf V}(\hat{\cal D})=\cup_{t=1}^l \cup_{\xi\in {\bf V}
_{t,\hat{\cal D}}}[\hat \xi, \, \xi]
$$
satisfies ${\bf V}_{\max}(\hat{\cal D})=\cup_{1\le t\le l} {\bf
V}_{t,\hat{\cal D}}$, $S_{\cal D}\underset{p}{\lesssim}
S_{\hat{\cal D}}$ and $Q_{\cal D}\underset{p,C_*}{\lesssim}
Q_{\hat{\cal D}}$;
\item if
\begin{align}
\label{t1ll} \cup _{t=1}^l \cup _{\xi \in {\bf V}_{t,\hat{\cal
D}}} {\bf V}_{j_l-j_t} ^{\hat{\cal A}}(\xi)\ne {\bf V}^{\hat{\cal
A}}_{j_l-j_0}(\hat \xi),
\end{align}
then  ${\rm card}\,{\bf V}_{t,\hat{\cal D}}={\rm card}\, {\bf
V}_{t,{\cal D}}$ for any $1\le t\le l$.
\end{enumerate}
If (\ref{t1l}) holds for some $l$, then the construction is
interrupted. In this case, $\hat{\cal D}$ is the desired tree. If
(\ref{t1ll}) holds for any $l\le s$, then we take as $\hat{\cal
D}$ the tree with the vertex set $\cup_{1\le t\le s} \cup_{\xi\in
{\bf V}_{t,\hat{\cal D}}} [\hat \xi, \, \xi]$. In this case,
$S_{\hat{\cal D}}=S_{{\cal D}}$ and $Q_{\hat{\cal D}}=Q_{{\cal
D}}$.

{\bf The base of induction.} Let $l=1$. If ${\rm card}\, {\bf
V}_{1,{\cal D}}<\frac{1}{2} 2^{\psi_*(j_1)-\psi_*(j_0)}$, then we
take as ${\bf V}_{1,\hat{\cal D}}$ an arbitrary subset
$E_1\subset{\bf V}_{j_1-j_0}^{\hat{\cal A}}(\hat \xi)$ such that
${\rm card}\, E_1={\rm card}\, {\bf V} _{1,{\cal D}}$. By
(\ref{def_hata}), we have (\ref{t1ll}).

Let ${\rm card}\, {\bf V}_{1,{\cal D}}\ge \frac{1}{2}
2^{\psi_*(j_1)-\psi_*(j_0)}$. Then we set ${\bf V}_{1,\hat{\cal
D}}={\bf V}^{\hat{\cal A}}_{j_1-j_0}(\hat \xi)$ (in this case, (\ref{t1l})
holds). Hence, ${\bf V}(\hat{\cal D})=\cup _{j=0}^{j_1-j_0} {\bf
V}_j^{\hat{\cal A}}(\hat \xi)$, ${\bf V}_{\max}(\hat{\cal D})={\bf
V}_{1,\hat{\cal D}}$ and $$S_{\hat{\cal
D}}^{-p}\stackrel{(\ref{def_hata})}{=}
2^{\psi_*(j_1)-\psi_*(j_0)}u_{j_1}^{-p},\quad Q_{\hat{\cal
D}}^{p}\stackrel{(\ref{def_hata})}{=}
2^{\psi_*(j_1)-\psi_*(j_0)}w_{j_1}^{p}.$$ Further,
$$
S_{{\cal D}}^{-p} \ge {\rm card}\, {\bf V}_{1, {\cal D}}\cdot
u_{j_1}^{-p}\ge \frac{1}{2}2^{\psi_*(j_1)-\psi_*(j_0)}
u_{j_1}^{-p},
$$
which implies $S_{\cal D}\underset{p}{\lesssim} S_{\hat{\cal D}}$.
Prove that $Q_{\cal D}\underset{p, \, C_*}{\lesssim}
Q_{\hat{\cal D}}$. Indeed,
$$
Q_{\cal D}=\sum \limits _{t=1}^s {\rm card}\, {\bf V}_{t, \, {\cal
D}} \cdot w_{j_t}^p\stackrel{(\ref{sll23333})}{\le} w_{j_1}^p
\cdot {\rm card}\, {\bf V}^{\cal A}_{j_1-j_0}(\xi_*)
\stackrel{(\ref{vjj0}),(\ref{2p_j})}{\underset{C_*}{\lesssim}}
w_{j_1}^p \cdot 2^{\psi_*(j_1)-\psi_*(j_0)}=Q_{\hat{\cal D}}^{p}.
$$

{\bf The induction step.} Let $1\le l<s$,
\begin{align}
\label{2c0} \sum \limits _{t=1}^l {\rm card}\, {\bf V}_{t,{\cal
D}} \cdot 2^{\psi_*(j_l)-\psi_*(j_t)}<\frac{1}{2}\cdot
2^{\psi_*(j_l)-\psi_*(j_0)}.
\end{align}
Suppose that there are the sets ${\bf V}_{t,\hat{\cal D}} \subset
{\bf V}_{j_t-j_0}^{\hat{\cal A}}(\hat \xi)$, $1\le t\le l$,
satisfying (\ref{zzz}) and
\begin{align}
\label{cet} {\rm card}\, {\bf V}_{t,\hat{\cal D}}={\rm card}\,
{\bf V}_{t,{\cal D}}, \;\; 1\le t\le l.
\end{align}
Then
$$
\sum \limits _{t=1}^l \sum \limits _{\xi\in {\bf V}_{t,\hat{\cal
D}}} {\rm card}\, {\bf V}^{\hat{\cal A}}_{j_l-j_t}(\xi)
\stackrel{(\ref{def_hata}),(\ref{cet})}{\le}$$$$\le \sum \limits
_{t=1}^l {\rm card}\, {\bf V}_{t,{\cal D}} \cdot
2^{\psi_*(j_l)-\psi_*(j_t)}\stackrel{(\ref{2c0})}{<} \frac 12
2^{\psi_*(j_l)-\psi_*(j_0)}\stackrel{(\ref{def_hata})}{<} {\rm
card}\, {\bf V}^{\hat{\cal A}}_{j_l-j_0}(\hat \xi).
$$
Therefore, properties 1--3 of the sets ${\bf V}_{t,\hat{\cal D}}$
hold (property 2 is trivial, since (\ref{t1ll}) holds instead of
(\ref{t1l}); property 3 follows from (\ref{cet}), property 1
holds since we supposed that the sets satisfy (\ref{zzz})).

Construct the set ${\bf V}_{l+1,\hat{\cal D}}\subset {\bf
V}^{\hat{\cal A}}_{j_{l+1}-j_0}(\hat \xi) \backslash \cup_{t=1}^l
\cup_{\xi \in {\bf V}_{t,\hat{\cal D}}} {\bf V}^{\hat{\cal A}}
_{j_{l+1}-j_t}(\xi)$.

Let
\begin{align}
\label{cvl1} {\rm card}\, {\bf V}_{l+1,{\cal D}} +\sum \limits
_{t=1}^l {\rm card}\, {\bf V}_{t,{\cal D}}\cdot
2^{\psi_*(j_{l+1})-\psi_*(j_t)} < \frac{1}{2}
2^{\psi_*(j_{l+1})-\psi_*(j_0)}.
\end{align}
In this case, we take an arbitrary subset
$$
{\bf V}_{l+1,\hat{\cal D}}\subset{\bf V}_{j_{l+1}-j_0}^{\hat{\cal
A}}(\hat \xi) \backslash \cup_{t=1}^l \cup_{\xi \in {\bf
V}_{t,\hat{\cal D}}} {\bf V} ^{\hat{\cal A}}_{j_{l+1}-j_t}(\xi),
\;\; {\rm card}\, {\bf V}_{l+1,\hat{\cal D}} ={\rm card}\, {\bf
V}_{l+1,{\cal D}}.
$$
This set exists, since
$$
{\rm card}\, {\bf V}_{l+1,{\cal D}}+\sum \limits _{t=1}^l \sum
\limits _{\xi \in {\bf V}_{t,\hat{\cal D}}} {\rm card}\, {\bf
V}_{j_{l+1} -j_t} ^{\hat{\cal A}}(\xi) \stackrel
{(\ref{def_hata})}{=} {\rm card}\, {\bf V}_{l+1,{\cal D}}+ \sum
\limits _{t=1}^l {\rm card}\,{\bf V}_{t,\hat{\cal D}} \cdot
2^{\psi_*(j_{l+1})-\psi_*(j_t)}
$$
$$
\stackrel{(\ref{cet}),(\ref{cvl1})}{<}\frac 12 \cdot
2^{\psi_*(j_{l+1})-\psi_*(j_0)} \stackrel{(\ref{def_hata})}{<}{\rm
card}\, {\bf V}^{\hat{\cal A}} _{j_{l+1}-j_0}(\hat \xi).
$$
Then we have (\ref{zzz}), (\ref{2c0}) and (\ref{cet}) with $l+1$
instead of $l$. Hence, properties 1--3 for the sets $\{{\bf
V}_{t,\hat{\cal D}}\}_{t=1}^{l+1}$ hold.

Let
\begin{align}
\label{l_pl_1} {\rm card}\, {\bf V}_{l+1,{\cal D}} +\sum \limits
_{t=1}^l {\rm card}\, {\bf V}_{t,{\cal D}}\cdot 2^{\psi_*(j_{l+1})
-\psi_*(j_t)} \ge \frac{1}{2} 2^{\psi_*(j_{l+1})-\psi_*(j_0)}.
\end{align}
Then we set
\begin{align}
\label{vlel1} {\bf V}_{l+1,\hat{\cal D}}= {\bf V}^{\hat{\cal
A}}_{j_{l+1}-j_0}(\hat \xi) \backslash \cup_{t=1}^l \cup_{\xi \in
{\bf V}_{t,\hat{\cal D}}} {\bf V}^{\hat{\cal A}}
_{j_{l+1}-j_t}(\xi).
\end{align}
By construction, we have property 1 of the sets $\{{\bf
V}_{t,\hat{\cal D}}\}_{t=1}^{l+1}$ and (\ref{t1l}) (with $l+1$ instead of
$l$); i.e.,
\begin{align}
\label{zzz080} {\bf V}_{\nu,\hat{\cal D}}\cap \left(\cup _{\xi\in
{\bf V}_{t,\hat{\cal D}}} {\bf V}^{\hat{\cal
A}}_{j_\nu-j_t}(\xi)\right) =\varnothing, \quad 1\le t<\nu\le l+1,
\end{align}
$$
\cup _{t=1}^{l+1} \cup _{\xi\in {\bf
V}_{t,\hat{\cal D}}} {\bf V}_{j_{l+1}-j_t}^{\hat{\cal
A}}(\xi)={\bf V}^{\hat{\cal A}}_{j_{l+1}-j_0}(\hat \xi).
$$

Therefore, it is sufficient to check property 2. Define the tree
$\hat{\cal D}$ by
$$
{\bf V}(\hat{\cal D})=\cup_{t=1}^{l+1}\cup _{\xi\in {\bf
V}_{t,\hat{\cal D}}} [\hat \xi, \, \xi].
$$
From (\ref{zzz080}) it follows that
\begin{align}
\label{vmax} {\bf V}_{\max}(\hat{\cal D}) =\cup_{t=1}^{l+1} {\bf
V}_{t,\hat{\cal D}}.
\end{align}

We claim that $S_{{\cal D}}\underset{p}{\lesssim} S_{\hat{\cal
D}}$ and $Q_{{\cal D}}\underset{p,C_*}{\lesssim} Q_{\hat{\cal
D}}$. Indeed,
\begin{align}
\label{sd} S_{\hat{\cal D}}^{-p}
\stackrel{(\ref{cet}),(\ref{vmax})} {=}\sum \limits _{t=1}^l
u_{j_t}^{-p}{\rm card}\, {\bf V}_{t,{\cal D}} + u_{j_{l+1}}^{-p}
{\rm card}\, {\bf V}_{l+1,\hat{\cal D}},
\end{align}
\begin{align}
\label{shatd} S_{{\cal D}}^{-p} \ge\sum \limits _{t=1}^l
u_{j_t}^{-p}{\rm card}\, {\bf V}_{t,{\cal D}} + u_{j_{l+1}}^{-p}
{\rm card}\, {\bf V} _{l+1,{\cal D}},
\end{align}
\begin{align}
\label{qd} Q_{\hat{\cal D}}^{p}
\stackrel{(\ref{cet}),(\ref{vmax})} {=}\sum \limits _{t=1}^l
w_{j_t}^{p}{\rm card}\, {\bf V}_{t,{\cal D}} + w_{j_{l+1}}^{p}
{\rm card}\, {\bf V}_{l+1,\hat{\cal D}},
\end{align}
\begin{align}
\label{qhatd}
\begin{array}{c}
Q_{{\cal D}}^{p} =\sum \limits _{t=1}^l w_{j_t}^{p}{\rm card}\,
{\bf V}_{t,{\cal D}} + \sum \limits _{t=l+1} ^s w_{j_t}^{p} {\rm
card}\, {\bf V} _{t,{\cal D}} \stackrel{(\ref{sll23333})} {\le} \\
\le \sum \limits _{t=1}^l w_{j_t}^{p}{\rm card}\, {\bf V}_{t,{\cal
D}} +w_{j_{l+1}}^p {\rm card}\, {\bf V}^{\cal
A}_{j_{l+1}-j_0}(\xi_*) \stackrel{(\ref{vjj0}),(\ref{2p_j})}
{\underset{C_*}{\lesssim}} \\ \lesssim\sum \limits _{t=1}^l
w_{j_t}^{p}{\rm card}\, {\bf V}_{t,{\cal D}} +w_{j_{l+1}}^p \cdot
2^{\psi_*(j_{l+1})-\psi_*(j_0)}.
\end{array}
\end{align}

In addition,
\begin{align}
\label{el1} {\rm card}\, {\bf V}_{l+1,\hat{\cal D}}\le {\rm
card}\, {\bf V}^{\hat{\cal A}}_{j_{l+1}-j_0}(\hat\xi)
\stackrel{(\ref{def_hata})}{=} 2^{\psi_*(j_{l+1})-\psi_*(j_0)}.
\end{align}

{\it Case 1.} Let
\begin{align}
\label{xi_v_td} \sum \limits _{t=1}^l {\rm card}\, {\bf
V}_{t,{\cal D}}\cdot 2^{\psi_*(j_{l+1})-\psi_*(j_t)}< \frac{1}{4}
2^{\psi_*(j_{l+1})-\psi_*(j_0)}.
\end{align}
Then
\begin{align}
\label{cv1d} {\rm card}\, {\bf V}_{l+1,{\cal D}} \ge
2^{\psi_*(j_{l+1})-\psi_*(j_0) -2}.
\end{align}

Indeed,
$$
{\rm card}\, {\bf V}_{l+1,{\cal D}} \stackrel{(\ref{l_pl_1}),
(\ref{xi_v_td})}{\ge} \frac{1}{2}2^{\psi_*(j_{l+1})-\psi_*(j_0)}
-\frac{1}{4} 2^{\psi_*(j_{l+1})-\psi_*(j_0)}
=\frac{1}{4}2^{\psi_*(j_{l+1})-\psi_*(j_0)}.
$$
From (\ref{sd}), (\ref{shatd}),  (\ref{el1}) and (\ref{cv1d}) it
follows that $S_{{\cal D}}\underset{p}{\lesssim} S_{\hat{\cal
D}}$.

Prove that $Q_{{\cal D}}\underset{p,C_*}{\lesssim} Q_{\hat{\cal D}}$.
By (\ref{qd}) and (\ref{qhatd}), it suffices to check that
${\rm card}\, {\bf V}_{l+1,\hat{\cal D}}\ge 2^{\psi_*(j_{l+1})-\psi_*(j_0)-1}$.
We have
$$
{\rm card}\, {\bf V}_{l+1,\hat{\cal D}}\stackrel{(\ref{def_hata}),
(\ref{cet}),(\ref{vlel1})}{\ge} {\rm card}\, {\bf V}^{\hat{\cal
A}}_{j_{l+1}-j_0}(\hat \xi)- \sum \limits _{t=1}^l {\rm card}\,
{\bf V}_{t,{\cal D}}\cdot
2^{\psi_*(j_{l+1})-\psi_*(j_t)}\stackrel{(\ref{xi_v_td})}{\ge}
$$
$$
=2^{\psi_*(j_{l+1})-\psi_*(j_0)} -\frac 14 \cdot
2^{\psi_*(j_{l+1})-\psi_*(j_0)}\ge
2^{\psi_*(j_{l+1})-\psi_*(j_0)-1}.
$$

{\it Case 2.} Let
\begin{align}
\label{sltl} \sum \limits _{t=1}^l {\rm card}\,  {\bf V}_{t,{\cal
D}}\cdot 2^{\psi_*(j_{l+1})-\psi_*(j_t)}\ge \frac{1}{4}
2^{\psi_*(j_{l+1})-\psi_*(j_0)}.
\end{align}
Then by (\ref{sig_uj}), (\ref{wjps}) and (\ref{2p_j}), there exists
$\sigma'_1=\sigma'_1(p, \, C_*)$ such that for any $\sigma\in (0, \,
\sigma'_1)$
$$
\sum \limits _{t=1}^l u_{j_t}^{-p} {\rm card}\, {\bf V}_{t,{\cal
D}} \ge \sum \limits _{t=1}^l u_{j_{l+1}}^{-p}
2^{\psi_*(j_{l+1})-\psi_*(j_t)}{\rm card}\, {\bf V}_{t,{\cal D}}
\stackrel{(\ref{sltl})}{\ge} \frac{u_{j_{l+1}}^{-p}}{4}
2^{\psi_*(j_{l+1})-\psi_*(j_0)},
$$
$$
\sum \limits _{t=1}^l w_{j_t}^{p} {\rm card}\, {\bf V}_{t,{\cal
D}} \ge \sum \limits _{t=1}^l w_{j_{l+1}}^{p}
2^{\psi_*(j_{l+1})-\psi_*(j_t)}{\rm card}\, {\bf V}_{t,{\cal D}}
\stackrel{(\ref{sltl})}{\ge} \frac{w_{j_{l+1}}^{p}}{4}
2^{\psi_*(j_{l+1})-\psi_*(j_0)}.
$$
This together with (\ref{sd}), (\ref{shatd}), (\ref{qd}), (\ref{qhatd}) and
(\ref{el1}) implies that
$$S_{{\cal D}}^{-p}\ge \sum \limits _{t=1}^l
u_{j_t}^{-p} {\rm card}\, {\bf V}_{t,{\cal D}}
\underset{p}{\asymp} S_{\hat{\cal D}}^{-p},$$
$$Q_{{\cal D}}^{p}\underset{p, C_*}{\asymp} \sum \limits _{t=1}^l
w_{j_t}^{p} {\rm card}\, {\bf V}_{t,{\cal D}}
\underset{p}{\asymp} Q_{\hat{\cal D}}^{p}.$$ This completes the
proof.
\end{proof}

Let (\ref{ccc}) and (\ref{uuu}) hold, let $\hat \sigma$ be such as
in Lemma \ref{reduct}, and let $\sigma_0$ be such as in Lemma
\ref{symmetr}. Take $\sigma\in (0, \, \min\{\hat \sigma, \,
\sigma_0\})$. By Lemma \ref{hardy_cr},
\begin{align}
\label{sauv} \mathfrak{S}_{{\cal A}_{\xi_*},u,w} \underset{p,\hat
C}{\asymp} \sup _{{\cal D}\in {\cal J}'_{\xi_*}} \|w\chi_{{\cal
A}_{\xi_*}\backslash \mathaccent'27{\cal D}}\|_{l_p({\cal
A}_{\xi_*})} B_{{\cal D}} \stackrel{(\ref{rd_def_qd})}{=}\sup
_{{\cal D}\in {\cal J}'_{\xi_*}} R_{{\cal D}}B_{{\cal
D}}\stackrel{(\ref{sdasbd}),(\ref{rdqd})} {\underset{p} {\asymp}}
\sup _{{\cal D}\in {\cal J}'_{\xi_*}} Q_{{\cal D}}S_{{\cal D}}.
\end{align}
\begin{Lem}
\label{hata} Let $\xi_*\in {\bf V}^{\cal A}_{j_0} (\xi_0)$, $\hat
u(\xi)=u_j$, $\hat w(\xi)=w_j$ for any $\xi\in {\bf
V}_{j-j_0}^{\hat{\cal A}} (\hat\xi)$. Then there exists
$\sigma_2=\sigma_2(p, \, C_*)>0$ such that $\mathfrak{S}_{{\cal
A}_{\xi_*},u,w}\underset{p,\hat C,C_*}{\lesssim}
\mathfrak{S}_{\hat{\cal A},\hat u,\hat w}$ for any $\sigma \in (0,
\, \sigma_2)$.
\end{Lem}
\begin{proof}
Suppose that the supremum of the right-hand side in (\ref{sauv})
is attained at the tree ${\cal D}\in {\cal J}'_{\xi_*}$. Apply
Lemma \ref{symmetr} and construct the tree $\hat{\cal D}\subset
\hat{\cal A}$ rooted at $\hat\xi$ such that $S_{{\cal D}}
\underset{p}{\lesssim} S_{\hat{\cal D}}$ and $Q_{{\cal D}}
\underset{p,C_*}{\lesssim} Q_{\hat{\cal D}}$. Apply (\ref{sauv})
to the trees ${\cal D}$ and $\hat{\cal D}$ and notice that
$\hat{\cal D}\in {\cal J}'_{\hat\xi}$ in respect to the tree
$\hat{\cal D}$. We get
$$
\mathfrak{S}_{{\cal A}_{\xi_*},u,w} \underset{p,\hat C}{\asymp}
S_{{\cal D}} Q_{{\cal D}}\, \underset{p,C_*} {\lesssim}\,
S_{\hat{\cal D}}Q_{\hat{\cal
D}}\stackrel{(\ref{sdasbd}),(\ref{rdqd})}{\le} B_{\hat{\cal
D}}R_{\hat{\cal D}} \underset{p,\hat C}{\lesssim} \mathfrak{S}_
{\hat{\cal D},\hat u,\hat w}\le \mathfrak{S}_ {\hat{\cal A},\hat
u,\hat w}
$$
(see Remark \ref{subtr}).
\end{proof}
\subsection{Estimates for the special class of weights}
Let $r=d$, $p=q$ and let the conditions (\ref{def_h}),
(\ref{yty}), (\ref{ghi_g0}), (\ref{psi_cond}), (\ref{muck}),
(\ref{beta}), (\ref{phi_g}), (\ref{ll}), (\ref{g0ag}) hold. From
(\ref{beta}) it follows that $\beta\le d$.

Let ${\cal T}$, $F$ be the tree and the mapping such as in Lemma
\ref{cor_omega_t}, and let $\overline{s}=\overline{s}(a, \, d)\in
\N$ be such as in Proposition \ref{dif_nu}. Let $m\in \N$ be
divisible in $\overline{s}$. Consider the partition $\{{\cal
D}_{j,i}\}_{j\in \Z_+, \, i\in \tilde I_j}$ of the tree ${\cal T}$
defined at the page \pageref{aij}. Fix $N\in \N$. Let
\label{a_def} ${\cal A}={\cal A}(m)$ be the tree with the set of
vertices $\{\eta_{j,i}\}_{0\le j\le N,\, i\in \tilde I_j}$ and
with the set of edges defined by
$$
{\bf V}^{\cal A}_1(\eta_{j,i})=\{\eta_{j+1,s}\}_{s\in \tilde
I^1_{j,i}}.
$$
Here $\tilde I^1_{j,i}$ is defined in (\ref{ijtl}). By Remark
\ref{jjj}, if ${\cal D}_{j',i'}$ follows the tree ${\cal
D}_{j,i}$, then $j'=j+1$ and $i'\in \tilde I^1_{j,i}$. Hence,
${\rm card}\, {\bf V}^{\cal A}_l(\eta_{j,i})={\rm card}\, \tilde
I^l_{j,i}$ for any $l\in \Z_+$.

By Lemma \ref{litt}, for any $j_0$, $j\in \{0, \, \dots, \, N\}$,
$j\ge j_0$, and for any $\xi\in {\bf V}_{j_0}^{\cal
A}(\eta_{0,1})$ we have
\begin{align}
\label{acvj} {\rm card}\, {\bf V} ^{\cal
A}_{j-j_0}(\xi)\underset{a,d,c_0}{\lesssim}
\frac{h(2^{-mj_0})}{h(2^{-mj})}\stackrel{(\ref{def_h})}{=}
2^{\psi(j)-\psi(j_0)}
\end{align}
with
\begin{align}
\label{pjmt} \psi(j)=m\theta j-\log _2\Lambda(2^{-mj}).
\end{align}

Denote $\xi_0=\eta_{0,1}$. Set
\begin{align}
\label{ujwj}\begin{array}{c} u_j:=u(\xi)=\varphi_g(2^{-mj})\cdot
2^{-\frac{mdj}{p'}}\stackrel{(\ref{ghi_g0})}{=}
2^{mj\left(\beta_g-\frac{d}{p'}\right)}\Psi_g(2^{-mj}),
\\ w_j:=w(\xi)=\varphi_v(2^{-mj})\cdot 2^{-\frac{mdj}{p}}
\stackrel{(\ref{ghi_g0})}{=} 2^{mj\left(\beta_v-\frac{d}{p}
\right)}\Psi_v(2^{-mj}), \;\; \xi\in {\bf V}_j^{\cal A}(\xi_0).
\end{array}
\end{align}
\begin{Lem}
\label{p_eq_q_hardy}  There exists $m_*=m_*(\mathfrak{Z})\in \N$
such that for any $m\ge m_*$, $\xi_*\in {\bf V}_{j_0}^{{\cal
A}}(\xi_0)$ we have $\mathfrak{S}_{{\cal A}_{\xi_*},
u,w}\underset{\mathfrak{Z}}{\lesssim}
2^{mj_0(\beta-d)}\Psi(2^{-mj_0})$ in the case a) of (\ref{beta});
in the case b) for $\alpha>0$ we have $\mathfrak{S}_{{\cal
A}_{\xi_*},u,w}\underset{\mathfrak{Z}}{\lesssim}
j_0^{-\alpha}\rho(j_0)$; if $\alpha=0$ and $\rho\equiv 1$, then
$\mathfrak{S}_{{\cal
A}_{\xi_*},u,w}\underset{\mathfrak{Z}}{\lesssim} 1$.
\end{Lem}
\begin{proof}
First suppose that
\begin{align}
\label{bgdp} \beta_g-\frac{d}{p'}-\frac{\theta}{p}>0.
\end{align}
We have
$$
u_j\cdot 2^{-\frac{\psi(j)}{p}}=2^{mj\left(\beta_g-\frac{d}{p'}
-\frac{\theta}{p}\right)}\cdot \Psi_g(2^{-mj})\Lambda^{\frac 1p}
(2^{-mj}),
$$
$$
w_j\cdot 2^{\frac{\psi(j)}{p}}=2^{mj\left(\beta_v-\frac{d}{p}
+\frac{\theta}{p}\right)}\cdot \Psi_v(2^{-mj})\Lambda^{-\frac 1p}
(2^{-mj}).
$$
From (\ref{muck}) and (\ref{bgdp}) it follows that (\ref{sig_uj})
and (\ref{wjps}) hold with  $\sigma
\underset{\mathfrak{Z}}{\lesssim} \lambda_*^{m}$,
$\lambda_*=\lambda_*(\mathfrak{Z})\in (0, \, 1)$. From
(\ref{acvj}) follows (\ref{vjj0}) with $C_*=C_*(a, \, d, \, c_0)$.
There exists $m_*$ such that $\sigma<\sigma_2(p, \, C_*)$ for any
$m\ge m_*$ (see Lemma \ref{hata}). Let the tree $(\hat {\cal A},
\, \hat \xi)$ satisfy (\ref{def_hata}) with $\psi_*$ defined by
(\ref{psi_st}), and let $\hat u(\xi)=u_j$, $\hat w(\xi)=w_j$ for
$\xi\in {\bf V}_{j-j_0}(\hat \xi_0)$. By Lemma \ref{hata},
$\mathfrak{S}_{{\cal A}_{\xi_*},u,w} \underset{\mathfrak{Z}}
{\lesssim} \mathfrak{S}_{\hat{\cal A},\hat u,\hat w}$.

The quantity $\mathfrak{S}_{\hat{\cal A},\hat u,\hat v}$ equals to
the minimal constant $C$ in
\begin{align}
\label{vg} \sum \limits_{\xi \in {\bf V}(\hat{\cal A})} \hat
w^p(\xi) \left( \sum \limits _{\hat \xi\le \xi'\le \xi}\hat
u(\xi')f^{1/p}(\xi')\right)^p \le C^p\sum \limits_{\xi \in {\bf
V}(\hat{\cal A}_{\hat{\xi}})} f(\xi), \;\; f:{\bf V}(\hat{\cal
A})\rightarrow \R_+.
\end{align}

Denote by ${\cal F}(f)$ the left-hand side of (\ref{vg}).

We claim that the function ${\cal F}$ is concave. Indeed, let
$\lambda\in [0, \, 1]$, $f_1$, $f_2:{\bf V}(\hat{\cal A} _{\hat
\xi})\rightarrow \R_+$. Applying the inverse Minkowski inequality
and the homogeneity property, we get
$$
\sum \limits _{\xi \in {\bf V}(\hat{\cal A})} \hat w^p(\xi)\left(
\sum \limits _{\hat \xi\le\xi'\le \xi} \hat u(\xi')
\left((1-\lambda) f_1(\xi')+ \lambda
f_2(\xi')\right)^{1/p}\right)^{p}\ge
$$
$$
\ge (1-\lambda)\sum \limits _{\xi \in {\bf V}(\hat{\cal A})} \hat
w^p(\xi)\left(\sum \limits _{\hat \xi\le \xi'\le \xi} \hat u(\xi')
f_1^{1/p}(\xi')\right)^p+$$$$+\lambda \sum \limits _{\xi \in {\bf
V}(\hat{\cal A})}\hat w^p(\xi)\left(\sum \limits _{\hat \xi\le
\xi'\le \xi} \hat u(\xi') f_2^{1/p}(\xi')\right)^p.
$$

Set $n_j={\rm card}\,{\bf V}_1^{\hat{\cal A}}(\xi)$, $\xi\in {\bf
V}_j^{\hat{\cal A}}(\hat \xi)$, $j\in \Z_+$. It follows from
(\ref{def_hata}) that this quantity does not depend on $\xi$.
Prove that
\begin{align}
\label{sym}
\begin{array}{c}
\sup \{{\cal F}(f):\, \|f\|_{l_1(\hat{\cal A})} \le 1\}=\\ =\sup
\{{\cal F}(f):\,\|f\|_{l_1(\hat{\cal A})} \le 1, \; \forall j\in
\Z_+, \;\forall \xi', \, \xi''\in {\bf V}_j(\hat \xi)
\;\;f(\xi')=f(\xi'')\}
\end{array}
\end{align}
(see the notation (\ref{flpg})).

Construct $f_{k;i_1, \, \dots, \, i_k, \, i_{k+1}}$ by induction
on $k\in \{0, \, 1, \, \dots, \, N-j_0\}$. Set $f_{0}=f(\hat
\xi)$. Let $0\le k\le N-j_0-1$, $f_{k;i_1, \, \dots, \,
i_k}=f(\xi)$ for some $\xi\in {\bf V}_k^{\hat{\cal A}}(\hat \xi)$.
Then we define $f_{k+1;i_1, \, \dots, \, i_k, \, i_{k+1}}$ for
$1\le i_{k+1}\le n_k$ so that
$$
\{f_{k+1;i_1, \, \dots, \, i_k, \,i_{k+1}}\}_{i_{k+1}=1}^{n_k}=
\{f(\xi'):\; \xi'\in {\bf V}^{\hat{\cal A}}_1(\xi)\}.
$$

Denote by $\mathbb{S}_j$ the set of permutations of $j$ elements.

For $0\le t\le N-j_0-1$, $\sigma \in \mathbb{S}_{n_t}$ we set
$$
(f^{t, \sigma})_{k;i_1, \, \dots, \, i_k}= \left\{
\begin{array}{l} f_{k;i_1, \, \dots, \, i_k}, \, \text{ for }k\le t,
\\ f_{k;i_1, \, \dots , \, \sigma(i_{t+1}), \,
\dots, \, i_k}, \, \text{ for }k> t, \end{array} \right.
$$
$$
\phi_{(t)}(f)= \frac{1}{{\rm card}\, \mathbb{S}_{n_t}}\sum \limits
_{\sigma \in \mathbb{S}_{n_t}}f^{t, \sigma}.
$$
Since the function ${\cal F}$ is concave and ${\cal
F}(f^{t,\sigma})= {\cal F}(f)$, we get
$$
{\cal F}(f)\le{\cal F}(\phi_{(0)}(f))\le {\cal
F}(\phi_{(1)}\phi_{(0)}(f))\le \dots \le {\cal
F}(\phi_{(N-j_0-1)}\dots \phi_{(0)}(f)).
$$
It remains to observe that $\left(\phi_{(N-j_0-1)}\dots
\phi_{(0)}(f)\right)(\xi')= \left(\phi_{(N-j_0-1)}\dots
\phi_{(0)}(f)\right)(\xi'')$ for any $\xi'$, $\xi''\in {\bf
V}_k(\hat \xi)$, $0\le k\le N-j_0$.

Thus, (\ref{sym}) holds. Hence, it suffices to find the minimal
constant $C$ in (\ref{vg}) for the family of functions $f$ such
that $f_{{\bf V}_k(\hat \xi)}=f_k$, $0\le k\le N-j_0$. Set
$m_k=n_0\dots n_k$. From (\ref{def_hata}) it follows that
$n_k=2^{\psi_*(j_0+k+1)- \psi_*(j_0+k)}$ and
\begin{align}
\label{mj_as} m_k=2^{\psi_*(j_0+k+1)-\psi_*(j_0)}
\stackrel{(\ref{2p_j}),(\ref{pjmt})}{\asymp} 2^{\theta
mk}\frac{\Lambda(2^{-mj_0})}{\Lambda(2^{-m(j_0+k+1)})}.
\end{align}
Let $x_k=\left(m_{k-1} f_k\right)^{1/p}$, $m_{-1}=1$. Then it
follows from the definition of $\hat u$ and $\hat w$ that
(\ref{vg}) can be written as
\begin{align}
\label{vg1} \left(\sum \limits _{k=0}^{N-j_0} m_{k-1}
w_{k+j_0}^p\left(\sum \limits _{l=0}^k u_{l+j_0}m_{l-1}^{-\frac
1p} x_l\right)^p\right)^{1/p} \le C\left(\sum \limits
_{k=0}^{N-j_0} x_k^p\right)^{1/p}.
\end{align}
Applying Theorem \ref{hardy_diskr}, we get
$$
C\underset{p}{\asymp} \sup _{0\le k\le N-j_0} \left(\sum \limits
_{l=k}^{N-j_0} m_{l-1}w_{l+j_0}^p\right)^{\frac 1p} \left(\sum
\limits _{l=0}^k u_{l+j_0}^{p'}
m_{l-1}^{-\frac{p'}{p}}\right)^{\frac{1}{p'}}.
$$
Apply Lemma \ref{sum_lem}, taking into account Remark \ref{ste}.
From (\ref{muck}), (\ref{ujwj}) and (\ref{mj_as}) it follows that
$\sum \limits _{l=k}^{N-j_0} m_{l-1}w_{l+j_0}^p
\underset{\mathfrak{Z}}{\asymp} m_{k-1}w_{k+j_0}^p$. The condition
$\beta_g>\frac{d}{p'}+\frac{\theta}{p}$  yields the inequality
$\sum \limits_{l=0}^k u_{l+j_0}^{p'}m_{l-1}^{-\frac{p'}{p}}
\underset{\mathfrak{Z}}{\asymp} m_{k-1}^{-\frac{p'}{p}}
u_{k+j_0}^{p'}$. Therefore,
$$
C\underset{\mathfrak{Z}}{\asymp} \sup _{0\le k\le N-j_0}
u_{k+j_0}w_{k+j_0} \stackrel{(\ref{ujwj})}{=} \sup _{j_0\le t\le
N} 2^{mt(\beta-d)}\Psi(2^{-mt}) =:M.
$$
In the case (\ref{beta}), a), we have
$M\underset{\mathfrak{Z}}{\asymp} 2^{mj_0(\beta-d)}
\Psi(2^{-mj_0})$. In the case (\ref{beta}), b) for $\alpha>0$ we
get $M\underset{\mathfrak{Z},m}{\asymp} j_0^{-\alpha}\rho(j_0)$.
If $\alpha=0$ and $\rho\equiv 1$, then $M=1$.

Let, now, $\beta_g-\frac{d}{p'}-\frac{\theta}{p}\le 0$. Since
$\beta_v<\frac{d-\theta}{p}$, there exists $\beta_{\tilde
g}>\frac{d}{p'}+\frac{\theta}{p}$ such that $\beta_{\tilde
g}+\beta_v<d$. Set $\tilde u(\xi)=u(\xi)\cdot 2^{(\beta_{\tilde
g}-\beta_g)mj}$, $\xi\in {\bf V}_j^{{\cal A}}(\xi_0)$. Then
$$\mathfrak{S}_{{\cal A}_{\xi_*},
u,w}\le \mathfrak{S}_{{\cal A}_{\xi_*}, \tilde u,w}\cdot
2^{mj_0(\beta_g-\beta_{\tilde g})}
\underset{\mathfrak{Z}}{\lesssim}
2^{mj_0(\beta-d)}\Psi(2^{-mj_0}).$$ This completes the proof.
\end{proof}
\section{The discrete Hardy-type inequality on the tree: case $p\ne q$}
Let the tree ${\cal A}$ be such as in the previous section, and
let
\begin{align}
\label{ujwj11}\begin{array}{c} u(\xi)=\varphi_g(2^{-mj})\cdot
2^{-\frac{mdj}{p'}}=
2^{mj\left(\beta_g-\frac{d}{p'}\right)}\Psi_g(2^{-mj}),
\\ w(\xi)=\varphi_v(2^{-mj})\cdot 2^{-\frac{mdj}{q}}=
2^{mj\left(\beta_v-\frac{d}{q}\right)}\Psi_v(2^{-mj}), \;\; \xi\in
{\bf V}_j(\xi_0).
\end{array}
\end{align}
Let $\xi_*\in {\bf V}_{j_0}(\xi_0)$. Denote by
$\mathfrak{S}_{{\cal A}_{\xi_*},u,v}^{p,q}$ the minimal constant
$C$ in the inequality
\begin{align}
\label{hardy_pq} \left(\sum \limits _{\xi\in {\bf V}({\cal
A}_{\xi_*})} w^q(\xi)\left(\sum \limits _{\xi_*\le \xi'\le \xi}
u(\xi')f(\xi') \right)^q\right)^{\frac 1q} \le C\|f\|_{l_p({\cal
A}_{\xi_*})}.
\end{align}

\begin{Lem}
\label{pgq} Let $p>q$. Then there exists $m_*=m_*(\mathfrak{Z})\in
\N$ such that for any $m\ge m_*$
$$\mathfrak{S}_{{\cal A}_{\xi_*},u,w}^{p,q}
\underset{\mathfrak{Z}}{\lesssim} 2^{mj_0\left(\beta-d- \frac
dq+\frac dp\right)}\Psi(2^{-mj_0})$$ for the case a) in
(\ref{beta}); in the case b), for $\alpha>\left(\frac 1q-\frac
1p\right)(1-\gamma)$
$$
\mathfrak{S}_{{\cal A}_{\xi_*},u,w}^{p,q}\underset{\mathfrak{Z}}
{\lesssim} 2^{-\theta\left(\frac 1q-\frac
1p\right)mj_0}j_0^{-\alpha+\frac 1q-\frac 1p}\rho(j_0).
$$
\end{Lem}
\begin{proof}
First consider the case $p=\infty$. Let $\beta_g>d$. Then
$$
\left(\sum \limits _{\xi\in {\bf V}({\cal A}_{\xi_*})}
w^q(\xi)\left(\sum \limits _{\xi_*\le\xi'\le \xi} u(\xi')f(\xi')
\right)^q\right)^{\frac 1q}\le $$$$\le\left(\sum \limits _{\xi\in
{\bf V}({\cal A}_{\xi_*})}w^q(\xi)\left( \sum \limits _{\xi_*\le
\xi'\le \xi} u(\xi')\right)^q\right) ^{\frac
1q}\|f\|_{l_\infty({\cal A}_{\xi_*})}\stackrel
{(\ref{acvj}),(\ref{pjmt}),(\ref{ujwj11})}{\underset{a,d,c_0,m}{\lesssim}}
$$
$$
\lesssim\left(\sum \limits _{j=j_0}^N 2^{mj\left(\beta_vq-d
\right)}\Psi_v^q(2^{-mj})\cdot 2^{m\theta(j-j_0)}
\frac{\Lambda(2^{-mj_0})}{\Lambda(2^{-mj})} \left(\sum \limits
_{l=j_0}^j 2^{ml\left(\beta_g
-d\right)}\Psi_g(2^{-ml})\right)^q\right)^{1/q}
\|f\|_{l_\infty({\cal A}_{\xi_*})}.
$$
This together with the condition $\beta_g>d$, Lemma \ref{sum_lem}
and Remark \ref{ste} yield
$$
\mathfrak{S}_{{\cal A}_{\xi_*},u,w}^{\infty,q}
\underset{\mathfrak{Z}}{\lesssim} \left(\sum \limits _{j=j_0}^N
2^{mj(\beta q-d-dq+\theta)}\Psi^q(2^{-mj})\cdot 2^{-\theta
mj_0}\frac{\Lambda(2^{-m j_0})} {\Lambda(2^{-mj})}\right)^{1/q}.
$$
In the case (\ref{beta}), a), the right-hand side can be estimated
from above up to a multiplicative constant by
$2^{mj_0\left(\beta-d-\frac dq\right)}\Psi (2^{-mj_0})$; in the
case (\ref{beta}), b) it is estimated by
$$
\left(\sum \limits _{j=j_0}^N (mj)^{-\alpha q}\rho^q(mj) \cdot
2^{-\theta mj_0}\frac{j_0^\gamma \tau(mj_0)}{j^\gamma
\tau(mj)}\right)^{\frac 1q} \underset{\mathfrak{Z},m}{\lesssim}
2^{-\frac{\theta mj_0}{q}}j_0^{-\alpha+\frac 1q}\rho(j_0)
$$
(here we use Lemma \ref{sum_lem} and Remark \ref{ste} again). If
$\beta_g\le d$, then we choose $\beta_{\tilde g}>d$ so that
$\beta_{\tilde g}+\beta_v<d+\frac{d-\theta}{q}$ (it is possible by
(\ref{muck})), and we set $\tilde u(\xi)=u(\xi)\cdot
2^{mj(\beta_{\tilde g}-\beta_g)}$, $\xi\in {\bf V}_j^{{\cal
A}}(\xi_0)$. Then
$$\mathfrak{S}_{{\cal A}_{\xi_*},u,w}^{\infty,q}
\underset{\mathfrak{Z}}{\lesssim} \mathfrak{S}_{{\cal
A}_{\xi_*},\tilde u,w}^{\infty,q}\cdot
2^{mj_0(\beta_g-\beta_{\tilde g})}\underset{\mathfrak{Z}}
{\lesssim} 2^{mj_0\left(\beta-d-\frac dq\right)}\Psi
(2^{-mj_0}).$$

Let, now, $q<p<\infty$. Let $\beta_{g,1}+\beta_{g,2}=\beta_g$,
$\beta_{v,1}+\beta_{v,2}=\beta_v$,
$$
u_1(\xi)=2^{mj\left(\beta_{g,1}-\frac{d(p-q)}{p}\right)}
\Psi_g(2^{-mj}), \;\; u_2(\xi)=2^{mj\left(
\beta_{g,2}-\frac{d(q-1)}{p}\right)},
$$
$$
w_1(\xi)=2^{mj\left(\beta_{v,1}-\frac{d(p-q)}{pq}\right)}
\Psi_v(2^{-mj}), \;\; w_2(\xi)=2^{mj\left(
\beta_{v,2}-\frac{d}{p}\right)}, \;\; \xi\in {\bf V}_j(\xi_0).
$$
Then $u_1(\xi)u_2(\xi)=u(\xi)$, $w_1(\xi)w_2(\xi)=w(\xi)$.
Applying the H\"{o}lder inequality, we get
$$
\left(\sum \limits _{\xi \in {\bf V}({\cal A}_{\xi_*})}
w_1^q(\xi)w_2^q(\xi)\left(\sum \limits _{\xi_*\le\xi'\le \xi}
u_1(\xi')u_2(\xi')f(\xi')\right)^q\right)^{1/q}\le
$$
$$
\le \left(\sum \limits _{\xi \in {\bf V}({\cal A}_{\xi_*})}
w_1^q(\xi)w_2^q(\xi)\left(\sum \limits _{\xi_*\le\xi'\le \xi}
u_1^{\frac{p}{p-q}}(\xi')\right)^{q\left(1-\frac qp\right)}
\left(\sum \limits _{\xi_*\le \xi'\le \xi}u_2^{\frac
pq}(\xi')f^{\frac pq}(\xi') \right) ^{\frac{q^2}{p}}
\right)^{1/q}\le
$$
$$
\le \left(\sum \limits _{\xi \in {\bf V}({\cal A}_{\xi_*})}
w_1^{\frac{pq}{p-q}}(\xi)\left(\sum \limits _{\xi_*\le\xi'\le \xi}
u_1^{\frac{p}{p-q}}(\xi')\right)^q\right)^{\frac 1q-\frac 1p}
\left(\sum \limits _{\xi \in {\bf V}({\cal A}_{\xi_*})}
w_2^p(\xi)\left(\sum \limits _{\xi_*\le\xi'\le \xi} u_2^{\frac pq}
(\xi')f^{\frac pq}(\xi')\right)^q\right)^{\frac 1p}=
$$
$$
=\left(\sum \limits _{\xi \in {\bf V}({\cal A}_{\xi_*})} \tilde
w_1^q(\xi)\left(\sum \limits _{\xi_*\le\xi'\le \xi} \tilde
u_1(\xi')\right)^q\right)^{\frac 1q\left(1-\frac qp\right)}
\left(\sum \limits _{\xi \in {\bf V}({\cal A}_{\xi_*})} \tilde
w_2^q(\xi)\left(\sum \limits _{\xi_*\le\xi'\le \xi} \tilde
u_2(\xi')f_2(\xi')\right)^q\right)^{\frac 1q\cdot \frac qp},
$$
with $\tilde u_1(\xi)=u_1^{\frac{p}{p-q}}(\xi)$, $\tilde
w_1(\xi)=w_1^{\frac{p}{p-q}}(\xi)$, $\tilde u_2(\xi)=u_2^{\frac
pq}(\xi)$, $\tilde w_2(\xi)= w_2^{\frac pq}(\xi)$,
$f_2(\xi)=f^{\frac pq}(\xi)$ (then $f_2\in l_q({\cal
A}_{\xi_*})$).

Check that we can apply to each of multipliers the Hardy-type
inequality with $(p_1, \, q_1)=(\infty, \, q)$ and $(p_2, \,
q_2)=(q, \, q)$. Indeed, for $\xi\in {\bf V}_j(\xi_0)$ we have
$$\tilde u_1(\xi)=2^{mj\left(\beta_{g,1}\frac{p}{p-q}
-d\right)}\Psi_g^{\frac{p}{p-q}}(2^{-mj}), \;\; \tilde
u_2(\xi)=2^{mj\left(\beta_{g,2}\frac{p}{q}-\frac{d}{q'}\right)},$$
$$\tilde w_1(\xi)=2^{mj\left(\beta_{v,1}\frac{p}{p-q}-\frac dq\right)}
\Psi_v^{\frac{p}{p-q}}(2^{-mj}), \;\;\tilde
w_2(\xi)=2^{mj\left(\beta_{v,2}\frac pq -\frac dq\right)}.$$ By
Remark \ref{ste}, since the functions $\Psi_g$ and $\Psi_v$
satisfy (\ref{psi_cond}) ($\rho_g$ and $\rho_v$ satisfy
(\ref{ll}), respectively), we observe that their powers satisfy
the similar conditions. First choose $\beta_{v,1}$ and
$\beta_{v,2}$ so that
$$
\beta_{v,1}<(d-\theta)\left(\frac 1q-\frac 1p\right),
\quad\beta_{v,2}< \frac{d-\theta}{p}, \quad
\beta_{v,1}+\beta_{v,2}=\beta_v
$$
hold (it is possible, since $\beta_v<\frac{d-\theta}{q}$). Then we
choose $\beta_{g,1}$, $\beta_{g,2}$. In the case (\ref{beta}), a)
require
$$
\beta_{g,1}+\beta_{v,1}<\left(1-\frac qp\right)\left(d+ \frac
dq-\frac{\theta}{q}\right),\quad \beta_{g,2}+\beta_{v,2}<\frac
{qd}{p},\quad \beta_{g,1}+\beta_{g,2}=\beta_g.
$$
It is possible, since $\beta_g+\beta_v<d+\frac dq-\frac dp-
\theta\left(\frac 1q-\frac 1p\right)=\left(1-\frac
qp\right)\left(d+ \frac dq-\frac{\theta}{q}\right)+\frac {qd}{p}$.
In the case (\ref{beta}), b) we require
$$
\beta_{g,1}+\beta_{v,1}=\left(1-\frac qp\right)\left(d+ \frac
dq-\frac{\theta}{q}\right),\;\; \beta_{g,2}+\beta_{v,2}=\frac
{qd}{p}.
$$
The condition $\alpha\frac{p}{p-q}>\frac{1-\gamma}{q}$ holds by
(\ref{g0ag}).

Also observe that $\|f_2\|_{l_q({\cal A}_{\xi_*})}^{q/p}
=\|f\|_{l_p({\cal A}_{\xi_*})}$.

Thus, in the case (\ref{beta}), a)
$$
\mathfrak{S}_{{\cal A}_{\xi_*},u,w}^{p,q} \underset
{\mathfrak{Z}}{\lesssim} \left[2^{mj_0\left(
(\beta_{g,1}+\beta_{v,1})\frac{p}{p-q}-d-\frac dq\right)}
\Psi^{\frac{p}{p-q}}(2^{-mj_0})\right]^{1-\frac
qp}\times$$$$\times
\left[2^{mj_0\left((\beta_{g,2}+\beta_{v,2})\frac pq-d\right)}
\right]^{\frac qp}=2^{mj_0\left( \beta -d-\frac dq+\frac
dp\right)}\Psi(2^{-mj_0}),
$$
as well as in the case (\ref{beta}), b)
$$
\mathfrak{S}_{{\cal A}_{\xi_*},u,w}^{p,q}
\underset{\mathfrak{Z}}{\lesssim} \left(2^{-\frac{m\theta
j_0}{q}}j_0^{-\frac{\alpha p}{p-q}+\frac
1q}\rho^{\frac{p}{p-q}}(j_0)\right) ^{1-\frac qp}=
2^{-\theta\left(\frac 1q-\frac 1p\right)mj_0} j_0^{-\alpha+\frac
1q-\frac 1p}\rho(j_0).
$$
This completes the proof.
\end{proof}
\begin{Lem}
\label{plq} Let $p<q$. Then there exists $m_*=m_*(\mathfrak{Z})\in
\N$ such that for any $m\ge m_*$
$$\mathfrak{S}_{{\cal A}_{\xi_*},u,w}^{p,q}
\underset{\mathfrak{Z}}{\lesssim} 2^{mj_0\left(\beta-d- \frac
dq+\frac dp\right)}\Psi(2^{-mj_0})$$ in the case a) of
(\ref{beta}); in the case b), if $\alpha>0$, then
$$\mathfrak{S}_{{\cal A}_{\xi_*},u,w}^{p,q}\underset{\mathfrak{Z}}
{\lesssim} j_0^{-\alpha}\rho(j_0).$$
\end{Lem}
\begin{proof}
Set $\lambda=\frac 1p-\frac 1q$, and define the quantity $p_1$ by
$\frac 1p=\frac{1-\lambda}{p_1}+\lambda$. Then $\frac 1q=
\frac{1-\lambda}{p_1}$. Applying the H\"{o}lder inequality, we get
$$
S:=\left(\sum \limits _{\xi\in {\bf V}({\cal A}_{\xi_*})}
w^q(\xi)\left(\sum\limits_{\xi_*\le \xi'\le
\xi}u(\xi')f(\xi')\right)^q \right)^{1/q}= $$$$=\left(\sum \limits
_{\xi\in {\bf V}({\cal A} _{\xi_*})} w^q(\xi)\left(\sum \limits
_{\xi_*\le\xi'\le \xi} u(\xi') f^{\frac pq}(\xi')f^{1-\frac
pq}(\xi')\right)^q\right)^{1/q}\le
$$
$$
\le \left(\sum \limits _{\xi\in {\bf V}({\cal A}_{\xi_*})}
w^q(\xi)\left(\sum \limits _{\xi_*\le\xi'\le \xi} u
^{\frac{1}{1-\lambda}}(\xi')f^{\frac{p}{q(1-\lambda)}}(\xi')
\right)^{(1-\lambda)q}\left(\sum \limits _{\xi_*\le \xi'\le \xi}
f^{\frac{1-\frac pq}{\lambda}}(\xi')\right)^{\lambda
q}\right)^{\frac 1q}\le
$$
$$
\le \left(\max _{\xi\in {\bf V}({\cal A}_{\xi_*})} \sum \limits
_{\xi'\le \xi} f^p(\xi')\right)^\lambda \left(\sum \limits
_{\xi\in {\bf V}({\cal A}_{\xi_*})} w^{\frac{p_1}{1-\lambda}}(\xi)
\left(\sum \limits _{\xi_*\le \xi'\le \xi} u^{\frac{1}{1-\lambda}}
(\xi') f^{\frac{p}{q(1-\lambda)}}(\xi') \right)^{p_1}\right)
^{\frac{1-\lambda}{p_1}}\le
$$
$$
\le \|f\|_{l_p({\cal A}_{\xi_*})}^{\lambda p} \left(\sum \limits
_{\xi \in {\bf V}({\cal A}_{\xi_*})} \tilde w^{p_1}(\xi)\left(\sum
\limits _{\xi_*\le\xi'\le \xi} \tilde u(\xi')\tilde
f(\xi')\right)^{p_1}\right)^{\frac{1-\lambda}{p_1}},
$$
with $\tilde w(\xi)=w^{\frac{1}{1-\lambda}}(\xi)$, $\tilde u(\xi)
=u^{\frac{1}{1-\lambda}}(\xi)$, $\tilde
f(\xi)=f^{\frac{p}{q(1-\lambda)}}(\xi)$. We have
$$
\|\tilde f\|_{l_{p_1}({\cal A}_{\xi_*})}^{1-\lambda} =\left(\sum
\limits _{\xi\in {\bf V}({\cal A}_{\xi_*})}
f^{\frac{pp_1}{q(1-\lambda)}} \right)^{\frac{1-\lambda}{p_1}}=
\left(\sum \limits _{\xi\in {\bf V}({\cal A}_{\xi_*})}
f^p(\xi)\right)^{\frac{1}{q}}=\|f\|_{l_p({\cal A}_{\xi_*})}^{\frac
pq}.
$$
Hence,
$$
S\le \|f\|_{l_p({\cal A}_{\xi_*})}^{1-\frac pq}\cdot
(\mathfrak{S}^{p_1,p_1}_{{\cal A}_{\xi_*},\tilde u,\tilde
w})^{1-\lambda}\|f\|_{l_p({\cal A}_{\xi_*})}^{\frac pq}=
(\mathfrak{S}^{p_1,p_1}_{{\cal A}_{\xi_*},\tilde u,\tilde
w})^{1-\lambda}\|f\|_{l_p({\cal A}_{\xi_*})}.
$$
If $\xi\in {\bf V}_j(\xi_0)$, then $$\tilde w(\xi)=2^{mj
\left(\frac{\beta_v}{1-\lambda}-\frac{d}{(1-\lambda)q}\right)}
\Psi_v^{\frac{1}{1-\lambda}}(2^{-mj})=2^{mj
\left(\frac{\beta_v}{1-\lambda}-\frac{d}{p_1}\right)}
\Psi_v^{\frac{1}{1-\lambda}}(2^{-mj}),$$ $$\tilde u(\xi)=2^{mj
\left(\frac{\beta_g}{1-\lambda}-\frac{d}{(1-\lambda)p'}\right)}
\Psi_g^{\frac{1}{1-\lambda}}(2^{-mj}).$$  Therefore,
$$
\tilde u(\xi)= 2^{mj\left(\tilde \beta_g-\frac{d}{p_1'}
\right)}\tilde \Psi_g(2^{-mj}), \quad \tilde w(\xi)=
2^{mj\left(\tilde \beta_v-\frac{d}{p_1}\right)}\tilde
\Psi_v(2^{-mj}),
$$
with $\tilde \Psi_g=\Psi_g^{\frac{1}{1-\lambda}}$, $\tilde
\Psi_v=\Psi_v^{\frac{1}{1-\lambda}}$, $\tilde
\beta_g=\frac{\beta_g}{1-\lambda}-\frac{d}{(1-\lambda)p'}+\frac{d}{p'_1}$
and $\tilde \beta_v=\frac{\beta_v}{1-\lambda}$. Then $-\tilde
\beta_vp_1+d-\theta\stackrel{(\ref{muck})}{>}0$.

In the case (\ref{beta}), a), we have $\beta <d+\frac dq-\frac
dp$. Check that $\tilde \beta_g+\tilde \beta_v<d$. It is
equivalent to
$\beta_g+\beta_v-\frac{d}{p'}+\frac{d(1-\lambda)}{p_1'}<d(1-\lambda)$,
i.e., $\beta -d+\frac dp-\frac{d(1-\lambda)}{p_1}<0$. It remains
to observe that $\frac{1-\lambda}{p_1}=\frac 1q$. Hence, by Lemma
\ref{p_eq_q_hardy}, $(\mathfrak{S}^{p_1,p_1}_{{\cal A},\tilde
u,\tilde w})^{1-\lambda}\underset{\mathfrak{Z}}{\lesssim}
2^{mj_0\left(\beta-d-\frac dq+\frac dp\right)} \Psi(2^{-mj_0})$.

Consider the case (\ref{beta}), b). Then $\beta=d+\frac dq-\frac
dp$. Hence, $\tilde \beta_g+\tilde \beta_v=d$, and by Lemma
\ref{p_eq_q_hardy} we get $(\mathfrak{S}^{p_1,p_1}_{{\cal
A},\tilde u,\tilde w})^{1-\lambda}\underset {\mathfrak{Z}}
{\lesssim} j_0^{-\alpha}\rho(j_0)$.
\end{proof}

\section{The proof of the embedding theorem}
In this section we prove the main result of this article. In
particular, we obtain Theorem \ref{main}.

Let $m_*=m_*(\mathfrak{Z})$ (see Lemmas \ref{p_eq_q_hardy},
\ref{pgq} and \ref{plq}), and let $\{({\cal D}_{j,i}, \, \hat
\xi_{j,i})\}_{j\in \Z_+, \, i\in \tilde I_j}$ be the partition of
${\cal T}$ for $m=m_*$ (see the definition on page \pageref{aij}).
\begin{Trm}
Let ${\cal D}\subset {\cal T}_{\hat \xi_{j_0,i_0}}$ be a subtree
with the minimal vertex $\hat \xi_{j_0,i_0}$. Then for any
function $f\in {\rm span}\, W^r_{p,g}(\Omega)$ there exists a
polynomial $Pf$ of degree not exceeding $r-1$ such that
\begin{align}
\label{pol_appr} \|f-Pf\|_{L_{q,v}(\Omega_{{\cal D},F})}
\underset{\mathfrak{Z}}{\lesssim} 2^{m_*j_0\left(\beta
-\delta\right)}\Psi(2^{-m_*j_0})\left\|\frac{\nabla^r
f}{g}\right\| _{L_p(\Omega_{{\cal D},F})}
\end{align}
in the case (\ref{beta}), a),
\begin{align}
\label{pol_appr1} \|f-Pf\|_{L_{q,v}(\Omega_{{\cal D},F})}
\underset{\mathfrak{Z}}{\lesssim} 2^{-m_*\theta\left(\frac
1q-\frac 1p\right)_+j_0}j_0^{-\alpha+\left(\frac 1q-\frac
1p\right)_+}\rho(j_0)\left\|\frac{\nabla^r f}{g}\right\|
_{L_p(\Omega_{{\cal D},F})}
\end{align}
in the case (\ref{beta}), b). Here the mapping $f\mapsto Pf$ can
be extended to a linear continuous operator
$P:L_{q,v}(\Omega)\rightarrow {\cal P}_{r-1}(\Omega)$.
\end{Trm}
\begin{proof}
We shall denote $\tilde \Omega=\Omega_{{\cal D},F}$.

{\bf Step 1.} The set $C^\infty(\tilde\Omega)\cap
W^r_{p,g}(\tilde\Omega)$ is dense $W^r_{p,g}(\tilde\Omega)$ (it
can be proved in the same way as for a non-weighted case, see
\cite[p. 16]{mazya1}).\footnote{Here $C^\infty(\Omega)$ is the
space of functions that are smooth on the open set $\Omega$, but
not necessarily extendable to smooth functions on the whole space
$\R^d$.} Therefore, it is sufficient to check (\ref{pol_appr}) and
(\ref{pol_appr1}) for $f\in C^\infty(\tilde\Omega)$.

By Lemma \ref{cor_omega_t}, $\tilde\Omega\in {\bf FC}(b_*)$,
$b_*=b_*(a, \, d)$. Let $x_*\in \tilde \Omega$, $\gamma_x(\cdot)$,
$T(x)$ be such as in Definition \ref{fca}, and let $R_0={\rm
dist}_{\|\cdot\|_{l_2^d}}\, (x_*, \, \partial \tilde\Omega)$. From
assertion 2 of Lemma \ref{cor_omega_t} it follows that we can take
the center of the cube $F(\hat \xi_{j_0,i_0})$ as the point $x_*$.

It is sufficient to show that if $f\in C^\infty(\Omega^{w_0})$,
$f|_{B_{R_0/2}(x_*)}=0$, then (\ref{pol_appr}), (\ref{pol_appr1})
hold with $Pf=0$ (the general case can be proved in the same way
as in \cite{vas_sing}; here we can take as $f\mapsto Pf$ the
Sobolev's projection operators).

Let $\varphi(x)=\frac{|\nabla^r f(x)|}{g(x)}$. By Theorem
\ref{reshteor}, for any $x\in \tilde\Omega$ there exists a set
$G_x\subset \cup _{t\in [0, \, T(x)]}B_{b_*t}(\gamma_x(t))$ such
that
$$
\{(x, \, y)\in \tilde\Omega \times \tilde\Omega:\; x\in \tilde
\Omega, \, y\in G_x\}\quad \text{ is measurable},
$$
$$
|f(x)|\underset{r,d,a}{\lesssim} \int \limits_{G_x}
|x-y|^{r-d}g(y)\varphi(y)\, dy.
$$
By Assertion 2 of Lemma \ref{cor_omega_t},
\begin{align}
\label{x_in_del} \text{ if }x\in \Delta,\;\; \Delta\in
\Theta(\Omega),\text{ then }G_x\subset \Omega_{\le \Delta}.
\end{align}

Thus, it is sufficient to prove that
\begin{align}
\label{tovqgf} \left(\int \limits _{\tilde \Omega}
v^q(x)\left(\int \limits _{G_x}g(y)\varphi(y)|x-y|^{r-d}\,
dy\right)^q\, dx\right)^{1/q}\underset{\mathfrak{Z}}{\lesssim}
C(j_0)\|\varphi\|_{L_p(\tilde \Omega)},
\end{align}
with $C(j_0)=2^{m_*j_0\left(\beta -\delta\right)}
\Psi(2^{-m_*j_0})$ in the case (\ref{beta}), a), and
$$C(j_0)=2^{-m_*\theta \left(\frac 1q-\frac
1p\right)_+j_0}j_0^{-\alpha+\left(\frac 1q-\frac
1p\right)_+}\rho(j_0)$$ in the case (\ref{beta}), b).

Extending the function $\varphi$ by zero to $\Omega_{{\cal
T}_{\hat \xi_{j_0,i_0}},F}\backslash \Omega_{{\cal D},F}$ and
applying the B. Levi's theorem, we may assume that ${\bf V}({\cal
D})=\{\xi\in {\bf V}({\cal T}_{\hat \xi_{j_0,i_0}}): \; \rho
_{{\cal T}}(\hat \xi_{j_0,i_0}, \, \xi)\le N\}$ for some $N\in
\N$.

{\bf Step 2.} Consider the case $r=d$. Let $({\cal A}, \, \xi_0)=
({\cal A}(m_*), \, \xi_0)$ be the tree defined on the page
\pageref{a_def}. If $\xi=\eta_{j,i} \in {\bf V}({\cal A})$, then
we set ${\cal D}[\xi]={\cal D}_{j,i}$, $\Omega[\xi]=\Omega_{{\cal
D}[\xi],F}$,
\begin{align}
\label{gxi} g_\xi=2^{\beta_gm_*j}\Psi_g(2^{-m_*j}),\;\;
v_\xi=2^{\beta_vm_*j} \Psi_v(2^{-m_*j}).
\end{align}
By (\ref{diam_dj}), the property 2 of the partition $\{{\cal
D}_{j,i}\}_{j\in \Z_+,\, i\in \tilde I_j}$ and (\ref{ghi_g0}), we
have
\begin{align}
\label{ddj} {\rm diam}\, \Omega[\xi] \underset{a,d}{\asymp}
2^{-m_*j}, \;\; g(x)\underset{\mathfrak{Z}}{\asymp} g_\xi, \; v(x)
\underset{\mathfrak{Z}}{\asymp} v_\xi, \;\; x\in \Omega_{{\cal
D}[\xi],F}.
\end{align}
Set $\xi_*=\eta_{j_0,i_0}$. Then
$$
\left(\int \limits _{\tilde \Omega}v^q(x)\left(\int \limits
_{G_x}g(y)\varphi(y)\, dy\right)^q\, dx\right)^{1/q}
\stackrel{(\ref{x_in_del})}{\le}
$$
$$
\le\left(\sum \limits _{\xi\in {\bf V}({\cal A}_{\xi_*})}\int
\limits _{\Omega[\xi]} v^q(x)\left(\sum \limits _{\xi_*\le\xi'\le
\xi} \int \limits _{\Omega[\xi']}g(y)\varphi(y)\, dy\right)^q\,
dx\right)^{1/q} \stackrel{(\ref{ddj})}{\underset{\mathfrak{Z}}
{\asymp}}
$$
$$
\asymp \left(\sum \limits _{\xi \in {\bf V}({\cal A} _{\xi_*})}
v_\xi^q \int \limits _{\Omega[\xi]} \left(\sum \limits
_{\xi_*\le\xi'\le \xi} g_{\xi'} \int \limits _{\Omega[\xi']}
\varphi(y)\, dy\right)^q\, dx\right)^{1/q}
\stackrel{(\ref{ujwj11}),(\ref{gxi}), (\ref{ddj})}
{\underset{\mathfrak{Z}}{\lesssim}}
$$
$$
\lesssim\left(\sum \limits _{\xi \in {\bf V}({\cal A} _{\xi_*})}
w^q(\xi)\left(\sum \limits _{\xi_*\le\xi'\le \xi}
u(\xi')\|\varphi\|_{L_p(\Omega[\xi'])}\right)^q\right)^{1/q}
\underset{\mathfrak{Z}}{\lesssim}
$$
$$
\lesssim C(j_0)\left( \sum \limits _{\xi\in {\bf V}({\cal
A}_{\xi_*})}\|\varphi\|^p _{L_p(\Omega[\xi])}\right)^{1/p}
=C(j_0)\|\varphi\| _{L_p(\tilde \Omega)}
$$
(the penultimate relation follows from Lemmas \ref{p_eq_q_hardy},
\ref{pgq} and \ref{plq}).

{\bf Step 3.} Let $r\ne d$. Set
$$
G_x^1=\{y\in G_x:\; |x-y|\ge 2\,{\rm dist}_{|\cdot|}(x, \,
\Gamma)\}, \;\; G_x^2=\{y\in G_x:\; |x-y|<2\,{\rm
dist}_{|\cdot|}(x, \, \Gamma)\}.
$$
Then in order to prove (\ref{tovqgf}) it suffices to check the
inequalities
\begin{align}
\label{tovqgf1} \left(\int \limits _{\tilde \Omega}
v^q(x)\left(\int \limits _{G_x^1}g(y)\varphi(y)|x-y|^{r-d}\,
dy\right)^q\, dx\right)^{1/q}\underset{\mathfrak{Z}}{\lesssim}
C(j_0)\|\varphi\|_{L_p(\tilde \Omega)},
\end{align}
\begin{align}
\label{tovqgf2} \left(\int \limits _{\tilde \Omega}
v^q(x)\left(\int \limits _{G_x^2}g(y)\varphi(y)|x-y|^{r-d}\,
dy\right)^q\, dx\right)^{1/q}\underset{\mathfrak{Z}}{\lesssim}
C(j_0)\|\varphi\|_{L_p(\tilde \Omega)}.
\end{align}

Prove (\ref{tovqgf1}). At first we check that for $y\in G_x^1$
\begin{align}
\label{xmyasdist} |x-y|\underset{a,d}{\asymp} {\rm
dist}_{|\cdot|}(y, \, \Gamma).
\end{align}
Indeed, let $z_x\in \Gamma$, $|x-z_x|={\rm dist}_{|\cdot|}(x, \,
\Gamma)$. Then
$$
{\rm dist}_{|\cdot|}(y, \, \Gamma)\le |y-z_x|\le
|y-x|+|x-z_x|=$$$$=|x-y|+ {\rm dist}_{|\cdot|}(x, \, \Gamma)\le
|x-y|+\frac{|x-y|}{2}=\frac{3|x-y|}{2}.
$$
Prove the inverse inequality. Let $y\in F(\omega)$, $\omega\in
{\bf V}({\cal T})$. From (\ref{x_in_del}) it follows that $x\in
\Omega _{{\cal T}_\omega,F}$. Since $\Omega _{{\cal
T}_\omega,F}\in {\bf FC}(b_*)$, we have ${\rm diam}(\Omega _{{\cal
T}_\omega,F})\underset{a,d}{\lesssim} 2^{-m_\omega}$. From
assertion 2 of Theorem \ref{whitney} it follows that ${\rm
dist}_{|\cdot|}(y, \, \partial \Omega) \underset{d}{\asymp}
2^{-m_\omega}$. Hence,
$$
{\rm dist}_{|\cdot|}(y, \, \Gamma)\ge {\rm dist}_{|\cdot|}(y, \,
\partial \Omega)\underset{d}{\asymp} 2^{-m_\omega}
\stackrel{(\ref{xxiw})}{\underset{a,d}{\gtrsim}} {\rm
diam}\,\Omega _{{\cal T}_\omega,F}\ge |x-y|.
$$
Thus, (\ref{xmyasdist}) is proved, and
$$
\left(\int \limits _{\tilde \Omega} v^q(x)\left(\int \limits
_{G_x^1} g(y)\varphi(y)|x-y|^{r-d}\, dy\right)^q\, dx\right)^{1/q}
\stackrel{(\ref{ghi_g0})}{\underset{\mathfrak{Z}} {\lesssim}}
$$
$$
\lesssim \left(\int \limits _{\tilde \Omega} v^q(x)\left(\int
\limits _{G^1_x} \tilde g(y)\varphi(y)\, dy\right)^q\,
dx\right)^{1/q},
$$
with $\tilde g(y)=\varphi_{\tilde g}({\rm dist}_{|\cdot|}(y, \,
\Gamma))$,
$$
\varphi_{\tilde g}(t)=\varphi_g(t)\cdot t^{r-d}=t^{-\beta_{\tilde
g}} \Psi_g(t), \;\; \beta_{\tilde g}=\beta_g+d-r.
$$
Since $\beta-\delta =\beta_g+\beta_v-r-\frac dq+\frac dp=\beta
_{\tilde g}+\beta_v-d-\frac dq+\frac dp$, it remains to apply the
estimate which was obtained at the previous step.

Prove (\ref{tovqgf2}). If $y\in G_x^2$, then
\begin{align}
\label{100} {\rm dist}_{|\cdot|}(y, \, \Gamma)\le {\rm
dist}_{|\cdot|}(x, \, \Gamma)+ |x-y|\le 3 \,{\rm
dist}_{|\cdot|}(x, \, \Gamma).
\end{align}
Let $x\in \Omega[\eta_{j,i}]$, $y\in \Omega[\eta_{j',i'}]$. From
(\ref{x_in_del}) and property 3 of the partition $\{{\cal
D}_{j,i}\}_{j\in \Z_+,i\in \tilde I_j}$ it follows that $j'\le j$.
By (\ref{dist_x1}), ${\rm dist}_{|\cdot|}(x, \,
\Gamma)\underset{a,d,m_*}{\asymp} 2^{-m_*j}$, ${\rm
dist}_{|\cdot|}(y, \, \Gamma)\underset{a,d,m_*}{\asymp}
2^{-m_*j'}$. This together with (\ref{100}) yield that there
exists $j_*=j_*(a, \, d, \, m_*)$ such that $j-j_*\le j'\le j$.
Notice that
\begin{align}
\label{xy2mj} |x-y|\underset{a,d,m_*}{\lesssim}2^{-m_*j}.
\end{align}

Denote by ${\cal I}_{\eta_{j,i},j_*}$ the maximal subgraph on the
vertex set
$$
{\bf V}({\cal I}_{\eta_{j,i},j_*})=\bigcup _{j'\ge j-j_*}\bigcup
_{\eta_{j',i'}\le \eta_{j,i}} {\bf V}({\cal D} _{\eta_{j',i'}})
$$
and set $\tilde\Omega[\eta_{j,i}]=\Omega_{{\cal
I}_{\eta_{j,i},j_*},F}$. Then for any $x\in \Omega[\eta_{j,i}]$,
the inclusion $G^2_x\subset \tilde \Omega[\eta_{j,i}]$ holds. In
addition, from (\ref{gxi}) and (\ref{ddj}) it follows that
\begin{align}
\label{gygxi} g(y)\underset{\mathfrak{Z}}{\asymp} g_{\eta_{j,i}},
\;\; y\in \tilde\Omega[\eta_{j,i}].
\end{align}
By (\ref{acvj}), for any $\xi'\in {\bf V}({\cal A}_{\xi_*})$
$${\rm card}\, \{\xi\in {\bf V}({\cal A}_{\xi_*}):\; \xi'\in {\cal
I}_{\xi,j_*}\} \underset{\mathfrak{Z}} {\lesssim} 1.$$ Therefore,
\begin{align}
\label{sum} \left(\sum \limits _{\xi\in {\bf V}({\cal A}_{\xi_*})}
\|\varphi\|^p_{L_p(\tilde \Omega[\xi])}\right)^{1/p}
\underset{\mathfrak{Z}}{\lesssim} \|\varphi\| _{L_p(\tilde
\Omega)}.
\end{align}

We have
$$
\left(\int \limits _{\tilde \Omega} v^q(x)\left(\int \limits
_{G_x^2} g(y)\varphi(y)|x-y|^{r-d}\, dy\right)^q\,
dx\right)^{1/q}\le
$$
$$
\le \left(\sum \limits _{\xi\in {\bf V}({\cal A}_{\xi_*})} \int
\limits _{\Omega[\xi]} v^q(x)\left(\int \limits _{\tilde
\Omega[\xi]} g(y)\varphi(y)|x-y|^{r-d}\, dy\right)^q\,
dx\right)^{1/q}\stackrel{(\ref{ddj}),(\ref{gygxi})}{\underset{\mathfrak{Z}}{\asymp}}
$$
$$
\asymp \left(\sum \limits _{\xi\in {\bf V}({\cal A}_{\xi_*})}
g_\xi^qv_\xi^q\int \limits _{\Omega[\xi]} \left(\int \limits
_{\tilde \Omega[\xi]} \varphi(y)|x-y|^{r-d}\, dy\right)^q\,
dx\right)^{1/q}=:S.
$$
Let $\xi=\eta_{j,i}$. By (\ref{ddj}) and (\ref{xy2mj}),
$\Omega[\xi]$ and $\tilde \Omega[\xi]$ are contained in a ball of
radius
\begin{align}
\label{omxidi} R_\xi \underset{a,d,m}{\asymp} 2^{-m_*j}.
\end{align}
Applying Theorem \ref{adams_etc} and the H\"{o}lder inequality, we
get
$$
S\underset{\mathfrak{Z}}{\lesssim} \left(\sum \limits _{\xi\in
{\bf V}({\cal A}_{\xi_*})}g_\xi^qv_\xi^q R_\xi^{\delta q}
\|\varphi\|^q_{L_p(\tilde \Omega[\xi])}\right)^{1/q}=:S_1.
$$
If $p\le q$, then
$$
S_1\le \max _{\xi\in {\bf V}({\cal A}_{\xi_*})} g_\xi v_\xi
R_\xi^\delta \left(\sum \limits _{\xi\in {\bf V}({\cal
A}_{\xi_*})} \|\varphi\|_{L_p(\tilde \Omega[\xi])}^p\right)^{\frac
1p}\stackrel{(\ref{gxi}), (\ref{sum}),
(\ref{omxidi})}{\underset{\mathfrak{Z}}{\lesssim}}
$$
$$
\lesssim 2^{m_*j_0(\beta-\delta)}
\Psi(2^{-m_*j_0})\|\varphi\|_{L_p(\tilde \Omega)} =C(j_0)
\|\varphi\|_{L_p(\tilde \Omega)}.
$$
If $p>q$, then by the H\"{o}lder inequality
$$
S_1\le \left(\sum \limits _{\xi \in {\bf V}({\cal
A}_{\xi_*})}\left(g_\xi v_\xi
R_\xi^\delta\right)^{\frac{pq}{p-q}}\right)^{\frac 1q-\frac 1p}
\left(\sum \limits _{\xi \in {\bf V}({\cal
A}_{\xi_*})}\|\varphi\|^p_{L_p(\tilde \Omega[\xi])}\right)^{\frac
1p} \stackrel{(\ref{gxi}),(\ref{acvj}), (\ref{sum}),
(\ref{omxidi})}{\underset{\mathfrak{Z}}{\lesssim}}
$$
$$
\lesssim \left(\sum \limits _{j=j_0}^\infty \frac{2^{-\theta
m_*j_0}\Lambda(2^{-m_*j_0})}{2^{-\theta m_*j}\Lambda(2^{- m_*j})}
2^{(\beta-\delta)\frac{pq}{p-q}m_*j} \Psi^{\frac{pq}{p-q}}
(2^{-m_*j})\right)^{\frac 1q-\frac 1p}\|\varphi\|_{L_p(\tilde
\Omega)} \stackrel{(\ref{beta})} {\underset{\mathfrak{Z}}
{\lesssim}} C(j_0)\|\varphi\| _{L_p(\tilde \Omega)}
$$
(see Lemma \ref{sum_lem}).
\end{proof}
\vskip 0.5cm Notice that if the condition (\ref{beta}), a) is
replaced by $\beta_g+ \beta_v>\delta-\theta\left(\frac 1q-\frac
1p\right)_+$, or if (\ref{beta}), b) holds and
$\alpha<(1-\gamma)\left(\frac 1q-\frac 1p\right)_+$, then the set
$W^r_{p,g}(\Omega)\cap C_0^\infty(\Omega)$ is unbounded in
$L_{q,v}(\Omega)$. Indeed, let $\varphi\in C_0^\infty([0, \,
1]^d)$, $\varphi\ge 0$, $\int \limits_{[0, \, 1]^d}|\nabla ^r
\varphi(x)|^p\, dx=1$, let $\hat k=\hat k(a, \, d)\in \N$ be such
as in Lemma \ref{nu_st}, and let $\hat \xi$ be the minimal vertex
of the tree ${\cal T}$, $\hat \xi\in \hat{\bf W}_{\nu_0}$. Then
for sufficiently large $\nu\in \N$
\begin{align}
\label{cwj} \sum \limits_{l=\nu}^{\nu+\hat k} {\rm card}\,
\hat{\bf W}_l \stackrel{(\ref{w_nu_d0_as})}
{\underset{\mathfrak{Z}}{\gtrsim}}
\frac{h(2^{-\nu_0})}{h(2^{-\nu})}\underset{\mathfrak{Z},\nu_0}{\gtrsim}
\frac{2^{\nu\theta}}{\Lambda(2^{-\nu})}.
\end{align}
Set $$\{\Delta_j\}_{j\in J_\nu}=\left\{F(\xi):\; \xi\in \cup
_{l=\nu}^{\nu+\hat k}\hat{\bf W}_l\right\}.$$ Then
$\Delta_j=z_j+t_j[0, \, 1]^d$, $t_j\stackrel{(\ref{mwas_kw}),
(\ref{h_w_nu})}{\underset{\mathfrak{Z}}{\asymp}} 2^{-\nu}$. In
addition, from (\ref{ghi_g0}) and (\ref{nw_dg}) it follows that
\begin{align}
\label{ggg} g(x)\underset{\mathfrak{Z}}{\asymp}
2^{\nu\beta_g}\Psi_g(2^{-\nu}), \quad v(x)\underset
{\mathfrak{Z}}{\asymp} 2^{\nu\beta_v}\Psi_v(2^{-\nu}), \quad x\in
\Delta_j, \quad j\in J_\nu.
\end{align}

Let $p\le q$. Take $j\in J_\nu$ and set
$\varphi_\nu(x)=c_\nu\varphi\left(\frac{x-z_j}{t_j}\right)$, with
$c_\nu>0$ such that $\left\|\frac{\nabla^r
\varphi_\nu}{g}\right\|_{L_p(\Omega)}=1$. Then $c_\nu
\stackrel{(\ref{ggg})}{\underset{\mathfrak{Z}}{\asymp}}
2^{\nu\beta_g}\Psi_g(2^{-\nu})2^{\nu\left(\frac dp -r\right)}$.
Hence,
$$
\|\varphi_\nu\| _{L_{q,v}(\Omega)} \stackrel
{(\ref{ggg})}{\underset{\mathfrak{Z}}{\asymp}} c_\nu\cdot
2^{\nu\beta_v} \Psi_v(2^{-\nu})\cdot 2^{-\frac{\nu d}{q}}
\underset{\mathfrak{Z}}{\asymp}
2^{\nu(\beta-\delta)}\Psi(2^{-\nu}).
$$
If $\beta-\delta>0$, then $\|\varphi_\nu\| _{L_{q,v}(\Omega)}
\underset{\nu\to\infty}{\to} \infty$. If $\beta=\delta$ and
$\alpha<0$, then $\|\varphi_\nu\| _{L_{q,v}(\Omega)}
\underset{\mathfrak{Z}}{\asymp} \nu^{-\alpha} \rho(\nu)
\underset{\nu\to\infty}{\to} \infty$.

Let $p>q$. First consider the case $\beta>\delta
-\theta\left(\frac 1q-\frac 1p\right)$. Set $\varphi_\nu(x)= c_\nu
\sum \limits_{j\in J_\nu} \varphi\left( \frac{x-z_j}{t_j}
\right)$, where $c_\nu>0$ is such that $\left\|\frac{\nabla^r
\varphi_\nu}{g}\right\|_{L_p(\Omega)}=1$. Then $c_\nu
\stackrel{(\ref{ggg})}{\underset{\mathfrak{Z}}{\asymp}}
2^{\nu\beta_g}\Psi_g(2^{-\nu})2^{\nu\left(\frac dp -r\right)}\cdot
({\rm card}\, J_\nu)^{-\frac 1p}$,
$$
\|\varphi_\nu\| _{L_{q,v}(\Omega)} \stackrel
{(\ref{ggg})}{\underset{\mathfrak{Z}}{\asymp}} c_\nu\cdot
2^{\nu\beta_v} \Psi_v(2^{-\nu})\cdot 2^{-\frac{\nu d}{q}} ({\rm
card}\, J_\nu)^{\frac 1q} \underset{\mathfrak{Z}}{\asymp}
2^{\nu(\beta-\delta)}\Psi(2^{-\nu}) ({\rm card}\, J_\nu)^{\frac
1q-\frac 1p} \stackrel{(\ref{cwj})}{\underset
{\mathfrak{Z}}{\gtrsim}}$$$$\gtrsim 2^{\nu\left(\beta-\delta+
\theta\left(\frac 1q-\frac 1p\right)\right)}\Psi(2^{-\nu})
\left(\Lambda(2^{-\nu})\right)^{\frac 1p-\frac 1q}.
$$
Therefore, $\|\varphi_\nu\| _{L_{q,v}(\Omega)}
\underset{\nu\to\infty}{\to} \infty$.

Let $\beta=\delta-\theta\left(\frac 1q-\frac 1p\right)$,
$\alpha<(1-\gamma)\left(\frac 1q-\frac 1p\right)$. For $s\in \N$
denote $${\cal N}_s=\{s+l(\hat k+1):\;\; l\in \Z_+, \;\; l(\hat
k+1)\le s\}.$$ Then ${\rm card}\, {\cal N}_s\underset{a,d}{\asymp}
s$. Set $\psi_s(x)=\sum \limits_{\nu\in {\cal N}_s} \sum \limits
_{j\in J_\nu} c_\nu \varphi\left(\frac{x-z_j}{t_j}\right)$, where
$c_\nu>0$ is such that $\left\|\frac{\nabla^r
\psi_s}{g}\right\|_{L_p(\Delta_j)} =({\rm card}\, J_\nu)^{-\frac
1p} ({\rm card}\, {\cal N}_s)^{-\frac 1p}$, $j\in J_\nu$, $\nu\in
{\cal N}_s$. Then $\left\|\frac{\nabla^r
\psi_s}{g}\right\|_{L_p(\Omega)}=1$ and
\begin{align}
\label{cnu} c_\nu
\stackrel{(\ref{ggg})}{\underset{\mathfrak{Z}}{\asymp}}
2^{\nu\beta_g} \Psi_g(2^{-\nu}) \cdot 2^{\nu\left(\frac dp
-r\right)}({\rm card}\, J_\nu)^{-\frac 1p} s^{-\frac 1p}
\end{align}
Hence,
$$
\|\psi_s\|_{L_{q,v}(\Omega)} \stackrel{(\ref{ggg})}
{\underset{\mathfrak{Z}}{\asymp}} \left(\sum \limits _{\nu \in
{\cal N}_s} 2^{\nu\beta_vq}\Psi_v^q(2^{-\nu})c_\nu^q\cdot 2^{-\nu
d}{\rm card}\, J_\nu\right)^{\frac 1q}
\stackrel{(\ref{phi_g}),(\ref{cwj}),(\ref{cnu})}
{\underset{\mathfrak{Z}}{\gtrsim}}
$$
$$
\asymp \left(\sum \limits _{\nu \in {\cal N}_s} 2^{\nu \beta
q}\Psi^q(2^{-\nu}) \cdot 2^{-\nu \delta q}\cdot 2^{\nu
\theta\left(1-\frac qp\right)} \nu^{-\gamma \left( 1-\frac
qp\right)}[\tau(\nu)]^{-1+\frac qp} s^{-\frac qp}\right)^{\frac
1q} \stackrel{(\ref{phi_g})}{\underset{\mathfrak{Z}}{\asymp}}$$$$
\asymp s^{-\alpha}\rho(s) \cdot s^{\left(\frac1q-\frac
1p\right)(1-\gamma)} [\tau(s)]^{-\frac 1q+\frac 1p}
\underset{s\to\infty}{\to} \infty.
$$

\end{document}